\documentclass[11pt]{amsart}
\UseRawInputEncoding
\usepackage{amssymb,mathrsfs,graphicx,enumerate,color}
\usepackage{amsmath, amssymb, bigints}
\usepackage{amsthm,amsfonts,amssymb,epsfig,graphics,amsmath,amsbsy,enumerate}
\usepackage[hidelinks]{hyperref}
\usepackage{mathtools}  % 在导言区添加
\usepackage{colortbl}
\definecolor{black}{rgb}{0.0, 0.0, 0.0}
\definecolor{red}{rgb}{1.0, 0.5, 0.5}
\usepackage{amsmath}
\usepackage{bigints}
\title[Stability of composite waves of viscous shock and rarefaction]
{Stability of Composite Waves of a non-classical Liu Shock and a Rarefaction for  Non-Convex Viscoelastic System}

\author[Huang]{Feimin Huang}
\address[Feimin Huang]{\newline State Key Laboratory of Mathematical Sciences and Institute of Applied Mathematics, AMSS, Chinese Academy of Sciences, Beijing 100190, P. R. China
\newline
and School of Mathematical Sciences, University of Chinese Academy of Sciences,
\newline Beijing 100049, P. R. China}
\email{fhuang@amt.ac.cn}

\author[Wang]{Yi Wang}
\address[Yi Wang]{\newline State Key Laboratory of Mathematical Sciences and Institute of Applied Mathematics, AMSS, Chinese Academy of Sciences, Beijing 100190, P. R. China
\newline
and School of Mathematical Sciences, University of Chinese Academy of Sciences,
\newline Beijing 100049, P. R. China}
\email{wangyi@amss.ac.cn}
\author[Zhang]{Jian Zhang}
\address[Jian Zhang]{\newline  School of Mathematical Sciences, Chengdu University of Technology, 
\newline Chengdu 610059, P. R. China}
\email{zhangjian2020@amss.ac.cn}

\newtheorem{theorem}{Theorem}[section]
\newtheorem{lemma}{Lemma}[section]

\newtheorem{proposition}{Proposition}[section]
\newtheorem{remark}{Remark}[section]

\newcommand{\bbr}{\mathbb {R}}

\newcommand{\mb}{\mathbf}

\newcommand{\wt}{\widetilde}

\newcommand{\deltas}{\delta_S}
\newcommand{\deltar}{\delta_R}

\numberwithin{figure}{section}
\newcommand{\beq}{\begin{equation}}
\newcommand{\eeq}{\end{equation}}
\newcommand{\bsp}{\begin{split}}
\newcommand{\esp}{\end{split}}

\newcommand{\di}{\displaystyle}

\def\eps{\varepsilon }

\newcommand\adots{\mathinner{\mkern2mu\raise1pt\hbox{.}
\mkern3mu\raise4pt\hbox{.}\mkern1mu\raise7pt\hbox{.}}}

\def\charf {\mbox{{\text 1}\kern-.30em {\text l}}}
 \newcommand{\bmat}{\begin{pmatrix}} 
  \newcommand{\emat}{\end{pmatrix}}

\begin{document}
%%%%%%%%%%%%%%%%
\bibliographystyle{plain}

\date{\today}

\subjclass[2010]{}

\begin{abstract}
We investigate the time-asymptotic stability of composite waves consisting of a non-classical Liu shock and a rarefaction wave for non-convex viscoelastic system.
By choosing delicate time-dependent shift functions to the viscous shock and suitable weight functions for the relative entropy, we succeed in proving the time-asymptotic stability of the composite of a viscous Liu shock and a rarefaction wave from two different characteristic families for the one-dimensional viscoelastic system with a cubic non-convex stress-strain relation by $a$-contraction method. Here the non-classical Liu shock can be either degenerate or non-degenerate, which can be seen as the system version of the scalar Oleinik shock. Note that both the shift and the weight functions are chosen quite differently in the two cases of nearly degenerate and fully non-degenerate. The paper represents the first time-asymptotic stability result for the composite waves consisting of both non-classical Liu shock and rarefaction wave from two different families for the non-convex physical system.

\end{abstract}

\maketitle \centerline{\date}

\tableofcontents

\section{Introduction}
\setcounter{equation}{0}

Consider the one-dimensional viscoelastic system
\begin{align}\label{VES}
\begin{aligned}
\left\{ \begin{array}{ll}
        v_t - u_x = 0, & (t,x)\in \mathbb R_+ \times \mathbb R,\\[4pt]
        u_t - \sigma(v)_x = \mu u_{xx},\qquad  & %\mu > 0,
        \end{array} \right.
\end{aligned}
\end{align}
where the unknown functions \(v = v(t,x)\) and \(u = u(t,x)\) represent the strain and the velocity of the viscoelastic  materials, respectively,  and the constant \(\mu > 0\) is the viscosity coefficient. The stress function $\sigma(v) $ is given by the following typical cubic non-convex constitutive relation
\begin{equation}\label{stress}
\sigma(v) = v - v^3 .
\end{equation}
System \eqref{VES}-\eqref{stress} serves as a canonical prototype for studying wave propagation in materials with a non-convex elastic energy. 
%The general framework of nonlinear viscoelasticity can be traced back to foundational works in continuum mechanics \cite{Antman2005}. 
The specific cubic stress-strain relation \eqref{stress} provides a simple smooth function that captures a non-convex stress-strain response, and it has been adopted as a standard mathematical model in \cite{ShearerYang1995}.%Truskinovsky1993}.
\vskip0.1cm
The system \eqref{VES}-\eqref{stress} is supplemented with the initial condition
\begin{equation}\label{in-1}
(v,u)(t=0,x) = (v_0(x),u_0(x)), \qquad x \in \mathbb R,
\end{equation}
which are assumed to converge to distinct far-field states at infinity:
\begin{equation}\label{in}
(v_0(x),u_0(x)) \longrightarrow (v_{\pm},u_{\pm}) \quad \text{as } x\to \pm\infty .
\end{equation}

The large-time behavior of the solutions to \eqref{VES}-\eqref{in} is closely related to the Riemann problem for the associated inviscid system
\begin{align}\label{E}
\begin{aligned}
\left\{ \begin{array}{ll}
        v_t - u_x = 0,\\[4pt]
        u_t - \sigma(v)_x = 0,
        \end{array} \right.
\end{aligned}
\end{align}
with Riemann initial data
\begin{align}\label{Ei}
(v,u)(t=0,x)=\left\{ \begin{array}{ll}
        (v_-,u_-), & x<0,\\[4pt]
        (v_+,u_+), & x>0,
        \end{array} \right.
\end{align}
where $ (v_{\pm},u_{\pm})$ are exactly determined by the far-field states of the initial values $(v_0(x),u_0(x))$ to the viscous system \eqref{VES}.

\vskip0.1cm
System \eqref{E} is a hyperbolic conservation law with a non-convex flux \eqref{stress}, which has been extensively studied for decades. %The system \eqref{E} is a hyperbolic conservation law with non-convex flux, and the investigation of such systems has a long-standing history in the field. 
In their pioneering work, Courant and Friedrichs \cite{CourantFriedrichs1948} provided a foundational analysis of elastic-plastic materials, demonstrating how non-convex stress-strain relations give rise to the intricate wave patterns. To date, the existence and admissibility criteria for these wave patterns, including the non-classical shock waves that arise in non-convex settings, have been systematically established in %classic monographs by  
Liu \cite{LT1,LT2,LT}, Smoller \cite{Smoller1994}, and Dafermos \cite{Dafermos2010} etc.

%System \eqref{E} is a hyperbolic conservation law with a non-convex flux. The study of such systems has a long history. In particular, Courant and Friedrichs \cite{CourantFriedrichs1948} provided a foundational discussion of elastic-plastic materials, illustrating how non-convex stress-strain relations lead to complex wave patterns. The existence and selection criteria for such wave patterns, including non-classical shock waves, are well established in the monographs by Liu \cite{LT1,LT2,LT}, Smoller \cite{Smoller1994}, and Dafermos \cite{Dafermos2010}.
\vskip0.1cm
\noindent
\textbf {Riemann Problem for the Inviscid Model \eqref{E}-\eqref{stress}. } The Jacobi matrix for the inviscid hyperbolic equation \eqref{E}-\eqref{stress} is
\begin{equation}\label{Jacobi}
\left(
\begin{array}{cr}
0 & -1\\
-\sigma'(v) & 0
\end{array}
\right),
\end{equation}
whose two eigenvalues are $\lambda_\pm(v) := \pm \lambda(v)$ with $\lambda(v) = \sqrt{\sigma'(v)}$, and therefore, the system  \eqref{E}-\eqref{stress} is
 strictly hyperbolic (i.e., $\pm\lambda(v)$ are real and distinct) provided that $|v| < \sqrt{3}/3$. 
We can compute that
$$
 \nabla \lambda_\pm(v) \cdot r_{\pm}(v) = \pm \lambda'(v) = \pm\frac{\sigma''(v)}{2\sqrt{\sigma'(v)}},
$$
where $r_{\pm}(v):=(1,\mp  \lambda(v))^T$ are the corresponding right eigenvectors to $\pm\lambda(v)$. Due to the non-convexity of the stress function $\sigma(v)$ around the inflection point $v=0$,
both characteristic fields are neither genuinely nonlinear nor linearly degenerate in any neighborhood of $v=0$. This non-convexity allows admissible shock waves that cross the inflection point, satisfying both the Rankine-Hugoniot jump condition and the Liu entropy criterion. 
\vskip0.1cm
Given a right state $(v_R,u_R)$, the Rankine-Hugoniot jump condition requires the existence of a shock speed $s$ for any left state $(v_L,u_L)$ satisfying
\begin{equation}\label{RH}
\begin{aligned}
&-s(v_R - v_L) - (u_R - u_L) = 0, \\
&-s(u_R - u_L) - (\sigma(v_R) - \sigma(v_L)) = 0.
\end{aligned}
\end{equation}
From \eqref{RH}, we obtain
\[
s = \pm\sqrt{\frac{\sigma(v_R) - \sigma(v_L)}{v_R - v_L}}, \quad \text{and} \quad
(u_R -u_L)^2 = \big[\sigma(v_R) - \sigma(v_L)\big](v_R - v_L).
\]
Hence, there are two branches of Hugoniot curve:
\[
H_i(v_R,u_R) := \left\{ (v,u) \,\middle|\, u = u_R +(-1)^i \sqrt{\frac{\sigma(v_R)-\sigma(v)}{v_R-v}}(v_R-v) \right\},
\]
corresponding to the shock speed $s_i = (-1)^i\sqrt{\frac{\sigma(v_R)-\sigma(v)}{v_R-v}}, i=1,2$.
% and
%\[
%H_2(v_R,u_R) := \left\{ (v,u) \,\middle|\, u = u_R + \sqrt{\bigl(\sigma(v_R)-\sigma(v)\bigr)(v_R-v)} \right\},
%\]
%corresponding to $s_i = \sqrt{\frac{\sigma(v_R)-\sigma(v)}{v_R-v}}$.
The $i$-shock curve $S_i(v_R,u_R)$ is defined as the subset of the Hugoniot curve $H_i(v_R,u_R)$ selected by the Liu entropy condition (see \cite{LT1,LT2,LT}), which will be invoked in the sequel.
%To determine the $i$-shock curve $S_i(v_R,u_R)$, i.e., the subset of the Hugoniot curve $H_i(v_R,u_R)$ selected by the Liu entropy condition, which  will be invoked  in the sequel as discussed in \cite{LT1, LT2, LT}.

\vskip0.1cm
\noindent\textbf{Liu Entropy Condition.} 
An $i$-discontinuity between $(v_L,u_L)$ and $(v_R,u_R)$, with $(v_L,u_L) \in H_i(v_R,u_R)$, is admissible if, for all states $(v,u)$ on the Hugoniot curve $H_i(v_R,u_R)$ lying between $(v_L,u_L)$ and $(v_R,u_R)$, the following criterion holds (cf.  \cite{LT1, LT2}):
\begin{equation}\label{Liu-entropy}
s\big((v,u), (v_R,u_R)\big) \leq s\big((v_L,u_L), (v_R,u_R)\big).
\end{equation}
Note that Liu entropy condition \eqref{Liu-entropy} reduces to the Oleinik entropy condition for scalar conservation laws \cite{Oleinik}, and is equivalent to the Lax entropy condition when the corresponding characteristic field is genuinely nonlinear for the strictly convex hyperbolic conservation law systems \cite{Lax}.
%\noindent\textbf{Liu Entropy Condition.} An $i$-discontinuity between $(v_L,u_L)$ and $(v_R,u_R)$, with $(v_R,u_R) \in H_i(v_L,u_L)$, is admissible if for all states $(v,u)$ on the Hugoniot curve $H_i(v_L,u_L)$ between $(v_L,u_L)$ and $(v_R,u_R)$, the following entropy condition holds:
%\begin{equation}\label{Liu-entropy}
%s\big((v_L,u_L), (v_R,u_R)\big) \leq s\big((v_L,u_L), (v,u)\big),
%\end{equation}
%where $s\big((v_L,u_L), (v,u)\big)$  is given by the Rankine-Hugoniot condition.

Without loss of generality, we consider the 2-shock on the Hugoniot curve $H_2(v_R,u_R)$ as an example, and 1-shock case can be defined similarly. For a given right state $(v_R,u_R)$ with $v_R > 0$, the Liu entropy condition  \eqref{Liu-entropy}  implies that the left state $(v_L,u_L)$ must satisfy $-\frac{v_R}{2} \leq v_L < v_R$, where the degenerate point $v_L = -\frac{v_R}{2}$ is uniquely determined by the following relation
\begin{equation}\label{ss}
\lambda_+(v_L) = s = \sqrt{\frac{\sigma(v_R) - \sigma(v_L)}{v_R - v_L}}.
\end{equation}
Hence, the 2-shock curve $S_2(v_R,u_R)$ with $v_R > 0$ is defined as
\begin{equation}\label{S2-curve}
S_2(v_R,u_R) := \Bigg\{ (v_L, u_L) \Bigg| -\frac{v_R}{2} \leq v_L < v_R, \; u_L = u_R + \sqrt{\left[\sigma(v_R) - \sigma(v_L)\right](v_R - v_L)} \Bigg\},
\end{equation}
while the 2-shock curve $S_2(v_R,u_R)$ with $v_R < 0$ and the 1-shock curve $S_1(v_R,u_R)$ with $v_R \lessgtr 0$ can be defined similarly. The shock considered in the present paper is defined via the Liu entropy condition  \eqref{Liu-entropy}, and we call it a Liu shock, which can cross the inflection point $v=0$, thereby being a non-classical shock in contrast to the classical Lax shock \cite{Lax}.

%which can cross the inflection point  $v=0$, thereby is a non-classical shock, compared with the classical Lax shock  \cite{Lax}.

Whenever $(v_L, u_L) \in S_1(v_R, u_R) \cup S_2(v_R, u_R)$, the solution $(v^s, u^s)$ to the Riemann problem \eqref{E}-\eqref{Ei} with the end states $(v_-, u_-) = (v_L, u_L)$ and $(v_+, u_+) = (v_R, u_R)$ is a Liu shock, given by
\begin{equation} \label{BERintro}
(v^s, u^s)(t, x) = 
\begin{cases}
(v_L, u_L), & x < s t, \\
(v_R, u_R), & x > s t,
\end{cases}
\end{equation}
where $s$ denotes the shock speed given by  Rankine-Hugoniot condition and Liu entropy condition.
\vskip0.1cm
Next, we turn to another elementary wave, namely the rarefaction wave. Within a given characteristic field, we consider two types of wave patterns involving rarefaction waves. One is a single rarefaction wave, and the other is a composite wave consisting of a rarefaction wave connected to a degenerate Liu shock.
%Two types of wave patterns involving a rarefaction wave in the same characteristic field will be considered.
%Two types of wave patterns involving rarefaction waves will be considered.
%Next we turn to another elementary wave pattern, namely the rarefaction wave. There are two kinds of rarefaction waves.
%One is a single rarefaction wave associated with a characteristic field, and the other is the composite wave consisting of a rarefaction wave connected to a degenerate Liu shock in the same characteristic field. 
\vskip0.1cm
We first describe the single rarefaction wave case. Given a right state $(v_R, u_R) $, the 1-rarefaction curve $R_1(v_R, u_R)$, which is the integral curve associated with the first eigenvalue $\lambda_-(v)$ of the Jacobi matrix \eqref{Jacobi}, can be defined by
\begin{equation}\label{eq:rarefaction1}
R_1 (v_R, u_R) := \Bigg\{ (v, u) \Bigg| 0 < v < v_R \ \text{or} \ v_R < v < 0, \ u = u_R - \int_{v_R}^{v} \lambda_-(s) \, ds \Bigg\}.
\end{equation}
The 2-rarefaction curve $R_2(v_R, u_R)$ can also be defined analogously as the integral curve associated with the second eigenvalue $\lambda_+(v)$. 
%Note that the rarefaction waves can only form on the convex (or concave) part of the flux $\sigma$.
%  \vskip0.1cm
% For any Riemann initial data \eqref{Ei} with $(v_-,u_-)=(v_L,u_L)$ and $(v_+,u_+)=(v_R,u_R)$ such that $(v_L,u_L) \in R_1(v_R,u_R)$, the solution $(v^r,u^r)(t,x)$ to the Riemann problem \eqref{E}-\eqref{Ei} is a 1-rarefaction wave, which can be expressed implicitly by
% \begin{equation} \label{BESintro}
% \lambda_-(v^r(t, x)) = \begin{cases}
% \lambda_-(v_L), & x < \lambda_-(v_L)t, \\[2mm]
% \dfrac{x}{t}, & \lambda_-(v_L)t \leq x \leq \lambda_-(v_R)t, \\[2mm]
% \lambda_-(v_R), & x > \lambda_-(v_R)t.
% \end{cases}
% \end{equation}
% The velocity component $u^r(t,x)$ is then determined by the relation
% \[
% z_-(v^r(t,x),u^r(t,x))
% =z_-(v_L,u_L)=z_-(v_R,u_R),
% \]
% where
% \[
% z_-(v,u):=u+\int^v\lambda_-(s)\,ds
% \]
% is the 1-Riemann invariant.
% The case of a 2-rarefaction wave can be treated similarly.

\vskip0.1cm

We next describe the composite wave case consisting of a rarefaction wave
connected to a degenerate Liu shock in the same characteristic family. For definiteness, we consider the composite wave in the second characteristic field, and the first characteristic field can be treated similarly.  Given a right state $(v_R, u_R)$ with $v_R>0$ (the case $v_R<0$ can be treated similarly), when $v_L < -\frac{v_R}{2}<0<v_R$, there exists a unique intermediate state $(v_M, u_M)$ with $v_M = -\frac{v_R}{2}$ such that $(v_M, u_M)$ is connected to $(v_R, u_R)$ by a degenerate 2-Liu shock and $(v_L, u_L)$ is connected to $(v_M, u_M)$ by a 2-rarefaction wave. The degenerate 2-Liu shock curve is given by \eqref{S2-curve}, while the 2-rarefaction curve is defined analogously to \eqref{eq:rarefaction1}.

\vskip0.1cm
In general, the solution of the Riemann problem \eqref{E}-\eqref{Ei} consists of at most two family waves, each of which can be a rarefaction wave, a Liu shock, or a composite wave consisting of a degenerate Liu shock and a rarefaction. The generic Riemann solutions to \eqref{E}-\eqref{Ei} with \eqref{stress} are much more complicated than the strictly convex flux case, and can consist of at most four waves, namely, a degenerate 1-Liu shock, a 1-rarefaction wave, a degenerate 2-Liu shock and a 2-rarefaction wave.
% In this paper, we consider the time-asymptotic stability of the composite wave of a viscous Liu shock (including both non-degenerate and degenerate cases) and a rarefaction wave from two different characteristic families to the non-convex viscoelastic system \eqref{VES}-\eqref{stress}. 
% \vskip0.1cm
%  For the corresponding Riemann solution, we consider the case for the two states $(v_-, u_-)$ and $(v_+, u_+) \in \mathbb{R}^2$ with $v_+>0$ and for which there exists a unique intermediate state $(v_m, u_m) \in \mathbb{R}^2$ such that $(v_-, u_-) \in R_1(v_m, u_m)$ and $(v_m, u_m) \in S_2(v_+, u_+)$ with $ -\frac{v_+}{2}\leq v_m <0$.  Then the Riemann solution $(v, u)$ to \eqref{E}-\eqref{Ei} is given by the composite wave
% $$
% (v,u)(t,x)=(v^r,u^r)(t,x)+(v^s,u^s)(t,x)-(v_m,u_m),
% $$
% where $(v^r,u^r)$ is the 1-rarefaction wave connecting $(v_-,u_-)$ to $(v_m,u_m)$ defined as in \eqref{BESintro}-\eqref{rareintro}, and $(v^s,u^s)$ is the 2-Liu shock connecting $(v_m,u_m)$ to $(v_+,u_+)$ defined as in \eqref{BERintro}, which is non-degenerate when  $ -\frac{v_+}{2}< v_m <0$ and degenerate when  $v_m=-\frac{v_+}{2}$.

\vskip0.1cm
\noindent
\textbf{Previous Time-Asymptotic Stability Results for the Viscous Model \eqref{VES}.}

Roughly speaking, the time-asymptotic behavior of the solution to the viscous system \eqref{VES}-\eqref{stress} with the initial data \eqref{in-1}-\eqref{in} is determined by the Riemann solution for the corresponding inviscid model \eqref{E}-\eqref{Ei}. When the Riemann solution to \eqref{E}-\eqref{Ei} is a Liu shock, the large-time ansatz for the viscous system \eqref{VES} is a viscous  Liu shock profile, described by the traveling wave solution $(\wt{v}^S(x - s t), \wt{u}^S(x - s t))$ to \eqref{VES}.
%\begin{equation}\label{VSintro}
%\left\{
%\begin{array}{ll}
%\displaystyle -s (\wt{v}^S)' - (\wt{u}^S)' = 0, \\[3mm]
%\displaystyle -s (\wt{u}^S)' - (\sigma(\wt{v}^S))' = \mu (\wt{u}^S)'', \\[3mm]
%\displaystyle (\wt{v}^S, \wt{u}^S)(-\infty) = (v_-, u_-), \quad (\wt{v}^S, \wt{u}^S)(+\infty) = (v_+, u_+).
%\end{array}
%\right.
%\end{equation}
When the Riemann solution to \eqref{E}-\eqref{Ei} is a rarefaction wave, the large-time ansatz for the corresponding viscous system \eqref{VES} can still be the self-similar rarefaction wave fan of the inviscid system \eqref{E}.
\vskip0.1cm
For the viscous system with strictly convex flux, there are huge literatures on the time-asymptotic stability of viscous shock, rarefaction wave and even viscous contact discontinuity since the pioneer work of Il'in and Oleinik \cite{Il'in-Oleinik} on Burgers equation, and much progress has been made in this field by the direct or weighted energy methods, spectral methods, point-wise Green function methods, $L^1$-stability methods and even the combined methods mentioned above, see \cite{FS, G, HKK, HXY, HM, HLM, HL,  Liu, LX, LZ, MN-S, MN-SS, MN86, matwan, MZ, NYZ, SX, VY} and references therein. Very recently, Kang, Vasseur, and Wang successfully proved the time-asymptotic stability of the generic Riemann solutions (containing viscous shock wave, rarefaction wave, and even viscous contact discontinuity) to both barotropic compressible Navier-Stokes equations  \cite{M-V-W} and full compressible Navier-Stokes(-Fourier) equations \cite{KVW3} by using the
$a$-contraction method with a time-dependent shift for viscous shock waves
\cite{KV}.

\vskip0.1cm
The study of non-convex conservation laws dates back to the scalar case, where the stability of single non-classical shocks (i.e. Oleinik shocks) has been well understood \cite{H, H1, JGK, MK,  LHL, MN3, ME, OR, We}, and the stability of composite waves consisting of an Oleinik shock and a rarefaction wave has been proved by Huang, Wang, and Zhang \cite{H-W-Z} for the cubic non-convex Burgers equation very recently. For the non-convex viscoelasticity system \eqref{VES}, Kawashima and Matsumura \cite{MK} established the first result on the time-asymptotic stability of non-degenerate viscous shocks for system \eqref{VES}, under the non-convex constitutive assumption that $\sigma^\prime(v)>0$ for all $v\in \mathbb{R}$ and $\sigma''(v) \gtrless 0$ for $v \gtrless 0$. This condition allows the stress function $\sigma(v)$ to have exactly one inflection point at $v=0$, with the prototypical example given by $\sigma(v)=v+v^3$. Later, Nishihara \cite{KN}  extended this result to degenerate viscous shocks for system \eqref{VES} under slightly weaker assumptions than those required in \cite{MK}.
It should be noted that the st 
ress function \eqref{stress} studied in the present paper satisfies the opposite sign condition: $\sigma''(v) \lessgtr 0$ for $v \gtrless 0$, which falls outside the scope of the conditions assumed in the aforementioned studies. Moreover, Kawashima and Matsumura \cite{MK}  explicitly posed the stability of non-classical viscous shocks under this opposite sign condition as an open problem, which was resolved by Matsumura and Mei \cite{MM} for the case of non-degenerate viscous shocks. After that, Mei and Nishihara  \cite{MN}  further improved the stability results obtained in previous works, proving stability for both non-degenerate and degenerate shocks under weaker assumptions on the nonlinear stress function, initial disturbance, and weight function. Remark that all the above results for the stability of single viscous shock for system \eqref{VES} are based on the anti-derivative techniques and therefore the zero-mass conditions are essentially needed for the initial perturbations around the shock profile.
\vskip0.1cm
On the other hand, the rarefaction wave can only form on the convex (or concave) part of the stress function $\sigma(v)$, and its time-asymptotic stability for \eqref{VES} can be proved almost the same as for the classical compressible Navier-Stokes equations.
 Matsumura and Nishihara \cite{MN86} first proved the stability of rarefaction waves for isentropic compressible Navier-Stokes equations, later the result in \cite{MN86} was extended to the case that both perturbations and wave strength can be large in \cite{MN2}. Note that the stability of rarefaction wave can be proved by the direct energy techniques in the framework of original perturbations, which is quite different from the anti-derivative framework used for the stability of viscous shock profiles. Due to these incompatibilities, the time-asymptotic stability of the composite wave including viscous shock and rarefaction wave for the non-convex viscoelastic system \eqref{VES}-\eqref{stress} is still open up to now. 

Motivated by the work of Kang, Vasseur, and Wang \cite{M-V-W, KVW3} on the time-asymptotic stability of generic Riemann solutions, including viscous shock waves, rarefaction waves, and viscous contact discontinuities for the barotropic and full compressible Navier--Stokes equations via the $a$-contraction method, we investigate the time-asymptotic stability of a composite wave consisting of a viscous Liu shock and a rarefaction wave belonging to two distinct characteristic families for the viscous viscoelastic system \eqref{VES}--\eqref{stress}.

% the stability of single shocks has been extensively studied. Using the anti-derivative method with a zero-mass condition, Matsumura and Nishihara \cite{MN-S} proved viscous shock stability for compressible Navier-Stokes equations. Goodman \cite{G} obtained similar results for a general system with artificial diffusion. To remove the zero-mass condition, Liu \cite{Liu}, Szepessy and Xin \cite{SX}, and Liu and Zeng \cite{LZ} introduced constant shifts, diffusion waves, and coupled diffusion waves in transverse characteristic fields. Mascia and Zumbrun \cite{MZ} proved spectral stability for 1D compressible Navier-Stokes under a condition slightly weaker than zero-mass condition. For Navier-Stokes systems with degenerate viscosities, Matsumura and Wang \cite{matwan} studied asymptotic shock stability, later generalized to a broader viscosity class in \cite{VY} using the Bresch-Desjardins entropy \cite{BD-06}.
\vskip0.1cm

%Regarding superposition waves, Kang, Vasseur, and Wang \cite{M-V-W} proved time-asymptotic stability for composites of viscous shock and rarefaction in barotropic Navier-Stokes equations via the $a$-contraction theory, later extending it to generic Riemann solutions for compressible Navier-Stokes-Fourier systems \cite{KVW3}. For the Navier-Stokes-Fourier system, Huang and Matsumura \cite{HM} proved stability for superpositions of two shocks, and Huang and Lee \cite{HL} proved stability of composite waves of viscous shocks and a viscous contact wave. 
%However, there are no results yet on the stability of composite waves involving a Liu shock and a rarefaction wave. This paper addresses the long-time asymptotic stability of composite waves consisting of a 1-rarefaction wave and a 2-Liu shock.
\vskip0.1cm
\noindent
\textbf{Composite Waves of 1-Rarefaction and 2-Liu shock.}

We focus on composite waves consisting of a 1-rarefaction and a 2-Liu shock. Without loss of generality, we assume $(v_+, u_+) \in \mathbb{R}_+ \times \mathbb{R}$, $(v_m, u_m) \in S_2(v_+, u_+)$, and $(v_-, u_-) \in R_1(v_m, u_m)$. Several new difficulties arise in the present setting. First, owing to the non-convexity of the stress function $\sigma(v)$, the shift and weight functions introduced in \cite{M-V-W} are no longer applicable. Second, the presence of an inflection point in $\sigma(v)$ requires substantially different choices of both the shift and weight functions in the nearly degenerate and fully non-degenerate shock regimes. Finally, the possible degeneracy of the viscous Liu shock poses an additional challenge. 

 From the definitions of the 1-rarefaction curve \eqref{eq:rarefaction1} and the 2-shock curve \eqref{S2-curve}, it follows that
\begin{equation}\label{s}
-\frac{v_+}{2}\leq v_m< v_- < 0 < v_+,
\end{equation}
or 
\begin{equation}\label{sss}
0<v_-<v_m<v_+.
\end{equation}
The present paper aims to prove the asymptotic stability of the composite waves described by \eqref{s}, where the 2-Liu shock crosses the inflection point $v=0$. While the case \eqref{sss}  excludes the inflection point $v=0$ and the corresponding composite wave is a 1-rarefaction and a 2-classical shock, whose time-asymptotic stability had been proved by Kang-Vasseur-Wang in \cite{M-V-W}.
\vskip0.1cm
We initially consider a superposition wave of the form
\begin{equation}\label{SW}
\Bigl( v^r\Bigl(\frac{x}{t}\Bigr) + \wt{v}^S(x - s t) - v_m,\; u^r\Bigl(\frac{x}{t}\Bigr) + \wt{u}^S(x - s t) - u_m \Bigr),
\end{equation}
where $(v^r, u^r)(x/t)$ is the 1-rarefaction wave defined in \eqref{rare} and $(\wt{v}^S, \wt{u}^S)(\xi)$ is the 2-viscous Liu shock wave defined in \eqref{V-S}. However, the classical anti-derivative method for shock stability does not align well with the direct $L^2$ energy method typically used for rarefaction stability. To overcome this, we adopt the $a$-contraction method (see \cite{H-W-Z, KVARMA, KV, VKM,  MKA, MJ, KVW2, KVW1, M-V-W, KVW3, Koh,  LV, VWZ, WW} and references therein) with a time-dependent shift $\mb X(t)$ for viscous shock and consider the shifted superposition wave
\begin{equation}\label{ansatz}
\Bigl( v^r\Bigl(\frac{x}{t}\Bigr) + \wt{v}^S\bigl(x - s t + \mb X(t)\bigr) - v_m,\; u^r\Bigl(\frac{x}{t}\Bigr) + \wt{u}^S\bigl(x - s t + \mb X(t)\bigr) - u_m \Bigr),
\end{equation}
 where the time-dependent shift function $\mb X(t)$ is to be determined. In this framework, the time-dependent shift $\mb X(t)$ and the weight function $a$ must be carefully chosen. In particular, the weight function $a$ used for the classical shocks with convex fluxes in \cite{M-V-W} can not be applicable here, since the Liu shock under consideration crosses the inflection point $v=0$ and can be degenerate when $v_m=-\frac{v_+}{2}$. 
Partially motivated by the time-asymptotic stability of the composite wave consisting of a viscous Oleinik shock and a rarefaction wave for the scalar cubic equation as in \cite{H-W-Z}, the weighted Poincar\'e inequality is utilized to the weighted perturbation $\Psi(t,y):=a^{\mb{X}}(\wt v^S(\xi))\psi(t,\xi)$ in two different cases. In the nearly degenerate case, it is used on a suitably chosen subinterval $\xi\in [\xi_*, +\infty)$  as in \eqref{quanity}, that is, $\wt v^S \in [v_*, v_+)$, which crosses the inflection point $v=0$. In the fully non-degenerate case, it is used on the whole line $\xi\in (-\infty, +\infty)$ as in \eqref{wp}, that is, $\wt v^S \in (v_-, v_+)$. Note that both the shift functions $\mb X(t)$ and the weight functions $a$ are chosen  quite differently for these two cases, which are much more complicated than those for the scalar cubic equation %with only degenerate Oleinik shock 
in \cite{H-W-Z}.
It should be also noticed that the constructions of shift functions and weight functions here are dedicately designed for the case $\sigma(v)=v-v^3$ satisfying the sign condition $\sigma''(v)\lessgtr 0$ for $v\gtrless 0$, and can not be applicable to the opposite case $\sigma(v)=v+v^3$ as in \cite{MK}, whose time-asymptotic stability of composite wave including both Liu shock and rarefaction wave remains an interesting open problem for the future study.

\vskip0.1cm
\noindent
%\textbf{Main difficulty and key ingredients‌}
% Motivated by Kang, Vasseur, and Wang \cite{M-V-W, KVW3} for the  time-asymptotic stability of the generic Riemann solutions including viscous shock wave, rarefaction wave, and even viscous contact discontinuity to barotropic/full compressible Navier-Stokes equations by using $a$-contraction method,
%we consider the time-asymptotic stability of the composite wave consisting of a viscous Liu shock and a rarefaction wave from two different characteristic families to the viscous viscoelastic system \eqref{VES}-\eqref{stress} in the present paper. Several new difficulties will occur. First, the non-convexity of the stress function $\sigma(v)$ makes the shift function and weight function used in \cite{M-V-W} can not be used now. Second, the inflection point of the stress function $\sigma(v)$ makes that both the shift and the weight functions are chosen quite differently in the two cases of nearly degenerate and fully non-degenerate shocks. Finally, the possible degeneracy of the viscous Liu shock will be another difficulty. We also remark that our shift functions and weighted functions is only valid to the non-convex stress function $\sigma(v)=v-v^3$, and can not be applied to the opposite case $\sigma(v)=v+v^3$.
%\textbf{Main difficulties and key ingredients.} 

\vskip0.3cm
\subsubsection*{Notations:}
\begin{itemize}
    \item $C$ denotes a generic positive constant, independent of time $T$ and the shift $\mb X(t)$, whose value may change from line to line.
    \item $L^p(\Omega)$ $(1 \le p \le \infty)$ and $H^1(\omega)$ denote the usual Lebesgue and Sobolev spaces on $\Omega\subset \mathbb{R}$, with norms
          \[
          \|f\|_{L^p} := \left( \int_{\Omega} |f|^p \, d\xi \right)^{1/p}, \quad
          \|f\| := \|f\|_{L^2}, \quad
          \|f\|_{H^1} := \left( \|f\|^2 + \|f_\xi\|^2 \right)^{1/2}.
          \]
    \item For a function $f=f(t,x): \mathbb{R}^+ \times \mathbb{R} \to \mathbb{R}$ and a time-dependent shift $\mb X(t)$,
          \[
          f^{\pm \mb X}(t, \xi) := f(t, \xi \pm \mb X(t)).
          \]
          In particular, for the weight function $a(\widetilde{v}^S(\xi))$ defined in \eqref{w.deg} or \eqref{w.nondeg},
          \[
          a^{\mb X}(\xi) := a(\widetilde{v}^S(\xi + \mb X(t))).
          \]
   %  \item We always denote $ \mathcal{E}(T):= \displaystyle \sup_{0\leq t \leq T} \|(\phi,\psi)(t)\|_{H^1(\bbr)}$.
\end{itemize}

\vskip0.3cm

The remainder of the paper is organized as follows. Section~\ref{sec-preli} is concerned with the preliminaries for some properties of the self-similar rarefaction wave and the viscous Liu shock. In addition, this section also states the main result of this paper. Section~\ref{sec-thm} is devoted to the proof of 
main result Theorem~\ref{thm:main} including the construction of the approximate rarefaction wave, the reformulation of the problem, the constructions of both shift functions and weight functions and the proof of Theorem~\ref{thm:main} via a continuity argument based on the local existence of solutions and uniform-in-time a priori estimates. Then Section~\ref{sec-acontraction} is focused on the uniform-in-time a priori estimates in Proposition~\ref{prop:apriori} via the $a$-contraction method. Finally, the proof of the identity \eqref{4.33} is shown in the Appendix.

\section{Preliminaries and Main Result} \label{sec-preli}
\setcounter{equation}{0}
This section collects some properties of the self-similar rarefaction wave and the viscous Liu shock, and then we state main result.

\subsection{The 1-Rarefaction Wave}
We begin with the self-similar 1-rarefaction wave for the state $(v_-, u_-) \in R_1(v_m, u_m)$ as defined in \eqref{eq:rarefaction1}. Let $w_- := \lambda_-(v_-)$ and $w_m := \lambda_-(v_m)$, and consider the Riemann problem for the inviscid Burgers equation:
\begin{equation} \label{BE}
\begin{cases}
\displaystyle w_t + w w_x = 0, \\[4mm]
\displaystyle w(0, x) = w_0^r(x) := \begin{cases}
w_-, & x < 0,\\
w_m, & x > 0.
\end{cases}
\end{cases}
\end{equation}
If $w_- < w_m$, that is, $0<v_-<v_m$\ \ or\ \ $v_m < v_- < 0$, then the unique entropy solution to \eqref{BE} is the self-similar rarefaction fan $w^r(t, x) = w^r(x/t)$, given explicitly by
\begin{equation} \label{BES}
w^r(t, x) = w^r\Bigl(\frac{x}{t}\Bigr) := \begin{cases}
w_-, & x < w_- t, \\[1mm]
\displaystyle \frac{x}{t}, & w_- t \leq x \leq w_m t, \\[2mm]
w_m, & x > w_m t.
\end{cases}
\end{equation}

The corresponding 1-rarefaction wave $(v^r, u^r)(t, x) = (v^r, u^r)(x/t)$  to the original Riemann problem \eqref{E}-\eqref{Ei} with $(v_+, u_+)$ replaced by $(v_m, u_m)$  can be constructed via the relation
\begin{equation}\label{rare}
\begin{aligned}
&\lambda_-(v_-) = w_-,\quad \lambda_-(v_m) = w_m, \\
&\lambda_-\bigl(v^r(\tfrac{x}{t})\bigr) = w^r(\tfrac{x}{t}), \\
&z_-\bigl(v^r(\tfrac{x}{t}), u^r(\tfrac{x}{t})\bigr) = z_-(v_-, u_-) = z_-(v_m, u_m),
\end{aligned}
\end{equation}
where the 1-Riemann invariant $\di z_-(v,u):=u+\int^v \lambda_-(s)\,ds$ is constant along 1-rarefaction wave.
% $z_-(v^r(t,x), u^r(t,x)) = z_-(v_-, u_-) = z_-(v_m,u_m).$ is the 1-Riemann invariant.
The self-similar rarefaction wave $(v^r, u^r)(t,x)$ is Lipschitz continuous for $t > 0$ and satisfies the inviscid system \eqref{E} almost everywhere in $(t,x)\in\mathbb{R}_+\times\mathbb{R}$.
%\begin{align}\label{E-ra}
%\begin{cases}
%v^r_t - u^r_x = 0, \\[1mm]
%u^r_t - \sigma(v^r)_x = 0.
%\end{cases}
%\end{align}
%The strength of the rarefaction wave is denoted by $\delta_R := |v_m - v_-|$. Note that $\delta_R$ is of the same order as the jump in velocity, i.e., $\delta_R \sim |u_m - u_-|$, which follows from the second relation in \eqref{rare}.
\subsection{The 2-Viscous Liu Shock Wave}
\label{subsec:viscous_shock}

We now consider the 2-viscous Liu shock wave $(\wt{v}^S, \wt{u}^S)$ connecting the states $(v_m, u_m)$ and $(v_+, u_+)$ with $(v_m, u_m) \in S_2(v_+, u_+)$ as defined in \eqref{S2-curve}. 
%The states satisfy the Rankine-Hugoniot condition \eqref{RH} and the Liu entropy condition \eqref{Liu-entropy}. Consequently, 
The Riemann problem \eqref{E}-\eqref{Ei} with $(v_-, u_-)$ replaced by $(v_m, u_m)$  admits an entropic 2-shock solution:
\begin{equation} \label{Shock}
(v^s, u^s)(t, x) = 
\begin{cases}
(v_m, u_m), & x < s t, \\
(v_+, u_+), & x > s t,
\end{cases}
\end{equation}
where the shock speed $s$ is determined by the Rankine-Hugoniot condition
\begin{equation} \label{speed}
s=\sqrt{\frac{\sigma(v_m) - \sigma(v_+)}{v_m - v_+}}=\sqrt{1-(v_m^2+v_mv_++v_+^2)},
\end{equation}
and the Liu entropy condition \eqref{Liu-entropy}.

For the viscous viscoelastic system \eqref{VES}, the corresponding large-time ansatz is governed by the viscous shock profile $(\wt{v}^S, \wt{u}^S)(\xi) ~(\xi = x - st)$ satisfying the ODE system:
\begin{equation} \label{V-S}
\left\{
\begin{aligned}
&-s (\wt{v}^S)' - (\wt{u}^S)' = 0, \quad ' = \frac{d}{d\xi}, \\
&-s (\wt{u}^S)' - (\sigma(\wt{v}^S))' = \mu (\wt{u}^S)'', \\
&(\wt{v}^S, \wt{u}^S)(-\infty) = (v_m, u_m), \quad (\wt{v}^S, \wt{u}^S)(+\infty) = (v_+, u_+).
\end{aligned}
\right.
\end{equation}
Integrating \eqref{V-S} over $[\xi, \pm \infty)$, we have the identity $-s(\wt{v}^S-v_+)=\wt{u}^S-u_+,$ or $-s(\wt{v}^S-v_m)=\wt{u}^S-u_m,$  and the profile $\wt v^S(\xi)$ can be reduced to the first-order scalar ODE:
\begin{equation} \label{VS_scalar}
\left\{
\begin{aligned}
&\mu s (\wt{v}^S)' = (\wt{v}^S - v_m)(v_+ - \wt{v}^S)(\wt{v}^S + v_+ + v_m), \\
&\wt{v}^S(-\infty) = v_m, \quad \wt{v}^S(+\infty) = v_+.
\end{aligned}
\right.
\end{equation}
%It follows from \eqref{VS_scalar} that the viscous shock profiles exhibit different decay rates in the degenerate case $v_m = -v_+/2$ and the non-degenerate case $v_m \in (-v_+/2, 0)$.
%From \eqref{VS_scalar}, one can see that the viscous shock profiles exhibit different decay rates in the degenerate case $v_m = -v_+/2$ and the non-degenerate case $v_m \in (-v_+/2, 0)$. Let $\delta_S := |v_+ - v_m| \sim |u_+ - u_m|$ denote the strength of the shock. The following lemma summarizes the key properties of the 2-viscous Liu shock profile (see \cite{H-W-Z, MN3, ME}). 

The properties of the viscous Liu shock profile $(\wt{v}^S, \wt{u}^S)(\xi) ~(\xi = x - st)$ can be summarized as follows. 

\begin{lemma} \label{lemma:shock_properties}
For $(v_m, u_m) \in S_2(v_+, u_+)$, there exists a unique (up to the constant translation) solution $(\wt{v}^S, \wt{u}^S)(\xi)$ to ODE system \eqref{V-S}. Moreover, there exists a constant $C > 0$, independent of the shock wave strength $\delta_S:=|v_+-v_m|\sim | u_+-u_m|\sim v_+$, such that the following estimates hold
%\footnote{The case $v_m = -v_+/2$ corresponds to the degenerate Liu shock, while the case $-\frac{v_+}{2} < v_m < v_+$ corresponds to the non-degenerate Liu shock.}:
\begin{align*}
    &\wt{u}^S_\xi < 0, \quad \wt{v}^S_\xi > 0, \\[2mm]
    & |\wt{v}^S(\xi) - v_m| \leq 
    \begin{cases}
        C\delta_S e^{-C\delta_S |\xi|}, & \text{if} \ \ -\dfrac{v_+}{2} < v_m < 0, \\[4mm]
        \dfrac{C\delta_S}{1 + C\delta_S^2 |\xi|}, &  \text{if } \ \  v_m = -\dfrac{v_+}{2},
    \end{cases} \quad \forall \xi < 0, \\[2mm]
    & |\wt{v}^S(\xi) - v_+| \leq C\delta_S e^{-C\delta_S |\xi|}, \quad \forall \xi > 0, \\[2mm]
    & |(\wt{v}^S_\xi, \wt{u}^S_\xi)| \leq 
    \begin{cases}
        C\delta_S^2 e^{-C\delta_S |\xi|}, &  \text{if } \ \  -\dfrac{v_+}{2} < v_m < 0, \\[4mm]
        \dfrac{C\delta_S^3}{(1 + C\delta_S^2 |\xi|)^2}, & \text{if } \ \  v_m = -\dfrac{v_+}{2},
    \end{cases} \quad \forall \xi < 0, \\[2mm]
    & |(\wt{v}^S_\xi, \wt{u}^S_\xi)| \leq C\delta_S^2 e^{-C\delta_S |\xi|}, \quad \forall \xi > 0, \\[2mm]
    & |(\wt{v}^S_{\xi\xi}, \wt{u}^S_{\xi\xi})| \leq C\delta_S^2 |(\wt{v}^S_\xi, \wt{u}^S_\xi)|, \quad \forall \xi \in \mathbb{R}.
\end{align*}
\end{lemma}
Roughly speaking, the profile  $(\wt{v}^S, \wt{u}^S)(\xi)$ always exhibits the exponentially decay at the far-field state  $(v_+, u_+)$, and has the  exponentially decay for the  non-degenerate Liu shock when $v_m \in (-v_+/2, 0)$ and the algebraically decay  for the degenerate Liu shock when $v_m=-v_+/2$ at the far-field state  $(v_m, u_m)$.

Now we can state the main result in the present paper as follows.

\begin{theorem}\label{thm:main}
For a given far-field state $(v_+, u_+) \in \mathbb{R}_+ \times \mathbb{R}$, there exists a constant $\delta_* > 0$ such that if $|v_+ |< \delta_*$, then the following holds true. 

For any $(v_m, u_m) \in S_2(v_+, u_+)$ and $(v_-, u_-) \in R_1(v_m, u_m)$ satisfying \eqref{s}, there exist constants $\delta^* > 0$ and $\varepsilon^* > 0$ such that if the rarefaction wave strength $\delta_R := |v_m - v_-| < \delta^*$ and the initial data $(v_0, u_0)$ satisfy
\begin{align}\label{i_p}
\begin{aligned}
& \big\|(v_0 - \wt{v}^S(\cdot+x_0), u_0 - \wt{u}^S(\cdot+x_0))\big\|_{L^2(\mathbb{R}_+)} \\
&+ \big\|(v_0 - (\wt{v}^S(\cdot+x_0) + v_- - v_m), u_0 - (\wt{u}^S(\cdot+x_0) + u_- - u_m))\big\|_{L^2(\mathbb{R}_-)} \\
& + \big\|(v_{0x} - \wt{v}^S_x(\cdot+x_0), u_{0x} - \wt{u}^S_x(\cdot+x_0))\big\|_{L^2(\mathbb{R})} < \varepsilon^*,
\end{aligned}
\end{align}
 where $x_0\in \mathbb{R}$ is the arbitrary initial shock location and $\mathbb{R}_+=-\mathbb{R}_- := (0, +\infty)$, then the Cauchy problem \eqref{VES}-\eqref{in} admits a unique global-in-time solution $(v, u)$. Moreover, there exists an absolutely continuous shift function $\mb X(t)$ with $\mb X(0)=x_0$ such that
\begin{align}\label{ext-main}
\begin{aligned}
& v(t, x) - \Bigl( v^r\Bigl(\frac{x}{t}\Bigr) + \wt{v}^S\bigl(x - s t + \mb X(t)\bigr) - v_m \Bigr) \in C\bigl(0, +\infty; H^1(\mathbb{R})\bigr), \\
& u(t, x) - \Bigl( u^r\Bigl(\frac{x}{t}\Bigr) + \wt{u}^S\bigl(x - s t + \mb X(t)\bigr) - u_m \Bigr) \in C\bigl(0, +\infty; H^1(\mathbb{R})\bigr), \\
& u_{xx}(t, x) - \wt{u}^S_{xx}\bigl(x - s t + \mb X(t)\bigr) \in L^2\bigl(0, +\infty; L^2(\mathbb{R})\bigr).
\end{aligned}
\end{align}
Furthermore, as $t \to +\infty$,
\begin{equation}\label{con}
\sup_{x \in \mathbb{R}} \Bigl| (v, u)(t, x) - \Bigl( v^r\Bigl(\frac{x}{t}\Bigr) + \wt{v}^S\bigl(x - s t + \mb X(t)\bigr) - v_m,\; u^r\Bigl(\frac{x}{t}\Bigr) + \wt{u}^S\bigl(x - s t + \mb X(t)\bigr) - u_m \Bigr) \Bigr| \to 0,
\end{equation}
and
\begin{equation}\label{as}
\lim_{t \to +\infty} | \dot{\mb X}(t) | = 0.
\end{equation}
\end{theorem}

\begin{remark}
Theorem~\ref{thm:main} states that if the far-field states $(v_\pm, u_\pm)$ in \eqref{in} are connected by a superposition of a Liu shock and a rarefaction wave, then the solution to \eqref{VES} converges to the composite wave consisting of a self-similar rarefaction and a shifted  viscous shock.
\end{remark}

\begin{remark}
The shift function $\mb X(t)$ satisfies the asymptotic behavior \eqref{as}, implying
\[
\lim_{t \to +\infty} \frac{\mb X(t)}{t} = 0.
\]
Thus, $\mb X(t)$ grows at most sublinearly in $t$, and the shifted viscous shock retains its original traveling wave profile asymptotically in time.
%$\mb X(0)=x_0$ denotes the initial position of the viscous shock profile and it can be chosen arbitrarily.
\end{remark}

\begin{remark}
Theorem~\ref{thm:main} also covers the case of a single viscous shock ($\delta_R \equiv 0$), providing an alternative proof for the stability of a single  viscous Liu shock (covering both non-degenerate and degenerate cases). 
\end{remark}

\section{Set-up of the Problem and Proof of Theorem \ref{thm:main}}\label{sec-thm}
\setcounter{equation}{0}

\subsection{Construction of the Approximate 1-Rarefaction Wave}
\label{subsec:approx_rare}
The 1-rarefaction wave $(v^r, u^r)(t,x)$ defined in \eqref{rare} is only Lipschitz continuous for $t > 0$. To prove its time-asymptotic stability for \eqref{VES}, we construct a smooth approximate 1-rarefaction wave through the following Burgers equation (cf. \cite{MN86}):
\begin{equation} \label{ABE}
\begin{cases}
\displaystyle w_t + w w_x = 0, \\[4mm]
\displaystyle w(0, x) = w_0(x) := \frac{w_m + w_-}{2} + \frac{w_m - w_-}{2} \tanh x.
\end{cases}
\end{equation}
Since  $w_0^\prime(x)>0$ for $w_-< w_m$, the Burgers equation admits a unique global smooth solution $w(t, x)$ satisfying the following properties.

\begin{lemma} \label{lemma2.1}   (\cite{MN86})
Suppose $w_m > w_-$ and set $\bar{w} := w_m - w_-$. Then the problem \eqref{ABE} admits a unique global smooth solution $w(t, x)$ satisfying:
\begin{enumerate}
    \item $w_- < w(t, x) < w_m$ and $w_x(t, x) > 0$ for all $x \in \mathbb{R}$ and $t \geq 0$.
    \item For all $t > 0$ and $p \in [1, \infty]$, the following estimates hold:
    \begin{align*}
       & \|w_x(t, \cdot)\|_{L^p} \leq C \min\left\{|\bar{w}|, \, |\bar{w}|^{1/p} t^{-1 + 1/p}\right\}, \\
        &\|w_{xx}(t, \cdot)\|_{L^p} \leq C \min\left\{|\bar{w}|, \, t^{-1}\right\}.
    \end{align*}
    \item For all $t \geq 0$,
    \begin{align*}
        &|w(t, x) - w_m| \leq \bar{w} \, e^{-2|x - w_m t|}, \quad x \geq w_m t, \\
        &|w(t, x) - w_-| \leq \bar{w} \, e^{-2|x - w_- t|}, \quad x \leq w_- t.
    \end{align*}
    \item $\displaystyle \lim_{t \to +\infty} \sup_{x \in \mathbb{R}} \left| w(t, x) - w^r\left( \frac{x}{t} \right) \right| = 0$.
\end{enumerate}
\end{lemma}

Correspondingly, define the smooth approximate 1-rarefaction wave $(\widetilde{v}^R, \widetilde{u}^R)(t, x)$ via the relations:
\begin{align} \label{AR}
  \begin{aligned}
   % &\lambda_-(v_-) = w_-,\quad \lambda_-(v_m) = w_m, \\
    &\lambda_-(\widetilde{v}^R(t, x)) = w(t, x), \\
    &z_-(\widetilde{v}^R(t, x), \widetilde{u}^R(t, x)) = z_-(v_-, u_-) = z_-(v_m, u_m).
  \end{aligned}
\end{align}
%where $w(t,x)$ is the smooth solution to \eqref{ABE}. 
Then $(\widetilde{v}^R, \widetilde{u}^R)(t,x)$ satisfies the inviscid Euler system:
\begin{equation} \label{ARW}
  \begin{cases}
    \widetilde{v}^R_t - \widetilde{u}^R_x = 0, \\[2mm]
    \widetilde{u}^R_t - \sigma(\widetilde{v}^R)_x = 0.
  \end{cases}
\end{equation}

The properties of $(\widetilde{v}^R, \widetilde{u}^R)$ are inherited from Lemma \ref{lemma2.1} (see \cite{MN86, MY}).

\begin{lemma} \label{lemma1.2}
Let $\delta_R := |v_m - v_-| \sim |u_m - u_-|$ denote the strength of the rarefaction wave. The smooth approximate wave $(\widetilde{v}^R, \widetilde{u}^R)(t,x)$ defined by \eqref{AR} satisfies:
\begin{enumerate}
    \item $\widetilde{v}^R_x = -\dfrac{\lambda_-(\widetilde{v}^R)}{3 \widetilde{v}^R} w_x < 0$ and $\widetilde{u}^R_x = -\lambda_-(\widetilde{v}^R) \widetilde{v}^R_x < 0$ for all $x \in \mathbb{R}$ and $t \geq 0$.
    \item For all $t \geq 0$ and $p \in [1, \infty]$,
    \begin{align*}
        &\|(\widetilde{v}^R_x, \widetilde{u}^R_x)\|_{L^p} \leq C \min\left\{\delta_R, \, \delta_R^{1/p} t^{-1 + 1/p}\right\}, \\
        &\|(\widetilde{v}^R_{xx}, \widetilde{u}^R_{xx})\|_{L^p} \leq C \min\left\{\delta_R, \, t^{-1}\right\}, \\
        &|\widetilde{u}^R_{xx}| \leq C |\widetilde{u}^R_{x}|, \quad \forall x \in \mathbb{R}.
    \end{align*}
    \item For all $t \geq 0$,
    \begin{align*}
        &|(\widetilde{v}^R, \widetilde{u}^R)(t, x) - (v_m, u_m)| \leq C \delta_R \, e^{-2|x - \lambda_-(v_m) t|}, \quad x \geq \lambda_-(v_m) t, \\
        &|(\widetilde{v}^R, \widetilde{u}^R)(t, x) - (v_-, u_-)| \leq C \delta_R \, e^{-2|x - \lambda_-(v_-) t|}, \quad x \leq \lambda_-(v_-) t.
    \end{align*}
    \item $\displaystyle \lim_{t \to +\infty} \sup_{x \in \mathbb{R}} \left| (\widetilde{v}^R, \widetilde{u}^R)(t, x) - (v^r, u^r)\left( \frac{x}{t} \right) \right| = 0$.
\end{enumerate}
\end{lemma}

\subsection{Reformulation of the Problem}
\label{subsec:reformulation}

Our goal is to establish the time-asymptotic stability of the ansatz \eqref{ansatz} as stated in Theorem \ref{thm:main}.
%\begin{equation*}
%(\bar{v}, \bar{u})(t,x) := \left( v^r\left(\frac{x}{t}\right) + \widetilde{v}^S(x - s t + \mb X(t)) - v_m,\; u^r\left(\frac{x}{t}\right) + \widetilde{u}^S(x - s t + \mb X(t)) - u_m \right),
%\end{equation*}
%where $\mb X(t)$ is a time-dependent shift to be determined. 
Since the self-similar rarefaction wave $(v^r(x/t), u^r(x/t))$ is equivalent to its smooth approximation $(\widetilde{v}^R, \widetilde{u}^R)(1+t, x)$ by Lemma \ref{lemma1.2} (4),
we work on the  time-asymptotic stability of  the following ansatz:
\begin{equation}\label{ansatz_smooth}
(\widetilde{v}, \widetilde{u})(t,x) := \left( \widetilde{v}^R(1+t, x) + \widetilde{v}^S(x - s t + \mb X(t)) - v_m,\; \widetilde{u}^R(1+t, x) + \widetilde{u}^S(x - s t + \mb X(t)) - u_m \right).
\end{equation}

For convenience, we rewrite the viscoelastic system \eqref{VES} in the moving coordinate $\xi = x - s t$. Let $(v, u)(t, \xi) := (v, u)(t, \xi + s t)$, we obtain
\begin{align} \label{VES_moving}
\begin{cases}
        v_t - s v_\xi - u_\xi = 0, \\
        u_t - s u_\xi - \sigma(v)_\xi - \mu u_{\xi\xi} = 0.
\end{cases}
\end{align}
Similarly, the ansatz $(\widetilde{v}, \widetilde{u})(t, \xi) := (\widetilde{v}, \widetilde{u})(t, \xi + s t)$ satisfies
\begin{align} \label{ansatz_eq}
\begin{cases}
       \widetilde{v}_t - s \widetilde{v}_\xi-\dot{\mb X}(t) (\widetilde{v}^S)^\mb X_\xi - \widetilde{u}_\xi = 0, \\
       \widetilde{u}_t - s \widetilde{u}_\xi - \dot{\mb X}(t) (\widetilde{u}^S)^{\mb X}_\xi- \sigma(\widetilde{v})_\xi - \mu \widetilde{u}_{\xi\xi} = F,
\end{cases}
\end{align}
where % we denote $((\widetilde{v}^S)^{\mb X}, (\widetilde{u}^S)^{\mb X}) := (\widetilde{v}^S, \widetilde{u}^S)(\xi +\mb  X(t))$, and 
the error term $F$ is given by
\begin{equation}\label{F}
\begin{aligned}
F &= \big[ -\sigma(\widetilde{v}) + \sigma(\widetilde{v}^R) + \sigma((\widetilde{v}^S)^{\mb X}) \big]_\xi - \mu (\widetilde{u}^R)_{\xi\xi} \\
  &= -\Big[ \big( \sigma'(\widetilde{v}) - \sigma'((\widetilde{v}^S)^{\mb X}) \big) (\widetilde{v}^S)^{\mb X}_\xi + \big( \sigma'(\widetilde{v}) - \sigma'(\widetilde{v}^R) \big) \widetilde{v}^R_\xi \Big] - \mu (\widetilde{u}^R)_{\xi\xi} \\
  &:= F_1 + F_2.
\end{aligned}
\end{equation}
Here, $F_1:=-\Big[ \big( \sigma'(\widetilde{v}) - \sigma'((\widetilde{v}^S)^{\mb X}) \big) (\widetilde{v}^S)^{\mb X}_\xi + \big( \sigma'(\widetilde{v}) - \sigma'(\widetilde{v}^R) \big) \widetilde{v}^R_\xi \Big]$ captures the interaction between the Liu shock and the rarefaction wave, while $F_2:= - \mu (\widetilde{u}^R)_{\xi\xi} $ corresponds to the error term due to the inviscid rarefaction profile.

Now, define the perturbation
\begin{equation} \label{Per v}
(\phi, \psi)(t, \xi) := (v, u)(t, \xi) - (\widetilde{v}, \widetilde{u})(t, \xi).
\end{equation}
Subtracting \eqref{ansatz_eq} from \eqref{VES_moving}, we have that the perturbation $(\phi, \psi)$ satisfies the system
\begin{align} \label{Per}
\begin{cases}
\phi_t - s \phi_\xi + \dot{\mb X}(t) (\widetilde{v}^S)^{\mb X}_\xi - \psi_\xi = 0, \\
\psi_t - s \psi_\xi + \dot{\mb X}(t) (\widetilde{u}^S)^{\mb X}_\xi - [\sigma(\phi + \widetilde{v}) - \sigma(\widetilde{v})]_\xi  - \mu \psi_{\xi\xi} = -F, \\
(\phi, \psi)(0, \xi) = (\phi_0, \psi_0)(\xi) := (v_0, u_0)(\xi) - (\widetilde{v}_0, \widetilde{u}_0)(\xi).
\end{cases}
\end{align}
%with $F$ defined in \eqref{F}.

\subsection{Construction of the Weight Functions}
\label{subsec:weight}
Due to the possible degeneracy of Liu-shock, we need to construct different weight functions according to the following two cases:
\begin{itemize}
    \item \textbf{Case 1 (Nearly degenerate Liu shock):} $v_m \in [-v_+/2, (-1/2 + \theta_*) v_+)$.
    
    \item \textbf{Case 2 (Fully non-degenerate Liu shock):} $v_m \in [(-1/2 + \theta_*) v_+, 0)$,
\end{itemize}
where $\theta_* > 0$ is a uniformly small positive constant, independent of $v_+, v_m$ and $v_-$.

%The weight function $a(\wt{v}^S(\xi))$ is defined on a case-by-case basis.

\vskip0.3cm
\textbf{\underline{The weight function $a(\wt{v}^S(\xi))$ for Case 1}:}
% $v_m \in [-v_+/2, (-1/2 + \theta_*) v_+)$}

Since the shock profile $\wt{v}^S(\xi)$ is strictly increasing (Lemma \ref{lemma:shock_properties}), there exist two unique points $\xi_0$ and $\xi_*$ such that $\wt{v}^S(\xi_0) = 0$ and $\wt{v}^S(\xi_*) = v_* := v_m/2$, respectively. We define the weight function $a(\wt{v}^S(\xi))$ for Case 1 as follows:
\begin{equation}\label{w.deg}
a(\wt{v}^S(\xi)) :=
\begin{cases}
\displaystyle \frac{v_+ - v_*}{s} (\wt{v}^S + v_+ + v_m), \hspace{3.5cm} \xi \in (\xi_0, +\infty) \; \Leftrightarrow \; \wt{v}^S \in (0, v_+), \\[4mm]
\displaystyle \frac{v_+ - v_*}{s} \left[ v_+ + v_m + \wt{v}^S - \frac{3}{v_m^2}(\wt{v}^S)^3 + \frac{3}{v_m^3}(\wt{v}^S)^4 \right], \  \xi \in (\xi_*, \xi_0] \; \Leftrightarrow \; \wt{v}^S \in (v_*, 0], \\[4mm]
\displaystyle \frac{v_+ - v_*}{s} \left(\frac{\wt{v}^S}{4} +v_+ + \frac{19}{16}v_m\right), \hspace{2.75cm} \xi \in (-\infty, \xi_*) \; \Leftrightarrow \; \wt{v}^S \in (v_m, v_*].
\end{cases}
\end{equation}
One can verify that the above weight function $a \in C^2(\mathbb{R})$ and has the following bounds
\begin{align}
    0 < \frac{v_+ - v_*}{s} \left( \frac{23}{16}v_m + v_+ \right) <a< \frac{v_+ - v_*}{s} (2v_+ + v_m), \label{w.bounds2} \\
    \frac{v_+ - v_*}{4s} \leq a' := \frac{d a}{d\wt{v}^S} \leq \frac{v_+ - v_*}{s}, \quad \text{and} \quad a'' \geq 0. \label{w.deriv2}
\end{align}

\vskip0.3cm
\textbf{\underline{The weight function $a(\wt{v}^S(\xi))$ for Case 2}:}
% $v_m \in [(-1/2 + \theta_*) v_+, 0)$}

We define the weight function $a(\wt{v}^S(\xi))$ for Case 2 as follows:
\begin{equation}\label{w.nondeg}
a(\wt{v}^S) := \frac{\alpha \delta_S}{s} (\wt{v}^S + v_+ + v_m),
\end{equation}
where $\alpha := \frac{2}{3} + \frac{1}{6\theta_* + 3} > 0$, and $s$ is the shock speed given in \eqref{speed}. It follows directly that
\begin{align}
    0 < \frac{\alpha \delta_S}{s}(v_+ + 2v_m) < a < \frac{\alpha \delta_S}{s}(2v_+ + v_m), \label{w.bounds1} \\
    a^\prime:= \frac{d a}{d\wt{v}^S} = \frac{\alpha \delta_S}{s} > 0, \quad \text{and} \quad a'' \equiv 0. \label{w.deriv1}
\end{align}
The weight functions $a(\wt{v}^S(\xi))$ in \eqref{w.deg} and \eqref{w.nondeg} play a crucial role in our stability analysis by the $a$-contraction method.

\subsection{Construction of the Shift Functions}
\label{subsec:shift}

Correspondingly, the time-dependent shift functions $\mb X(t)$ are constructed as follows. %Denote $\big((\wt{v}^S)^{\mb X}(\xi), (\wt{u}^S)^{\mb X}(\xi) \big) := \big(\wt{v}^S(\xi + \mb X(t)), \wt{u}^S(\xi + \mb X(t))\big)$.
\vskip0.3cm
\textbf{\underline{Shift function $\mb X(t)$ for Case 1}:}
%$v_m \in [-v_+/2, (-1/2 + \theta_*) v_+)$}
%Define $\mb X(t)$ as the solution to the ODE
\begin{equation}\label{ode.shift2}
\begin{cases}
\displaystyle \dot{\mb X}(t) = \frac{24}{(v_+ - v_*)^2} \left[ \int_{\mathbb{R}} a((\wt{v}^S)^{\mb X}) \, \sigma'((\wt{v}^S)^{\mb X}) (\wt{v}^S)^{\mb X}_\xi \, \phi \, d\xi
+ \int_{\mathbb{R}} a((\wt{v}^S)^{\mb X}) (\wt{u}^S)^{\mb X}_\xi \, \psi \, d\xi \right], \\[4mm]
\mb X(0) = x_0,
\end{cases}
\end{equation}
\vskip0.3cm
\textbf{\underline{Shift function $\mb X(t)$ for Case 2}:}
% $v_m \in [(-1/2 + \theta_*) v_+, 0)$}
%Define $\mb X(t)$ as the solution to the ODE
\begin{equation}\label{ode.shift1}
\begin{cases}
\displaystyle \dot{\mb X}(t) = \frac{16}{\delta_S^2} \left[ \int_{\mathbb{R}} a((\wt{v}^S)^{\mb X}) \, \sigma'((\wt{v}^S)^{\mb X}) (\wt{v}^S)^{\mb X}_\xi \, \phi \, d\xi
+ \int_{\mathbb{R}} a((\wt{v}^S)^{\mb X}) (\wt{u}^S)^{\mb X}_\xi \, \psi \, d\xi \right], \\[4mm]
\mb X(0) = x_0,
\end{cases}
\end{equation}
where the initial shock location $x_0\in \mathbb{R}$ in both \eqref{ode.shift2} and \eqref{ode.shift1} can be chosen arbitrarily, with or without zero mass for the initial perturbation around the shock profile, and even with infinite mass. 

%where $x_0$ is any fixed constant,  and without loss of generality we set $x_0 = 0$.

Once the ODE  in \eqref{ode.shift2} or \eqref{ode.shift1}  for the shift function $\mb X(t)$ is defined, the shift  function $\mb X(t)$ is uniquely determined by the classical Cauchy-Lipschitz theorem  as in \cite{KMY}. To apply this lemma, denote the right-hand side of the ODE  \eqref{ode.shift2} or \eqref{ode.shift1} by $Q(t, \mb X)$. Uniform bounds can be established as follows:
\begin{align}
    &\sup_{\mb X} |Q(t,\mb X)| \leq \frac{C\|a\|_{L^\infty(\mathbb{R})} }{v_+^2} \| (\phi, \psi) \|_{L^\infty(\mathbb{R})} \int_{\mathbb{R}} \left( |(\wt{u}^S)^{\mb X}_\xi| + |(\wt{v}^S)^{\mb X}_\xi| \right) d\xi \leq Cv_+\| (\phi, \psi) \|_{L^\infty(\mathbb{R})}, \label{bound.M} \\
    &\qquad\qquad\qquad\qquad\sup_{\mb X} |\partial_{\mb X} Q(t, \mb X)| \leq  Cv_+\big[\| (\phi_\xi, \psi_\xi) \|_{L^\infty(\mathbb{R})}+\delta_S^2\big]\label{bound.dM}
\end{align}
for some constant $C > 0$, which are uniformly with respect to $t$ and $\mb X$. Consequently, both \eqref{ode.shift2}) and  \eqref{ode.shift1}  admit a unique absolutely continuous solution $\mb X(t)$ on any interval $t\in[0, T]$. Furthermore, since $|\dot{\mb X}(t)| \leq C$ by \eqref{bound.M} , we have the uniform estimate
\begin{equation}\label{growth.shift}
|\mb X(t)| \leq C t, \quad \forall t \in [0, T].
\end{equation}

\subsection{Proof of Theorem \ref{thm:main}.}
\label{subsec:proof_framework}

We want to prove Theorem \ref{thm:main} by the standard continuity argument, based on the local-in-time existence of the solution and the uniform-in-time a priori estimates.

For the local-in-time existence of the solution to the viscoelastic system \eqref{VES}, the shift function $\mb X(t)$ to the shock profile in the ansatz \eqref{ansatz_smooth} is not necessary. Therefore, we consider the background profile
$$
(\hat v, \hat u)(t,x) := \left( \widetilde{v}^S(x - s t+x_0) + \widetilde{v}^R(1+t, x) - v_m,\; \widetilde{u}^S(x - s t+x_0) + \widetilde{u}^R(1+t, x) - u_m \right).
$$
In the moving coordinate $\xi = x - s t$, the above profile $(\hat v, \hat u)(t, \xi) := (\hat v, \hat u)(t, \xi + s t)$ satisfies the system
\begin{align} \label{profile_eq}
\begin{cases}
       \hat v_t - s \hat v_\xi - \hat u_\xi = 0, \\
      \hat u_t - s \hat u_\xi - \sigma(\hat v)_\xi - \mu \hat u_{\xi\xi} = \hat F,
\end{cases}
\end{align}
where
\begin{equation} \label{profile_error}
\hat F := -\Big[ \big( \sigma'(\hat v) - \sigma'(\widetilde{v}^S) \big) \widetilde{v}^S_\xi + \big( \sigma'(\hat v) - \sigma'(\widetilde{v}^R) \big) \widetilde{v}^R_\xi \Big] - \mu \widetilde{u}^R_{\xi\xi}.
\end{equation}
The perturbation $(\hat \phi, \hat \psi)(t,\xi) := (v - \hat v, u - \hat u)(t,\xi)$ satisfies, for any initial time $\tau \ge 0$,
\begin{align} \label{pert_eq_unshifted}
\begin{cases}
\hat \phi_t - s \hat \phi_\xi - \hat \psi_\xi = 0,\qquad t>\tau, x\in\mathbb{R},\\
\hat \psi_t - s \hat \psi_\xi - [\sigma(\hat \phi + \hat v) - \sigma(\hat v)]_\xi - \mu \hat \psi_{\xi\xi} = -\hat F,\\
(\hat \phi, \hat \psi)(\tau, \xi) = (\hat \phi_\tau, \hat \psi_\tau)(\xi).
\end{cases}
\end{align}

Then we can obtain the local-in-time existence of the solution to the Cauchy problem \eqref{pert_eq_unshifted}, whose proof is standard and we omit the detailed proof for simplicity.

\begin{proposition}[Local existence of the solution]\label{prop:local_existence}
For any $M > 0$, there exists a time $t_0 = t_0(M) > 0$, independent of the initial time $\tau$, such that if $(\hat \phi_\tau, \hat \psi_\tau) \in H^1(\mathbb{R})$ with $\|(\hat \phi_\tau, \hat \psi_\tau)\|_{H^1(\mathbb{R})} \leq M$, then the Cauchy problem \eqref{pert_eq_unshifted} admits a unique solution $(\hat \phi, \hat \psi)$ on the interval $[\tau, \tau + t_0]$. Moreover, this solution satisfies
\begin{equation}\label{local_soln_est}
\begin{cases}
\hat \phi \in C([\tau, \tau + t_0]; H^1(\mathbb{R})), \\
\hat \psi \in C([\tau, \tau + t_0]; H^1(\mathbb{R})) \cap L^2(\tau, \tau + t_0; H^2(\mathbb{R})), \\
\displaystyle\sup_{t \in [\tau, \tau + t_0]} \| (\hat \phi, \hat \psi)(t) \|_{H^1(\mathbb{R})} \leq 2M.
\end{cases}
\end{equation}
\end{proposition}

The following proposition provides the key uniform-in-time a priori estimates for the original perturbation $(\phi, \psi)$ around the shifted ansatz \eqref{ansatz_smooth}, whose proof will be shown in Section \ref{sec-acontraction}.

\begin{proposition}[Uniform-in-time a priori estimates]\label{prop:apriori}
For a given far-field state $(v_+,u_+)\in\mathbb R_+\times\mathbb R$, there exists a constant $\delta_*>0$ such that if $|v_+|<\delta_*$, the following holds true. 

Assume that $(v_m,u_m)\in S_2(v_+,u_+)$ and
$(v_-,u_-)\in R_1(v_m,u_m)$ satisfy \eqref{s}. Let $\mb X(t)$ be the absolutely continuous shift function determined by \eqref{ode.shift1} (or \eqref{ode.shift2}) with the corresponding weight function $a$ given in \eqref{w.nondeg} (or \eqref{w.deg}). Let $(\phi,\psi)$ be a solution of the perturbation system \eqref{Per} on $[0,T]$ for some $T>0$, satisfying
\[
\phi \in C([0,T]; H^1(\mathbb{R})), \qquad 
\psi \in C([0,T]; H^1(\mathbb{R})) \cap L^2(0,T; H^2(\mathbb{R})).
\]
Then there exist constants $\delta_0, \eps_0, C_0 >0$ such that if $\delta_R < \delta_0$ and $\di \sup_{t\in[0,T]} \|(\phi,\psi)(t)\|_{H^1(\mathbb{R})}\le \varepsilon_0$, then the
following estimate holds for any $t\in[0,T]$:
\begin{align}\label{priori_main}
\begin{aligned}
& \|(\phi,\psi)(t)\|^2_{H^1(\mathbb{R})}
+ \int_0^t \left|\dot{\mb X}(\tau)\right|^2 \, d\tau
+\delta_S \int_0^t \int_{\mathbb{R}} \left|(\widetilde{v}^S)^{\mb X}_\xi\right| (\phi^2 + \psi^2)   d\xi d\tau \\
&\quad \quad+ |v_-|\int_0^t \int_{\mathbb{R}} \left |\widetilde{u}^R_\xi\right|\phi^2 \, d\xi d\tau
+ \int_0^t \|\phi_\xi(\tau)\|^2 \, d\tau
+ \int_0^t \|\psi_\xi(\tau)\|^2_{H^1(\mathbb{R})} \, d\tau \\
&\leq C_0 \left( \|(\phi_0, \psi_0)\|^2_{H^1(\mathbb{R})} + \delta_R^{1/3} \right),
\end{aligned}
\end{align}
and 
\begin{equation}\label{bound_Xdot}
|\dot{\mb X}(t)| \leq C \|(\phi, \psi)(t, \cdot)\|_{L^\infty(\mathbb{R})}, \qquad 0 \le t \le T.
\end{equation}
\end{proposition}

\vskip0.3cm
Based the local-in-time existence result Propositions \ref{prop:local_existence} and the uniform-in-time a priori estimates \ref{prop:apriori}, we can prove our main result Theorem \ref{thm:main}  as follows.

\textbf{Global-in-time existence and uniform-in-time estimates.} We proceed by a continuation argument. Define
\[
\varepsilon_* := \min\left\{ \frac{\varepsilon_0}{4},\; \sqrt{ \frac{\varepsilon_0^2}{64 C_0} - \delta_R^{1/3} } \right\}, \qquad M := \frac{\varepsilon_0}{4},
\]
where $\varepsilon_0$ and $C_0$ are the positive constants from Proposition \ref{prop:apriori}. Choose $\delta_0$ suitably small, we can ensure $\frac{\varepsilon_0^2}{64 C_0} - \delta_R^{1/3} > 0$, hence $\varepsilon_* > 0$.

If $\|(\hat\phi_0, \hat\psi_0)\|_{H^1(\mathbb{R})} < \varepsilon_* \le M$,  then by Proposition \ref{prop:local_existence}, there exists a time $T_1 = T_1(M) > 0$ such that a unique solution $(\hat\phi, \hat\psi)$ of \eqref{pert_eq_unshifted} exists on $[0, T_1]$ and satisfies
\[
\sup_{t \in [0, T_1]} \|(\hat\phi, \hat\psi)(t)\|_{H^1(\mathbb{R})} \leq 2 \|(\hat\phi_0, \hat\psi_0)\|_{H^1(\mathbb{R})} \leq \frac{\varepsilon_0}{2}.
\]
By Sobolev embedding, $\|(\hat\phi, \hat\psi)\|_{L^{\infty}(\mathbb{R})} \le \sqrt{2} \|(\hat\phi, \hat\psi)\|_{H^1(\mathbb{R})} \le C \varepsilon_0$, we have $\|(v, u)\|_{L^{\infty}(\mathbb{R})} \le C$ on $[0, T_1]$. Consequently, the bounds \eqref{bound.M} and \eqref{bound.dM} hold for $t \in [0, T_1]$, which guarantee the existence of the shift $\mb X(t)$ on this interval. Moreover, by choosing $T_1$ possibly smaller, we can ensure
\[
\|(\widetilde{v}^S - (\widetilde{v}^S)^{\mb X}, \widetilde{u}^S - (\widetilde{u}^S)^{\mb X})(t)\|_{H^1(\mathbb{R})} \leq C t \|(\widetilde{v}^S_\xi, \widetilde{u}^S_\xi)\|_{H^1(\mathbb{R})} \leq \frac{\varepsilon_0}{8}, \quad \forall t \in [0, T_1].
\]
Therefore, for $t \in [0, T_1]$,
\begin{align*}
\|(\phi, \psi)(t)\|_{H^1(\mathbb{R})}
&\le \|(\widetilde{v}^S - (\widetilde{v}^S)^{\mb X}, \widetilde{u}^S - (\widetilde{u}^S)^{\mb X})(t)\|_{H^1(\mathbb{R})} + \|(\hat \phi, \hat \psi)(t)\|_{H^1(\mathbb{R})} \\
&\le \frac{\varepsilon_0}{8} + \frac{\varepsilon_0}{2} < \varepsilon_0.
\end{align*}
Since $\mb X(t)$ is absolutely continuous and $(\hat\phi, \hat\psi) \in C([0,T_1];H^1(\mathbb{R}))$, we also have $(\phi, \psi) \in C([0,T_1];H^1(\mathbb{R}))$.

Define the maximal time of existence:
\[
T_{\max} := \sup \left\{ t > 0 \;\Big|\; \sup_{\tau \in [0, t]} \|(\phi, \psi)(\tau)\|_{H^1(\mathbb{R})} \le \varepsilon_0 \right\}.
\]
If $T_{\max} < \infty$, then the continuity implies $\displaystyle\sup_{\tau \in [0, T_{\max}]} \|(\phi, \psi)(\tau)\|_{H^1(\mathbb{R})} = \varepsilon_0$. However, applying Proposition \ref{prop:apriori} on $[0, T_{\max}]$ and using the fact that $\|(\phi_0, \psi_0)\|_{H^1(\mathbb{R})} = \|(\hat\phi_0, \hat\psi_0)\|_{H^1(\mathbb{R})}<\varepsilon_*$  yields
\[
\sup_{\tau \in [0, T_{\max}]} \|(\phi, \psi)(\tau)\|_{H^1(\mathbb{R})}
\leq \sqrt{ C_0 \left( \|(\phi_0, \psi_0)\|_{H^1(\mathbb{R})}^2 + \delta_R^{1/3} \right) }
\leq \sqrt{ C_0 \left( \varepsilon_*^2 + \delta_R^{1/3} \right) }
\leq \frac{\varepsilon_0}{8},
\]
which contradicts the fact $\displaystyle\sup_{\tau \in [0, T_{\max}]} \|(\phi, \psi)(\tau)\|_{H^1(\mathbb{R})} = \varepsilon_0$. Hence $T_{\max} = \infty$, and the solution exists globally in time.

Taking $T \to \infty$ in \eqref{priori_main} gives the global uniform-in-time estimates:
\begin{equation}\label{global_est}
\begin{aligned}
&\sup_{t \ge 0} \|(\phi,\psi)(t)\|^2_{H^1(\mathbb{R})}
+ \int_0^{\infty} \left|\dot{\mb X}(\tau)\right|^2 \, d\tau
+ {\delta_S}\int_0^{\infty} \int_{\mathbb{R}} \left |(\widetilde{v}^S)^{\mb X}_\xi\right|(\phi^2 + \psi^2) \, d\xi d\tau \\
&\qquad +|v_-| \int_0^{\infty} \int_{\mathbb{R}}\left |\widetilde{u}^R_\xi\right|  \phi^2\, d\xi d\tau
+ \int_0^{\infty} \|\phi_\xi(\tau)\|^2 \, d\tau
+ \int_0^{\infty} \|\psi_\xi(\tau)\|^2_{H^1(\mathbb{R})} \, d\tau \\
&\leq C_0 \left( \|(\phi_0, \psi_0)\|^2_{H^1(\mathbb{R})} + \delta_R^{1/3} \right),
\end{aligned}
\end{equation}
and from \eqref{bound_Xdot},
\begin{equation}\label{global_Xdot_bound}
|\dot{\mb X}(t)| \leq C \|(\phi, \psi)(t, \cdot)\|_{L^\infty(\mathbb{R})}, \qquad \forall t > 0.
\end{equation}

Since the self-similar rarefaction wave $(v^r, u^r)(x/t)$ is Lipschitz continuous in $x$ for $t>0$, and $(\widetilde{v}^R, \widetilde{u}^R)$ approximates it smoothly, we have
\[
v(t,x) - \left( v^r(x/t) + \widetilde{v}^S(x - s t + \mb X(t)) - v_m \right) \in C(0,\infty; H^1(\mathbb{R})),
\]
and
\[
u(t,x) - \left( u^r(x/t) + \widetilde{u}^S(x - s t + \mb X(t)) - u_m \right) \in C(0,\infty; H^1(\mathbb{R})).
\]
Moreover, \eqref{global_est} shows $\psi_{\xi\xi} \in L^2(0,\infty; L^2(\mathbb{R}))$, and Lemma \ref{lemma1.2} gives $\widetilde{u}^R_{\xi\xi} \in L^2(0,\infty; L^2(\mathbb{R}))$. Consequently, we have
\[
u_{xx}(t,x) - \widetilde{u}^S_{xx}(x - s t + \mb X(t)) \in L^2(0,\infty; L^2(\mathbb{R})),
\]
which completes the proof of \eqref{ext-main}.

Next, we choose $\varepsilon^*$ and $\delta^*$ in Theorem \ref{thm:main}.
It remains to ensure that the initial perturbation $\|(\phi_0, \psi_0)\|_{H^1(\mathbb{R})}$ is smaller than $\varepsilon_*$. From Lemma \ref{lemma1.2}, we have the bounds
\begin{align}
\|(\widetilde{v}^R(1,\cdot) - v_m, \widetilde{u}^R(1,\cdot) - u_m)\|_{L^2(\mathbb{R}_+)} &\leq C \delta_R, \label{est:b1} \\
\|(\widetilde{v}^R(1,\cdot) - v_-, \widetilde{u}^R(1,\cdot) - u_-)\|_{L^2(\mathbb{R}_-)} &\leq C \delta_R, \label{est:b2} \\
\|(\widetilde{v}^R_x(1,\cdot), \widetilde{u}^R_x(1,\cdot))\|_{L^2(\mathbb{R})} &\leq C \delta_R. \label{est:b3}
\end{align}
A direct estimation using the triangle inequality and \eqref{est:b1}-\eqref{est:b3} yields
\begin{align*}
\|(\phi_0, \psi_0)\|_{H^1(\mathbb{R})}
&\leq \|(v_0(\cdot) - \widetilde{v}^S(\cdot+x_0), u_0(\cdot) - \widetilde{u}^S(\cdot+x_0))\|_{L^2(\mathbb{R}_+)}\\
&\quad + \|(v_0(\cdot) - (\widetilde{v}^S(\cdot+x_0) + v_- - v_m), u_0(\cdot) - (\widetilde{u}^S(\cdot+x_0) + u_- - u_m))\|_{L^2(\mathbb{R}_-)} \\
&\quad   + \|(v_{0x}(\cdot) - \widetilde{v}^S_x(\cdot+x_0), u_{0x}(\cdot) - \widetilde{u}^S_x(\cdot+x_0))\|_{L^2(\mathbb{R})} + C_1 \delta_R,
\end{align*}
for some constant $C_1 > 0$. Therefore, by choosing
\[
\varepsilon^* := \frac{\varepsilon_*}{2}, \qquad \delta^* := \min\left\{ \delta_*, \frac{\varepsilon_*}{2C_1} \right\},
\]
and requiring the initial data to satisfy the hypotheses of Theorem \ref{thm:main} with these parameters, we obtain $\|(\phi_0, \psi_0)\|_{H^1(\mathbb{R})} \le \varepsilon_*$.

\vskip0.3cm
\textbf{Time-Asymptotic Behavior.} We now prove \eqref{con} and \eqref{as}. From the global estimate \eqref{global_est}, we have
\[
\int_0^{\infty} \|(\phi_\xi, \psi_\xi)(t)\|^2 +\left|  \frac{d}{dt} \|(\phi_\xi, \psi_\xi)(t)\|^2\right|dt < \infty.
\]
Consequently,
\[
\lim_{t \to \infty} \|(\phi_\xi, \psi_\xi)(t)\| = 0.
\]
By the Sobolev inequality $\|f\|_{L^\infty(\mathbb{R})} \le \sqrt{2} \|f\|^{1/2} \|f_\xi\|^{1/2}$ and \eqref{global_est}, it follows that
\[
\lim_{t \to \infty} \|(\phi, \psi)(t)\|_{L^\infty(\mathbb{R})} \le \sqrt{2} \lim_{t \to \infty} \left( \|\phi(t)\|^{1/2} \|\phi_\xi(t)\|^{1/2} + \|\psi(t)\|^{1/2} \|\psi_\xi(t)\|^{1/2} \right) = 0.
\]
This proves \eqref{con}. Finally, from \eqref{global_Xdot_bound}, we obtain
\[
\lim_{t \to \infty} \left|\dot{\mb X}(t)\right| \le C \lim_{t \to \infty} \|(\phi, \psi)(t)\|_{L^\infty(\mathbb{R})} = 0,
\]
which is precisely \eqref{as}. This completes the proof of Theorem \ref{thm:main}.

\section{A Priori Estimates via the $a$-Contraction Method}
\label{sec-acontraction}
\setcounter{equation}{0}

This section is devoted to proving the uniform-in-time a priori estimates in Proposition \ref{prop:apriori} by using $a$-contraction method. We assume that the Cauchy problem \eqref{Per} admits a solution
\[
\phi \in C([0,T]; H^1(\mathbb{R})), \quad \psi \in C([0,T]; H^1(\mathbb{R})) \cap L^2(0,T; H^2(\mathbb{R}))
\]
for some $T > 0$. First, we prove the following uniform-in-time $L^2$-estimate.  %The proof hinges on the following key energy estimate.

\begin{proposition} \label{prop:key_energy}
Under the hypotheses of Proposition \ref{prop:apriori}, there exists a uniform-in-time constant $C > 0$ such that for all $t \in [0, T]$,
\begin{align}\label{eq:key_energy_estimate}
\begin{aligned}
\|(\phi,\psi)(t)\|^2
&+ \int_0^t \left|\dot{\mb X}(\tau)\right|^2 \, d\tau
+ {\delta_S}\int_0^t \int_{\mathbb{R}} \left|(\widetilde{v}^S)^{\mb X}_\xi\right|  (\phi^2 + \psi^2)\, d\xi d\tau \\
&+|v_-| \int_0^t \int_{\mathbb{R}} \left|\widetilde{u}^R_\xi\right| \phi^2 \, d\xi d\tau
+ \int_0^t \|\psi_\xi(\tau)\|^2 \, d\tau\leq C \left( \|(\phi_0, \psi_0)\|^2 + \delta_R^{1/3} \right).
\end{aligned}
\end{align}
\end{proposition}

The proof of Proposition \ref{prop:key_energy} is based on the following lemmas.

\begin{lemma}\label{lem:energy_identity}
Let $a = a(\widetilde{v}^S(\xi))$ be the weight function defined in  \eqref{w.deg} or \eqref{w.nondeg}. Then, for $t \in [0,T]$, the following identity holds:
\begin{equation}\label{eq:energy_identity}
\begin{aligned}
&\frac{d}{dt} \int_{\mathbb{R}}  \bigg[a^{\mb X}\Big( \int_{\widetilde{v}}^{\widetilde{v}+\phi} \sigma(z) \, dz \; - \sigma(\widetilde{v}) \phi \, \Big)   + \frac{1}{2} a^{\mb X}\psi^2 \bigg] d\xi \\[2mm]
& \quad + \dot{\mb X}(t) \, \mathbf{Y}(t) + \mathcal{J}^{\text{good}}(t) + \mathcal{J}^{\text{bad}}(t) = \mathbf{F}(t),
\end{aligned}
\end{equation}
where
\begin{equation} \label{eq:def_Y}
\begin{aligned}
\mathbf{Y}(t) &:=
\int_{\mathbb{R}} a^{\mb X} (\widetilde{u}^S)^{\mb X}_\xi \, \psi \,  d\xi
+ \int_{\mathbb{R}} a^{\mb X} (\widetilde{v}^S)^{\mb X}_\xi \sigma^\prime(\wt v)\phi \, d\xi \\
&\quad - \int_{\mathbb{R}} a^{\mb X}_\xi \Big[ \int_{\widetilde{v}}^{\widetilde{v}+\phi} \sigma(z) \, dz -\sigma(\widetilde{v}) \phi  \Big] \, d\xi - \frac{1}{2} \int_{\mathbb{R}} a^{\mb X}_\xi \psi^2 \, d\xi,
\end{aligned}
\end{equation}
\begin{equation} \label{eq:def_Jgood} 
\begin{aligned}
\mathcal{J}^{\text{good}}(t) &:= 
\int_{\mathbb{R}} \Big( \frac{s}{2} a^{\mb X}_\xi - \frac{\mu}{2} a^{\mb X}_{\xi\xi} \Big) \psi^2 \, d\xi
+ \int_{\mathbb{R}} \mu \, a^{\mb X} \, \psi^2_\xi \, d\xi \\
&\quad + 3 \int_{\mathbb{R}} a^{\mb X}\widetilde{u}^R_\xi \, \widetilde{v}^R \, \phi^2 \, d\xi
+ \frac{s}{2} \int_{\mathbb{R}}  a^{\mb X}_\xi\sigma'(\widetilde{v})  \, \phi^2 \, d\xi, 
\end{aligned}
\end{equation}
\begin{equation} \label{eq:def_Jbad} 
\begin{aligned}
\mathcal{J}^{\text{bad}}(t) &:= 
- 3s \int_{\mathbb{R}} a^{\mb X} (\widetilde{v}^S)^{\mb X} (\widetilde{v}^S)^{\mb X}_\xi \, \phi^2 \, d\xi
+ \int_{\mathbb{R}} a^{\mb X}_\xi \sigma'(\widetilde{v}) \, \phi \psi \, d\xi\, d\xi\\
&\quad + 3 \int_{\mathbb{R}} a^{\mb X} \widetilde{u}^R_\xi \left[ (\widetilde{v}^S)^{\mb X} - v_m \right] \, \phi^2 \, d\xi  - 3s \int_{\mathbb{R}} a^{\mb X}(\widetilde{v}^S)^{\mb X}_\xi (\widetilde{v}^R - v_m)  \, \phi^2 \, d\xi \\
&\quad+ \int_{\mathbb{R}} a^{\mb X}\left[-s( \widetilde{v}^S)^{\mb X}_\xi + \widetilde{u}^R_\xi\right] \, \phi^3   -s \int_{\mathbb{R}} a^{\mb X}_\xi  \left( \widetilde v \phi^3+\frac{\phi^4}{4} \right)  d\xi 
- \int_{\mathbb{R}} a^{\mb X}_\xi (3\widetilde{v}\phi^2+\phi^3) \, \psi \, d\xi, 
\end{aligned}
\end{equation}
and
\begin{equation} \label{eq:def_F}
\begin{aligned}
\mathbf{F}(t) &:= -\int_{\mathbb{R}} a^{\mb X} F \, \psi \, d\xi.
\end{aligned}
\end{equation}
\end{lemma}

\begin{remark}\label{rem:term_classification}
Since $(a^{\mb X})' > 0$, $(\widetilde{v}^S)^{\mb X}_\xi > 0$, and $\widetilde{u}^R_\xi \, \widetilde{v}^R > 0$, the terms collected in $\mathcal{J}^{\text{good}}(t)$ are nonnegative. In contrast, $\mathcal{J}^{\text{bad}}(t)$ contains terms of negative and  indefinite sign as well as higher-order nonlinearities. The wave interaction term $\mathbf{F}(t)$ can be controlled using the wave interaction estimates from Lemma \ref{L4.8} (or Lemma \ref{L4.8.}) and the decay properties of the approximate rarefaction wave.
\end{remark}

\begin{proof}[Proof of Lemma \ref{lem:energy_identity}]
Multiplying the equation $\eqref{Per}_1$ by $a^{\mb X} [\sigma(\widetilde{v}+\phi) - \sigma(\widetilde{v})]$ and the equation $\eqref{Per}_2$ by $a^{\mb X} \psi $ with the weight function $a^{\mb X}= a(\widetilde{v}^S(\xi + \mb X(t)))$ defined in \eqref{w.deg} or \eqref{w.nondeg}, and then integrating the sum of the resulting equalities over $\mathbb{R}$ with respect to $\xi$, 
%and using the identity
%\begin{equation}\label{eq:product_rule}
%\begin{aligned}
%a^{\mb X} \big[ \sigma(\widetilde{v}+\phi) - \sigma(\widetilde{v}) \big] \phi_t = &
%\frac{d}{dt} \Big[ \int_{\widetilde{v}}^{\widetilde{v}+\phi} \sigma(s) \, ds \; a^{\mb X} - \sigma(\widetilde{v}) \phi \, a^{\mb X} \Big] \\
%&- \widetilde{v}_t \big[ \sigma(\widetilde{v}+\phi) - \sigma(\widetilde{v}) - \sigma'(\widetilde{v}) \phi \big] a^{\mb X} \\
%&- \dot{\mb X}(t) \, a^{\mb X}_\xi \Big[ \int_{\widetilde{v}}^{\widetilde{v}+\phi} \sigma(s) \, ds - \sigma(\widetilde{v}) \phi \Big],
%\end{aligned}
%\end{equation}
we obtain
\begin{equation}\label{eq:energy_intermediate}
\begin{aligned}
&\frac{d}{dt} \int_{\mathbb{R}}  \bigg[a^{\mb X}\Big( \int_{\widetilde{v}}^{\widetilde{v}+\phi} \sigma(z) \, dz \; - \sigma(\widetilde{v}) \phi \, \Big)   + \frac{1}{2} a^{\mb X}\psi^2 \bigg] d\xi \\
&\quad+ \dot{\mb X}(t) \bigg[ \int_{\mathbb{R}} a^{\mb X} (\widetilde{u}^S)^{\mb X}_\xi \, \psi \,  d\xi
+ \int_{\mathbb{R}} a^{\mb X} (\widetilde{v}^S)^{\mb X}_\xi \big[ \sigma(\widetilde{v}+\phi) - \sigma(\widetilde{v}) \big]\, d\xi \bigg] \\
&\quad- \dot{\mb X}(t)  \int_{\mathbb{R}} a^{\mb X}_\xi \bigg[\Big(\int_{\widetilde{v}}^{\widetilde{v}+\phi} \sigma(z) \, dz - \sigma(\widetilde{v}) \phi   \Big)
+\frac{1}{2}  \psi^2  \bigg] d\xi \\
&\quad- \int_{\mathbb{R}} a^{\mb X} \big[\widetilde{u}_\xi  +\dot{\mb X}(t) (\widetilde{v}^S)^{\mb X}_\xi\big]  \big[ \sigma(\widetilde{v}+\phi) - \sigma(\widetilde{v}) -\sigma^\prime(\widetilde{v})\phi\big] d\xi \\
&\quad+ \int_{\mathbb{R}} s a^{\mb X}_\xi \bigg[\Big( \int_{\widetilde{v}}^{\widetilde{v}+\phi} \sigma(z) \, dz  \; - \sigma(\widetilde{v}) \phi \, \Big)   + \frac{1}{2} \psi^2\bigg] d\xi \\
&\quad+ \int_{\mathbb{R}} \mu \, a^{\mb X} \, \psi^2_\xi \, d\xi
- \frac{\mu}{2} \int_{\mathbb{R}} a^{\mb X}_{\xi\xi} \psi^2 \, d\xi \\
&\quad+\; \int_{\mathbb{R}} a^{\mb X}_\xi \big[ \sigma(\widetilde{v}+\phi) - \sigma(\widetilde{v}) \big] \, \psi \, d\xi 
= -\int_{\mathbb{R}} a^{\mb X}  F \, \psi \, d\xi.
\end{aligned}
\end{equation}
Note that 
$$
\sigma(\widetilde{v}+\phi) - \sigma(\widetilde{v}) -\sigma^\prime(\widetilde{v})\phi=-3\widetilde{v}\phi^2-\phi^3,
$$
and
$$
 \int_{\widetilde{v}}^{\widetilde{v}+\phi} \sigma(z) \, dz  \; - \sigma(\widetilde{v}) \phi=\frac12(1-3\widetilde v^2)\phi^2-\widetilde v \phi^3-\frac{\phi^4}{4}.
$$
This completes the proof of Lemma \ref{lem:energy_identity}.
\end{proof}

For later analysis, we decompose $\mathbf{Y}(t)$ into three parts:
\begin{equation}\label{eq:Y_decomposition}
\begin{aligned}
\mathbf{Y}(t) = &\underbrace{\int_{\mathbb{R}} a^{\mb X} (\widetilde{u}^S)^{\mb X}_\xi \psi \,d\xi
+ \int_{\mathbb{R}} a^{\mb X} (\widetilde{v}^S)^{\mb X}_\xi \, \sigma'((\widetilde{v}^S)^{\mb X}) \, \phi \, d\xi}_{\mathbf{Y}_1(t)} \\
&+ \underbrace{\int_{\mathbb{R}} a^{\mb X} (\widetilde{v}^S)^{\mb X}_\xi \big[ \sigma'(\widetilde{v}) - \sigma'((\widetilde{v}^S)^{\mb X}) \big] \phi \, d\xi}_{\mathbf{Y}_2(t)} \\
%&+ \underbrace{\int_{\mathbb{R}} a^{\mb X} (\widetilde{v}^S)^{\mb X}_\xi \Big[ \sigma(\widetilde{v}+\phi) - \sigma(\widetilde{v}) - \sigma'(\widetilde{v}) \phi \Big]  d\xi
%- \frac{1}{2} \int_{\mathbb{R}} a^{\mb X}  (\widetilde{v}^S)^{\mb X}_\xi \, \sigma''(\widetilde{v}) \, \phi^2 \, d\xi}_{\mathbf{Y}_3(t)} \\
& \underbrace{-\int_{\mathbb{R}} a^{\mb X}_\xi \Big[\int_{\widetilde{v}}^{\widetilde{v}+\phi} \sigma(z) \, dz-\sigma(\widetilde{v}) \phi \Big] \, d\xi
- \frac{1}{2} \int_{\mathbb{R}} a^{\mb X}_\xi \psi^2 \, d\xi}_{\mathbf{Y}_3(t)}.
\end{aligned}
\end{equation}

   \begin{lemma}\label{L4.2}The following decomposition holds:
    \begin{equation}\label{4.13}
    \mathbf{\dot X}(t)\mathbf{Y}(t)+\mathcal{J}^{good}(t)+\mathcal{J}^{bad}(t)=\mathbf{G^{S,X}}(t)+\mathbf{G^R}(t)+\mathbf{N}(t)+\mathbf{J}(t),
   \end{equation}
   where
    \begin{equation}\label{4.14}
    \begin{aligned}
\mathbf{G^{S,X}}(t):=& \int_{\mathbb{R}}\left( \frac{s}{2} a^{\mb X}_\xi -\frac{\mu}{2} a^{\mb X}_{\xi\xi}\right) \psi^2 d\xi+  \int_{\mathbb{R}} \mu a^{\mb X}\psi_\xi^2 d\xi+\int_{\mathbb{R}} a^{\mb X}_\xi \sigma^\prime(( \wt v^S)^{\mb X}) \phi\psi d\xi\\
&+\frac{s}{2}\int_{\mathbb{R}}\left[a^{\mb X}\sigma^\prime((\wt v^S)^{\mb X})\right]_\xi \phi^2 d\xi
+ \mathbf{\dot X}(t)\mathbf{Y_1}(t),
 \end{aligned}
   \end{equation}
       \begin{equation}\label{4.15}
     \begin{aligned}
\mathbf{G^R}(t):=3 \int_{\mathbb{R}} a^{\mb X} \wt u^R_\xi \wt v^R  \phi^2 d\xi,
 \end{aligned}
   \end{equation}
      % \begin{equation}\label{4.2+}
     %\begin{aligned}
%\mathbf{P}(t):= \mathbf{\dot X}(t)\mathbf{Y_1}(t)-\frac{2}{(v_+-v_*)^2}\left[\int_{\xi_*}^{+\infty}\psi w^{\mb X}(\wt v^S)^{\mb X}_\xi d\xi\right]^2,
 %\end{aligned}
   %\end{equation}

   \begin{equation}\label{4.16}
     \begin{aligned}
\mathbf{N}(t):= & \mathbf{\dot X}(t)\left[\mathbf{Y_2}(t)+\mathbf{Y_3}(t)\right] + \int_{\mathbb{R}} a^{\mb X}\left[-s( \widetilde{v}^S)^{\mb X}_\xi + \widetilde{u}^R_\xi\right] \, \phi^3  \, d\xi\\
&\quad -s \int_{\mathbb{R}} a^{\mb X}_\xi  \left(\widetilde v \phi^3+\frac{\phi^4}{4} \right) d\xi - \int_{\mathbb{R}} a^{\mb X}_\xi (3\widetilde{v}\phi^2+\phi^3) \, \psi \, d\xi.\\
%  & +\frac{s}{6}\int_{\mathbb{R}}[a^{\mb X}\sigma^{\prime\prime}(\wt v+\phi)]_\xi \phi^3 d\xi
%+s \int_{\mathbb{R}}[\sigma(\wt v+\phi)-\sigma(\wt v)-\sigma^\prime(\wt v)\phi]a^{\mb X}_\xi \phi d\xi\\
%&-3s\int_{\mathbb{R}} a^{\mb X}( \wt v^S)^{\mb X}_\xi \left(\wt v^R-v_m \right) \phi^2 d\xi+\int_{\mathbb{R}} a^{\mb X}_\xi\left[\sigma^\prime(\wt v)-\sigma^\prime((\wt v^S)^{\mb X})\right] \phi \psi d\xi\\
%&+\frac{s}{2}\int_{\mathbb{R}}a^{\mb X}_\xi \left[\sigma^\prime(\wt v)-\sigma^\prime( (\wt v^S)^{\mb X})\right]\phi^2 d\xi,
 \end{aligned}
   \end{equation}
and
    \begin{equation}\label{4.17}
      \begin{aligned}
\mathbf{J}(t):=&3\int_{\mathbb{R}} a^{\mb X}\wt u^R_\xi\left[(\wt v^S)^{\mb X}-v_m\right]\phi^2 d\xi-3s\int_{\mathbb{R}} a^{\mb X}( \wt v^S)^{\mb X}_\xi \left(\wt v^R-v_m \right) \phi^2 d\xi\\
&+\int_{\mathbb{R}} a^{\mb X}_\xi\left[\sigma^\prime(\wt v)-\sigma^\prime((\wt v^S)^{\mb X})\right] \phi \psi d\xi+\frac{s}{2}\int_{\mathbb{R}}a^{\mb X}_\xi \left[\sigma^\prime(\wt v)-\sigma^\prime( (\wt v^S)^{\mb X})\right]\phi^2 d\xi.
 \end{aligned}
   \end{equation}
   \end{lemma}
   
   \begin{proof}The decomposition follows directly from the definitions of $\mathcal{J}^{\mathrm{good}}(t)$ and $\mathcal{J}^{\mathrm{bad}}(t)$ in Lemma \ref{lem:energy_identity} and the expression for $\mathbf{Y}(t)$ in \eqref{eq:Y_decomposition}, by regrouping terms and applying the following two identities:
     \begin{equation}\label{4.18}
      \begin{aligned}
&\frac{s}{2}\int_{\mathbb{R}} a^{\mb X}_\xi \sigma^\prime(\wt v) \phi^2 d\xi-3s\int_{\mathbb{R}}a^{\mb X}(\wt v^S)^{\mb X}_\xi (\wt v^S)^{\mb X} \phi^2 d\xi\\
%=&\frac{s}{2}\int_{\mathbb{R}}[\sigma^\prime(\wt v)-\sigma^\prime( (\wt v^S)^{\mb X})]a^{\mb X}_\xi \phi^2 d\xi
%+\frac{s}{2}\int_{\mathbb{R}}\sigma^\prime( (\wt v^S)^{\mb X})a^{\mb X}_\xi \phi^2 d\xi\\
%&+\frac{s}{2}\int_{\mathbb{R}}[\sigma^\prime((\wt v^S)^{\mb X})]_\xi a^{\mb X} \phi^2 d\xi\\
&=\frac{s}{2}\int_{\mathbb{R}}a^{\mb X}_\xi\left[\sigma^\prime(\wt v)-\sigma^\prime( (\wt v^S)^{\mb X})\right] \phi^2 d\xi+\frac{s}{2}\int_{\mathbb{R}}\left[a^{\mb X}\sigma^\prime((\wt v^S)^{\mb X})\right]_\xi \phi^2 d\xi,
 \end{aligned}
   \end{equation}
   and 
  \begin{equation}\label{4.19}
      \begin{aligned}
\int_{\mathbb{R}} a^{\mb X}_\xi\sigma^\prime(\wt v) \phi\psi d\xi=\int_{\mathbb{R}} a^{\mb X}_\xi\sigma^\prime((\wt v^S)^{\mb X}) \phi\psi d\xi+\int_{\mathbb{R}} a^{\mb X}_\xi[\sigma^\prime(\wt v)-\sigma^\prime( (\wt v^S)^{\mb X})] \phi \psi d\xi,
 \end{aligned}
   \end{equation}     
   \end{proof}   
   
   \begin{remark}
\label{rem:term_classification2}
The term $\mathbf{G}^{\mb R}(t)$ is clearly a good  term. Moreover, with the specific choice of the weight function $a$ in \eqref{w.nondeg} (or \eqref{w.deg}) and the shift $\mb X(t)$ defined by \eqref{ode.shift1} (or \eqref{ode.shift2}), the term $\mathbf{G}^{\mb S,\mb X}(t)$ also constitutes a good term. The term $\mathbf{N}(t)$ consists of bad (higher-order or indefinite) terms, which can be controlled by Lemma \ref{L4.6} under the assumption of small perturbations and rarefaction strength. Meanwhile, $\mathbf{F}(t)$ and $\mathbf{J}(t)$ will be handled via the wave interaction estimates from Lemma \ref{L4.8} (or Lemma \ref{L4.8.}) and the decay properties of the approximate rarefaction wave.
\end{remark}
To prove Proposition \ref{prop:key_energy}, we first need to estimate $\mathbf{G}^{\mb S,\mb X}(t)$. This analysis is performed separately for the two regimes defined earlier:
\begin{itemize}
\item \textbf{Case 1 (Nearly degenerate shock):} $v_m \in [-v_+/2, (-1/2 + \theta_*) v_+)$.
\item \textbf{Case 2 (Fully non-degenerate shock):} $v_m \in [(-1/2 + \theta_*) v_+, 0)$.
\end{itemize}
A key tool in the estimate of $\mathbf{G}^{\mb S,\mb X}(t)$ is the following weighted Poincar\'e-type inequality (see \cite[Lemma 2.9]{VKM}).

\begin{lemma}\label{P}
For any function $f:[0,1]\to\mathbb{R}$ satisfying $\int_0^1 y(1-y)|f'(y)|^2\, dy<\infty$, the following inequality holds:
\begin{equation}\label{poincare}
\int_0^1 \Bigl| f(y) - \int_0^1 f(z)\, dz \Bigr|^2 dy \le \frac{1}{2} \int_0^1 y(1-y) |f'(y)|^2 \, dy.
\end{equation}
\end{lemma}

\vskip0.3cm
\textbf{\underline{Estimation of $\mathbf{G}^{\mb S,\mb X}(t)$ for Case 1}:}

   We consider the regime $v_m \in [-v_+/2, (-1/2 + \theta_*) v_+)$. Performing the change of variable $\xi \to \xi - \mb X(t)$ and using the definition of the weight function $a$ from \eqref{w.deg}, we decompose $\mathbf{G}^{\mb S,\mb X}(t)$ as
\begin{equation}\label{eq:G_decomp}
    \mathbf{G}^{\mb S,\mb X}(t) = \mathbf{G}^{\mb S}(t) + \mathbf{G}^{\mb X}(t),
\end{equation}
where
\begin{align}
    \mathbf{G}^{\mb S}(t) &:=
    \int_{\mathbb{R}} \left( \frac{s}{2} a_\xi - \frac{\mu}{2} a_{\xi\xi} \right) (\psi^{-\mb X})^2 d\xi
    + \int_{\mathbb{R}} \mu a (\psi^{-\mb X}_\xi)^2  d\xi \notag \\
    &\quad + \int_{\mathbb{R}} a_\xi \sigma'(\widetilde{v}^S)   \phi^{-\mb X} \psi^{-\mb X} d\xi
    + \frac{s}{2} \int_{\mathbb{R}} \left[ a \sigma'(\widetilde{v}^S)  \right]_\xi (\phi^{-\mb X})^2 d\xi \notag \\
    &\quad + \frac{2}{(v_+ - v_*)^2} \left( \int_{\xi_*}^{+\infty} a \widetilde{v}^S_\xi \,\psi^{-\mb X}  d\xi \right)^2, \label{eq:def_GS} \\[2mm]
    \mathbf{G}^{\mb X}(t) &:= \dot{\mb X}(t) \, \mathbf{Y}_1(t)
    - \frac{2}{(v_+ - v_*)^2} \left( \int_{\xi_*}^{+\infty} a  \widetilde{v}^S_\xi   \psi^{-\mb X} d\xi \right)^2. \label{eq:def_GX}
\end{align}

Using the piecewise definition of $a$ in \eqref{w.deg}, we further split $\mathbf{G}^{\mb S}(t)$ into two parts:
\begin{equation}\label{eq:GS_split}
    \mathbf{G}^{\mb S}(t) = \mathbf{G}^{\mb S}_1(t) + \mathbf{G}^{\mb S}_2(t),
\end{equation}
with
\begin{align}
    \mathbf{G}^{\mb S}_1(t) &:=
    \int_{\xi_*}^{+\infty} \left( \frac{s}{2} a_\xi - \frac{\mu}{2} a_{\xi\xi} \right) (\psi^{-\mb X})^2  d\xi
    + \int_{\xi_*}^{+\infty} \mu a  (\psi^{-\mb X}_\xi)^2 \, d\xi \notag \\
    &\quad + \int_{\xi_*}^{+\infty}  a_\xi \sigma'(\widetilde{v}^S)\phi^{-\mb X} \psi^{-\mb X}  d\xi
    + \frac{s}{2} \int_{\xi_*}^{+\infty} \left[a \sigma'(\widetilde{v}^S)  \right]_\xi (\phi^{-\mb X})^2  d\xi \notag \\
    &\quad + \frac{2}{(v_+ - v_*)^2} \left( \int_{\xi_*}^{+\infty}a  \widetilde{v}^S_\xi  \psi^{-\mb X}  d\xi \right)^2, \label{eq:def_GS1} \\[2mm]
    \mathbf{G}^{\mb S}_2(t) &:=
    \int_{-\infty}^{\xi_*} \left( \frac{s}{2} a_\xi - \frac{\mu}{2} a_{\xi\xi} \right) (\psi^{-\mb X})^2 \, d\xi
    + \int_{-\infty}^{\xi_*} \mu a \, (\psi^{-\mb X}_\xi)^2 \, d\xi \notag \\
    &\quad + \int_{-\infty}^{\xi_*}a_\xi  \sigma'(\widetilde{v}^\mb S) \phi^{-\mb X} \psi^{-\mb X} \, d\xi
    + \frac{s}{2} \int_{-\infty}^{\xi_*} \left[a \sigma'(\widetilde{v}^S)  \right]_\xi (\phi^{-\mb X})^2 \, d\xi. \label{eq:def_GS2}
\end{align}

\begin{lemma}\label{lem:GS1_lower_bound}
If $v_+$ satisfies $0<v_+\ll 1$, then the term $\mathbf{G}^{\mb S}_1(t)$ defined in \eqref{eq:def_GS1} satisfies the following lower bound:
\begin{equation}\label{4-03}
\begin{aligned}
\mathbf{G}^{\mb S}_1(t) \geq &
\frac{s}{2} \bigintsss_{\xi_*}^{+\infty} \big[ a \sigma'(\widetilde{v}^S) \big]_\xi
\Bigg[ \phi^{-\mb X} + \frac{a_\xi \, \sigma'(\widetilde{v}^S)}{s \big( a \sigma'(\widetilde{v}^S) \big)_\xi}  \psi^{-\mb X} \Bigg]^2 d\xi \\
&+ \frac{1}{8} \int_{\xi_*}^{+\infty} \mu a  (\psi^{-\mb X}_\xi)^2  d\xi + \frac{3 v_m^2}{8 s^2} (v_+ - v_*) \int_{\xi_*}^{+\infty}  \widetilde{v}^S_\xi (\psi^{-\mb X})^2   d\xi.
\end{aligned}
\end{equation}
\end{lemma}   
   
 %\begin{lemma}\label{L4.3} $\mathbf{G}^S(t)$ defined in \eqref{4.14} satisfies that
%\begin{equation}\label{4.21}
   %  \begin{aligned}
%\mathbf{G}^S(t)\geq \int_{\mathbb{R}}(1-\alpha)w^{\mb{X}} \psi_\xi^2d\xi+\int_{\mathbb{R}} (\wt v^S)^{\mb{X}}_\xi h((\wt v^S)^{\mb{X}}) \psi^2 d\xi\\
 %+\frac{s}{2}\int_{\mathbb{R}}[w^{\mb{X}} \sigma^\prime((\wt v^S)^{\mb{X}})]_\xi \left[\phi+\frac{w^{\mb{X}}_\xi \sigma^\prime((\wt v^S)^{\mb{X}})}{s(w^{\mb{X}} \sigma^\prime((\wt v ^S)^{\mb{X}}))_\xi }\psi\right]^2 d\xi,
 %\end{aligned}
%\end{equation}

%where 
%\begin{equation}\label{4.22}
  %\begin{aligned}
%h((\wt v^S)^{\mb{X}}):=-\frac{[(w^\mb{X})^\prime\sigma^\prime((\wt v^S)^{\mb{X}})]^2}{2s [w^{\mb{X}} \sigma^\prime((\wt v^S)^{\mb{X}}) ]^\prime}+\frac{\alpha \delta_S}{2s^2}[1+3((\wt v ^S)^{\mb{X}}))^2-2v_+^2-2v_+v_m-2v_m^2]\\
%+\frac{\alpha^2 \delta_S}{s^2}[3(v_++v_m)(\wt v ^S)^{\mb{X}}+6v_+v_m+3v_+^2+3v_m^2]\geq \frac{\alpha \delta_S[3\theta^*+2(\theta^*)^2]v_+^2}{4s^2}.
 %\end{aligned}
%\end{equation}
%\end{lemma}

\begin{proof}   
Define the normalized variable 
\begin{equation}\label{transform}
y := \frac{v_+ - \widetilde{v}^S}{v_+ - v_*}.
\end{equation} 
Then $\xi \in (\xi_*, +\infty)$ is equivalent to $y \in [0, 1]$. Since $y_\xi = -\frac{\widetilde{v}^S_\xi}{v_+ - v_*} < 0$, for any fixed $t > 0$ we can set 
\begin{equation}\label{quanity}
\Psi(t, y) := a(\widetilde{v}^S(\xi))\psi^{-\mb X}(t, \xi).
\end{equation}
To apply the weighted Poincar\'e inequality (Lemma \ref{P}) to $\Psi$, we first have
\begin{equation}\label{4.23}
 \begin{aligned}
2\Big(\int_0^1\Psi(t,y) dy\Big)^2&%=2\Big(\int_{\xi_*}^{+\infty}  \frac{a\wt v^S_\xi}{v_+-v_*} \psi^{-\mb{X}}d\xi\Big)^2\\
&=\frac{2}{(v_+-v_*)^2}\left(\int_{\xi_*}^{+\infty}a\wt v^S_\xi  \psi^{-\mb{X}} d\xi\right)^2.
 \end{aligned}
\end{equation}
Furthermore, we have 
\begin{equation}\label{4.24}
 \begin{aligned}
2\int_0^1\Psi^2(t,y)dy=\frac{2}{v_+-v_*}\int_{\xi_*}^{+\infty} a^2\wt v^S_\xi ( \psi^{-\mb X})^2 d\xi,
%=\frac{2}{\delta_S}\int_{\mathbb{R}}  \psi^2 (w^{\mb X})^2(\wt v^S)^{\mb X}_\xi d\xi\\
%=\frac{2\alpha^2\delta_S}{s^2}\int_{\mathbb{R}}  \psi^2[((\wt v ^S)^{\mb{X}})^2+2(v_++v_m)(\wt v ^S)^{\mb{X}}+2v_+v_m+v_+^2+v_m^2] (\wt v^S)^{\mb X}_\xi d\xi.
 \end{aligned}
\end{equation}
and 
\begin{equation}\label{4.25}
 \begin{aligned}
\di \int_0^1 & |\Psi_y(t,y)|^2y(1-y)dy =\int^{\xi_*}_{+\infty}\left[(a\psi^{-\mb X})_\xi\right]^2\frac{y(1-y)}{y_\xi}d\xi\\
=&\frac{1}{v_+-v_*}\int_{\xi_*}^{+\infty}[(a\psi^{-\mb{X}})_\xi]^2\frac{(v_+-\wt v^S)(\wt v^S-v_*)}{\wt v^S_\xi}d\xi\\
=&\frac{1}{v_+-v_*}\int_{\xi_*}^{+\infty}[(a\psi^{-\mb{X}})_\xi]^2\frac{\mu s(\wt v^S-v_*)}{(\wt v^S-v_m)(\wt v^S+v_m+v_+)}d\xi\\
=&\frac{1}{v_+-v_*}\int_{\xi_*}^{+\infty}a^2\frac{\mu s(\wt v^S-v_*)}{(\wt v^S-v_m)(\wt v^S+v_m+v_+)}(\psi^{-\mb{X}}_\xi)^2d\xi\\
&+\frac{1}{v_+-v_*}\int_{\xi_*}^{+\infty}a_\xi^2\frac{\mu s(\wt v^S-v_*)}{(\wt v^S-v_m)(\wt v^S+v_m+v_+)}(\psi^{-\mb{X}})^2d\xi\\
&+\frac{1}{v_+-v_*}\int_{\xi_*}^{+\infty}2aa_\xi\frac{\mu s(\wt v^S-v_*)}{(\wt v^S-v_m)(\wt v^S+v_m+v_+)}\psi^{-\mb{X}}\psi^{-\mb{X}}_\xi d\xi\\
:=&I_1+I_2+I_3.
 \end{aligned}
\end{equation}
It shows that
\begin{equation}\label{4.26}
 \begin{aligned}
I_3=&\frac{1}{v_+-v_*}\int_{\xi_*}^{+\infty}2 aa_\xi\frac{\mu s(\wt v^S-v_*)}{(\wt v^S-v_m)(\wt v^S+v_m+v_+)}\psi^{-\mb{X}}\psi^{-\mb{X}}_\xi d\xi\\
=&\frac{1}{v_+-v_*}\int_{\xi_*}^{+\infty} aa_\xi\frac{\mu s(\wt v^S-v_*)}{(\wt v^S-v_m)(\wt v^S+v_m+v_+)} [(\psi^{-\mb{X}})^2]_\xi d\xi\\
=&-\frac{1}{v_+-v_*}\int_{\xi_*}^{+\infty}\left[aa_\xi\frac{\mu s(\wt v^S-v_*)}{(\wt v^S-v_m)(\wt v^S+v_m+v_+)}\right]_\xi (\psi^{-\mb{X}})^2d\xi \\
%=&-I_2-\int_{\xi_*}^{+\infty}(\psi^{-\mb{X}})^2\left[ww_{\xi\xi}\frac{\mu s(\wt v^S-v_*)}{(\wt v^S-v_m)(\wt v^S+v_m+v_+)}+ww_\xi\left(\frac{\mu s(\wt v^S-v_*)}{(\wt v^S-v_m)(\wt v^S+v_m+v_+)}\right)_\xi\right]d\xi \frac{1}{v_+-v_*}\\
=&-\frac{1}{v_+-v_*}\int_{\xi_*}^{+\infty}a\big[ {a^{\prime\prime}}(\wt v^S-v_*)(v_+-\wt v^S)\\
&\qquad\qquad\qquad\qquad\qquad\qquad-a^\prime(2\wt v^S-v_*-v_+)\big]\wt v^S_\xi (\psi^{-\mb{X}})^2d\xi-I_2,
 \end{aligned}
\end{equation}
where we have used the fact 
\begin{equation*}
\begin{aligned}
\mu sa_{\xi\xi}=&\mu sa^{\prime\prime}(\wt v^S_\xi)^2+\mu sa^\prime \wt v^S_{\xi\xi}\\
=&\big[ a^{\prime\prime}(v_+-\wt v^S)(\wt v^S-v_m)(\wt v^S+v_m+v_+) +a^\prime\left(v_m^2+v_mv_++v_+^2-3(\wt v^S)^2\right)\big]\wt v^S_\xi.
 \end{aligned}
\end{equation*}

Applying the weighted Poincar\'e inequality \eqref{poincare} to the function $\Psi$ defined in \eqref{quanity}, and using the identities \eqref{4.23}-\eqref{4.26} together with the definition of $\mathbf{G}^{\mb S}_1(t)$ in \eqref{eq:def_GS1}, we obtain the following lower bound:
\begin{equation}\label{4.27}
 \begin{aligned}
\mathbf{G^S_1}(t)\geq& \bigintsss_{\xi_*}^{+\infty}\mu a \left[1-\frac{sa(\wt v^S-v_*)}{(\wt v^S-v_m)(\wt v^S+v_m+v_+)(v_+-v_*)}\right]  (\psi_\xi^{-\mb{X}})^2d\xi\\
&+ \frac{s}{2}\bigintsss_{\xi_*}^{+\infty}\left[a \sigma^\prime(\wt v^S)\right]_\xi\left[  \phi^{-\mb{X}} +\frac{a_\xi\sigma^\prime(\wt v^S)}{s\left(a \sigma^\prime(\wt v^S)\right)_\xi} \psi^{-\mb{X}}\right]^2 d\xi\\
&+\int_{\xi_*}^{+\infty}\left(H_1+H_2+H_3+H_4\right) \wt v^S_\xi  (\psi^{-\mb{X}})^2d\xi,
%&+\int_{\xi_*}^{+\infty}(\psi^{-\mb{X}})^2\wt v^S_\xi \underbrace{\left[-\frac{\left(w^\prime\sigma^\prime(\wt v^S)\right)^2}{2s\left(w \sigma^\prime(\wt v^S)\right)^\prime}\right]}_{H_1}d\xi\\
%&+\int_{\xi_*}^{+\infty}(\psi^{-\mb{X}})^2\wt v^S_\xi\\
%& \times\underbrace{w\left[ {w^{\prime\prime}}(\wt v^S-v_*)(v_+-\wt v^S)-w^\prime(2\wt v^S-v_*-v_+)+2w \right]\frac{1}{v_+-v_*}}_{H_2}d\xi\\
%&+\int_{\xi_*}^{+\infty}(\psi^{-\mb{X}})^2\wt v^S_\xi \underbrace{\frac{w^\prime}{2s}\left[1-2v_m^2-2v_+^2-2v_mv_++3(\wt v^S)^2\right]}_{H_3}d\xi\\
%&+\int_{\xi_*}^{+\infty}(\psi^{-\mb{X}})^2\wt v^S_\xi \underbrace{\left[-\frac{w^{\prime\prime}}{2s}(v_+-\wt v^S)(\wt v^S-v_m)(\wt v^S+v_m+v_+)\right]}_{H_4}d\xi,
\end{aligned}
\end{equation}
where $H_1$ to $H_4$ are given as
\begin{equation*}
 \begin{aligned}
 H_1&=-\frac{\left(a^\prime\sigma^\prime(\wt v^S)\right)^2}{2s\left(a \sigma^\prime(\wt v^S)\right)^\prime},\\
 H_2&=\frac{a}{v_+-v_*}\left[ {a^{\prime\prime}}(\wt v^S-v_*)(v_+-\wt v^S)-a^\prime(2\wt v^S-v_*-v_+)+2a\right] ,\\
 H_3&=\frac{a^\prime}{2s}\left[1-2v_m^2-2v_+^2-2v_mv_++3(\wt v^S)^2\right],\\
 H_4&=-\frac{a^{\prime\prime}}{2s}(v_+-\wt v^S)(\wt v^S-v_m)(\wt v^S+v_m+v_+).
 \end{aligned}
\end{equation*}
 
The construction of the weight function $a$ in \eqref{w.deg} yields the following crucial estimates:
\begin{equation}\label{4.28.}
 \begin{array}{ll}
\di 1-\frac{sa(\wt v^S-v_*)}{(v_+-v_*)(\wt v^S-v_m)(\wt v^S+v_m+v_+)}\\[5mm]
\di =\begin {cases}
\di \frac{v_*-v_m}{\wt v^S-v_m}=\frac{-\frac12 v_m}{\wt v^S-v_m}>%\frac{\frac12(v_+-v_m)-\frac12 v_+}{v_+-v_m}=
\frac{1}{2}-\frac{v_+}{2(v_+-v_m)}>\frac{1}{8},  &\quad \wt v^S\in(0,v_+),\\[4mm]
\di \frac{v_*-v_m}{\wt v^S-v_m}-\frac{3(\wt v^S)^3}{v_m^3}\frac{\wt v^S-v_*}{\wt v^S+v_++v_m}> \frac{1}{8},&\quad\wt v^S\in(v_*,0].
 \end {cases}
\end{array}
\end{equation}
In fact, since $v_m \in [-\frac{v_+}{2}, (-\frac12 + \theta_*) v_+)$ in Case 1, one has
$$
\frac{v_+}{2(v_+-v_m)}<\frac{v_+}{2\big(v_++(\frac12 - \theta_*) v_+\big)}=\frac{1}{3-2\theta_*}.
$$
Then the first inequality in \eqref{4.28.} holds provided that the uniform constant $\theta_*$ is chosen suitably small.  When $\wt v^S\in(v_*,0]$, we  have
\begin{equation*}\label{4.28}
 \begin{aligned}
& \frac{v_*-v_m}{\wt v^S-v_m}-\frac{3(\wt v^S)^3}{v_m^3}\frac{\wt v^S-v_*}{\wt v^S+v_++v_m}\\[1mm]
&=\frac{\frac12|v_m|}{\wt v^S-v_m}-\frac{3(\wt v^S)^3}{v_m^3}\frac{\wt v^S-v_*}{\wt v^S+v_++v_m}\\[1mm]
&> \frac{\frac12|v_m|}{\wt v^S-v_m}\Big|_{\wt v^S=0}-\frac{3}{v_m^3}\frac{1}{\wt v^S+v_++v_m}\Big|_{\wt v^S=v_*}{(\wt v^S-v_*)}(\wt v^S)^3\\[2mm]
&=\frac{1}{2}+\frac{3}{|v_m|^3}\frac{1}{v_++v_*+v_m}{(\wt v^S-v_*)}(\wt v^S)^3\\[1mm]
&\geq \frac{1}{2}+\frac{3}{|v_m|^3}\frac{1}{v_++v_*+v_m}{(\wt v^S-v_*)}(\wt v^S)^3\Big|_{\wt v^S=\frac{3v_*}{4}}\\[1mm]
&=\frac{1}{2}+ \left(\frac38\right)^4\frac{v_m}{v_++\frac{3}{2}v_m}>\frac{1}{2}-2 \left(\frac38\right)^4>\frac{1}{8}. %\ \ \ for \ \ 0< \theta_*\ll1.
\end{aligned}
\end{equation*}
Therefore, \eqref{4.28.} is proved. 

Next we  calculate the sum $H_1 + H_2 + H_3 + H_4$ in \eqref{4.27}. % separately on the two intervals corresponding to the sign of $\widetilde{v}^S$.
When $\wt v^S\in(0,v_+)$, by \eqref{w.deg}, we have
\begin{equation}\label{4.29}
\begin{aligned}
&H_1+H_2+H_3+H_4 \\
&= \frac{v_+-v_*}{2s^2\bigl[1-9(\wt v^S)^2-6(v_m+v_+)\wt v^S\bigr]}  \cdot \Bigl[
    \bigl(1+4{v_+^2}+9v_mv_++3v_m^2+6v_+\wt v^S+5v_m\wt v^S+3(\wt v^S)^2\bigr) \\
&\qquad\qquad\qquad \qquad\qquad \qquad\qquad \qquad \  \cdot \bigl(1-9(\wt v^S)^2-6(v_m+v_+)\wt v^S\bigr)  -1+6(\wt v^S)^2-9(\wt v^S)^4 \Bigr] \\[2mm]
&= \frac{v_+-v_*}{2s^2\Big[1-9(\wt v^S)^2-6(v_m+v_+)\wt v^S\Big]}
   \Big[3v_m^2+4v_+^2+9v_mv_+-v_m\wt v^S+O(v_+^4)\Big] \\[2mm]
&\geq \frac{v_+-v_*}{2s^2\Big[1-9(\wt v^S)^2-6(v_m+v_+)\wt v^S\Big]}
       \Big[3v_m^2+4v_+^2-9|v_m|v_++O(v_+^4)\Big] \\[2mm]
&\geq \frac{v_+-v_*}{2s^2\Big[1-9(\wt v^S)^2-6(v_m+v_+)\wt v^S\Big]}
       \Big[3v_m^2+16v_m^2-\frac{9}{\frac12-\theta_*}v_m^2+O(v_+^4)\Big] \\[2mm]
%&\geq \frac{v_+-v_*}{2s^2}
 %      \bigl(3v_m^2+4v_+^2-9|v_m|v_+\bigr) > \frac{v_+-%v_*}{2s^2}
%       \Big(3v_m^2+16v_m^2-\frac{9}{\frac12-%\theta_*}v_m^2\Big)\\[2mm]
 &     > \frac{3(v_+-v_*)}{8s^2} v_m^2,
\end{aligned}
\end{equation}
provided that both $v_+>0$ and $\theta_*>0$ are uniformly suitably small.

When $\wt v^S\in(v_*,0]$, as shown in the appendix, we have
\begin{equation}\label{4.33}
\begin{aligned}
H_1 + H_2 + H_3 + H_4
&= \frac{v_+-v_*}{2s^2(K_1+K_2)} \Bigl[ K_1\bigl(6(\wt v^S)^2 K_1+K_2+2K_3+K_4\bigr) \\
&\qquad\qquad\qquad\qquad-9(\wt v^S)^4 K_1^2+2K_2K_3+K_2K_4 \Bigr] \\
&= \frac{v_+-v_*}{2s^2(K_1+K_2)} \Bigl[ K_1\bigl(6(\wt v^S)^2 K_1+K_2+2K_3+K_4\bigr) + O(v_+^4) \Bigr],
\end{aligned}
\end{equation}
where 
\begin{equation}\label{K}
\begin{aligned}
K_1 &:= 1-\frac{9}{v_m^2}(\wt v^S)^2+\frac{12}{v_m^3}(\wt v^S)^3 = O(1), \\
K_2 &:= -3(\wt v^S)^2K_1-6\wt v^S\bigg[v_++v_m+\wt v^S-\frac{3}{v_m^2}(\wt v^S)^3+\frac{3}{v_m^3}(\wt v^S)^4 \bigg] = O(v_+^2), \\
K_3 &:= (v_++v_m)\Big(3v_++\frac{5}{2}v_m\Big)+\Big(\frac{9v_+^2}{v_m}+12v_++\frac{5v_m}{2}\Big)\wt v^S \\
    &\quad -\Big(\frac{45v_+^2}{v_m^2}+\frac{99v_+}{2v_m}+\frac{27}{2}\Big)(\wt v^S)^2 \\
    &\quad +\Big(\frac{48v_+^2}{v_m^3}+\frac{48v_+}{v_m^2}+\frac{33}{v_m}\Big)(\wt v^S)^3+\Big(\frac{15}{2v_m^2}-\frac{24v_+}{v_m^3}\Big)(\wt v^S)^4 \\
    &\quad +\Big(\frac{162v_+}{v_m^4}-\frac{27}{2v_m^3}\Big)(\wt v^S)^5-\Big(\frac{405}{2v_m^4}+\frac{279v_+}{v_m^5}\Big)(\wt v^S)^6 \\
    &\quad +\Big(\frac{324}{v_m^5}+\frac{144v_+}{v_m^6}\Big)(\wt v^S)^7-\frac{162}{v_m^6}(\wt v^S)^8 = O(v_+^2), \\
    \end{aligned}
\end{equation}
and
\begin{equation}\label{K4}
\begin{aligned}
K_4 &:= -2v_m^2-2v_+^2-2v_mv_+-\Big(18v_++\frac{18v_+^2}{v_m}\Big)\wt v^S \\
    &\quad +\Big(39+\frac{72v_+}{v_m}+\frac{72v_+^2}{v_m^2}\Big)(\wt v^S)^2-\Big(\frac{60}{v_m}+\frac{60v_+}{v_m^2}+\frac{60v_+^2}{v_m^3} \Big)(\wt v^S)^3 \\
    &\quad -\frac{45}{v_m^2}(\wt v^S)^4+\frac{72}{v_m^3}(\wt v^S)^5 = O(v_+^2),
\end{aligned}
\end{equation}
provided that $v_+>0$ is uniformly suitably small.
Consequently, for uniformly suitably small $v_+ > 0$ and $\wt v^S\in(v_*,0]$, we have the following uniform bounds
\[
|K_1| \leq C, \quad |K_2| \leq C v_+^2, \quad |K_3| \leq C v_+^2, \quad \text{and} \quad |K_4| \leq C v_+^2,
\]
where $C$ is a positive constant independent of $v_+$.
%\begin{equation}\label{4.30}
% \begin{aligned}
%H_1=&\frac{1}{2s^2}\Big[\frac{35}{16}v_+^3-\frac{145}{8}v_+^2\wt v^S-\frac{945v_+}{4}(\wt v^S)^2-{645}(\wt v^S)^3+\frac{555}{v_+}(\wt v^S)^4\\
%&+\frac{6750}{v_+^2}(\wt v^S)^5+\frac{14220}{v_+^3}(\wt v^S)^6-\frac{2880}{v_+^4}(\wt v^S)^7-\frac{25920}{v_+^5}(\wt v^S) ^8\Big]:=\frac{K_1}{2s^2},
%\end{aligned}
%\end{equation}
%\begin{equation}\label{4.31}
 %\begin{aligned}
%H_2=&\frac{1}{2s^2}\Big[\underbrace{ \left(\frac{5v_+}{4}-\frac{45}{v_+}(\wt v^S)^2-\frac{120}{v_+^2}(\wt v^S)^3\right)}_{K_2}\\
%&+\underbrace{\left(-\frac{15}{8}v_+^3+\frac{45}{2}v_+^2\wt v^S+\frac{915}{4}v_+(\wt v^S)^2+450(\wt v^S)^3-\frac{225}{v_+}(\wt v^S)^4-\frac{720}{v_+^2}(\wt v^S)^5\right)}_{K_3}\Big]\\
%&:=\frac{K_2+K_3}{2s^2},
%\end{aligned}
%\end{equation}
%and 
%\begin{equation}\label{4.32}
 %\begin{aligned}
%H_3=&-\frac{1}{2s^2}\frac{K_2^2\left[1-6(\wt v^S)^2+9(\wt v^S)^4\right]}{K_2\underbrace{-3(\wt v^S)^2K_2-\frac{15}{4}v_+^2\wt v^S-\frac{15}{2}v_+(\wt v^S)^2+\frac{90}{v_+}(\wt v^S)^4+\frac{180}{v_+^2}(\wt v^S)^5}_{K_4}}\\
%:=&-\frac{1}{2s^2}\frac{K_2^2\left[1-6(\wt v^S)^2+9(\wt v^S)^4\right]}{K_2+K_4}.
%\end{aligned}
%\end{equation}
%It is obvious that $K_1=O(v_+^3), K_2=O(v_+), K_3=O(v_+^3)$ and $K_4=O(v_+^3)$. Thus, we have 

We now estimate $6(\widetilde{v}^S)^2K_1+K_2+2K_3+K_4$ in  \eqref{4.33}. Direct computations yield that
\begin{equation}\label{4.34}
 \begin{aligned}
6&(\wt v^S)^2K_1+K_2+2K_3+K_4\\
=&v_m^2\bigg[\Big(3+\frac{9v_+}{v_m}+\frac{4v_+^2}{v_m^2}\Big)-\frac{\wt v^S}{v_m}+\Big({9}-\frac{27v_+}{v_m}-\frac{18v_+^2}{v_m^2}\Big)\Big(\frac{\wt v^S}{v_m}\Big)^2\\
&\qquad+\Big({6}+\frac{36v_+}{v_m}+\frac{36v_+^2}{v_m^2} \Big)\Big(\frac{\wt v^S}{v_m}\Big)^3-\Big({39}+\frac{48v_+}{v_m}\Big)\Big(\frac{\wt v^S}{v_m}\Big)^4\\
&\qquad+\Big({63}+\frac{324v_+}{v_m}\Big)\Big(\frac{\wt v^S}{v_m}\Big)^5-\Big({405}+\frac{558v_+}{v_m}\Big)\Big(\frac{\wt v^S}{v_m}\Big)^6\\
&\qquad+\Big({648}+\frac{288v_+}{v_m}\Big)\Big(\frac{\wt v^S}{v_m}\Big)^7-{324}\Big(\frac{\wt v^S}{v_m}\Big)^8\bigg].\\
\end{aligned}
\end{equation}
%Since $v_m\in[-\frac{v_+}{2},(-\frac{1}{2}+\theta_*)v_+)$, we have $\left(\frac{v_+}{v_m}+2\right)\in\left(\frac{-4\theta_*}{1-2\theta_*},0\right]$. Substituting $\frac{v_+}{v_m}=\left(\frac{v_+}{v_m}+2\right)-2$ into \eqref{4.34} and choosing $\theta_*$ sufficiently small, we obtain
Since $v_m\in[-\frac{v_+}{2},-(\frac{1}{2}-\theta_*)v_+)$, we have $\frac{v_+}{v_m}\in\left(-\frac{2}{1-2\theta_*},-2\right]$. 
Then the follow seven inequalities hold true.
$$
3+\frac{9v_+}{v_m}+\frac{4v_+^2}{v_m^2}>3-9\cdot\frac{2}{1-2\theta_*}+4\cdot 4=19-\frac{18}{1-2\theta_*}=1-\frac{36\theta_*}{1-2\theta_*}.
$$
$$
{9}-\frac{27v_+}{v_m}-\frac{18v_+^2}{v_m^2}>9+27\cdot2-18\cdot \frac{4}{(1-2\theta_*)^2}=-9-\frac{288\theta_*(1-\theta_*)}{(1-2\theta_*)^2}.
$$
$$
{6}+\frac{36v_+}{v_m}+\frac{36v_+^2}{v_m^2} >6-\frac{72}{1-2\theta_*}+36\cdot 4=78-\frac{144\theta_*}{1-2\theta_*}.
$$
$$
-\Big({39}+\frac{48v_+}{v_m}\Big)\geq -39+48\cdot 2=57.
$$
$$
{63}+\frac{324v_+}{v_m}>63-324\cdot \frac{2}{1-2\theta_*}=-585-\frac{1296\theta_*}{1-2\theta_*}
$$
$$
-\Big({405}+\frac{558v_+}{v_m}\Big)\geq -405+558\cdot 2=711.
$$
$$
{648}+\frac{288v_+}{v_m}>648-288\cdot \frac{2}{1-2\theta_*}=72-\frac{1152\theta_*}{1-2\theta_*}.
$$
Substituting the above inequalities into \eqref{4.34} and choosing $\theta_*$ suitably small, we can obtain
\begin{equation}\label{4.34-1}
 \begin{aligned}
6(\wt v^S)^2K_1+K_2+2K_3+K_4>v_m^2\bigg[1-\frac{\wt v^S}{v_m}-9\Big(\frac{\wt v^S}{v_m}\Big)^2+78\Big(\frac{\wt v^S}{v_m}\Big)^3
+57\Big(\frac{\wt v^S}{v_m}\Big)^4\\
-585\Big(\frac{\wt v^S}{v_m}\Big)^5
+711\Big(\frac{\wt v^S}{v_m}\Big)^6+72\Big(\frac{\wt v^S}{v_m}\Big)^7-324\Big(\frac{\wt v^S}{v_m}\Big)^8+O(\theta_*)\bigg]. %\geq 0.76 v_m^2.
\end{aligned}
\end{equation}
For $\wt v^S \in [v_*,0)$,  we have
$$
0<\frac{\wt v^S}{v_m}\leq \frac 12.
$$
Then we have 
$$
\begin{array}{ll}
l_1(\wt v^S)&\di :=\bigg[-324\Big(\frac{\wt v^S}{v_m}\Big)^2+72\Big(\frac{\wt v^S}{v_m}\Big)+711\bigg]\Big(\frac{\wt v^S}{v_m}\Big)^6\\[3mm]
&\di =\bigg[-324\bigg(\frac{\wt v^S}{v_m}-\frac 19\Big)^2+715\bigg]\Big(\frac{\wt v^S}{v_m}\Big)^6\geq666\Big(\frac{\wt v^S}{v_m}\Big)^6.
\end{array}
$$
Similarly, 
$$
l_2(\wt v^S):=\bigg[666\Big(\frac{\wt v^S}{v_m}\Big)^2-585\Big(\frac{\wt v^S}{v_m}\Big)+57\bigg]\Big(\frac{\wt v^S}{v_m}\Big)^4\\[3mm]\geq -\frac{21153}{296}\Big(\frac{\wt v^S}{v_m}\Big)^4> -72\Big(\frac{\wt v^S}{v_m}\Big)^4,
$$ 
and
 $$
 l_3(\wt v^S):=1-\frac{\wt v^S}{v_m}-9\Big(\frac{\wt v^S}{v_m}\Big)^2+42\Big(\frac{\wt v^S}{v_m}\Big)^3\geq l_3\Big(\frac{3+\sqrt {23}}{42}v_m\Big)>0.77.
 $$
 Using the above three inequalities,  one has
\begin{equation}\label{4.34-2}
 \begin{aligned}
1&-\frac{\wt v^S}{v_m}-9\Big(\frac{\wt v^S}{v_m}\Big)^2+78\Big(\frac{\wt v^S}{v_m}\Big)^3+57\Big(\frac{\wt v^S}{v_m}\Big)^4-585\Big(\frac{\wt v^S}{v_m}\Big)^5\\
&\underbrace{+711\Big(\frac{\wt v^S}{v_m}\Big)^6+72\Big(\frac{\wt v^S}{v_m}\Big)^7-324\Big(\frac{\wt v^S}{v_m}\Big)^8}_{l_1(\wt v^S)}\\
&\geq1-\frac{\wt v^S}{v_m}-9\Big(\frac{\wt v^S}{v_m}\Big)^2+78\Big(\frac{\wt v^S}{v_m}\Big)^3
\underbrace{+57\Big(\frac{\wt v^S}{v_m}\Big)^4-585\Big(\frac{\wt v^S}{v_m}\Big)^5+666\Big(\frac{\wt v^S}{v_m}\Big)^6}_{l_2(\wt v^S)}\\
&> 1-\frac{\wt v^S}{v_m}-9\Big(\frac{\wt v^S}{v_m}\Big)^2+78\Big(\frac{\wt v^S}{v_m}\Big)^3-72\Big(\frac{\wt v^S}{v_m}\Big)^4\\
&\geq \underbrace{1-\frac{\wt v^S}{v_m}-9\Big(\frac{\wt v^S}{v_m}\Big)^2+42\Big(\frac{\wt v^S}{v_m}\Big)^3}_{l_3(\wt v^S)}> 0.77,
\end{aligned}
\end{equation}
 Substituting \eqref{4.34-2} into  \eqref{4.34-1} and choosing $\theta_*$ suitably small imply that
\begin{equation}\label{449}
 6(\wt v^S)^2K_1+K_2+2K_3+K_4>0.76 v_m^2,\ \  {\rm for}\ \ \wt v^S\in(v_*,0].
\end{equation}
 Combining the estimates in \eqref{4.29}-\eqref{4.33} and using \eqref{K}-\eqref{449}, we conclude that for $0<v_+\ll 1$,
\begin{equation}\label{n4.35}
 \begin{aligned}
H_1+H_2+H_3+H_4\geq\frac{3v_m^2}{8s^2}(v_+-v_*).
\end{aligned}
\end{equation}
Substituting \eqref{4.28.} and \eqref{n4.35} into \eqref{4.27}, we proved \eqref{4-03}.
\end{proof}
 \begin{lemma}\label{L4.4} If $v_+$ satisfies $0<v_+\ll 1$, then $\mathbf{G^S_2}(t)$ defined in \eqref{eq:def_GS2} satisfies that
\begin{equation}\label{4.37}
     \begin{aligned}
\mathbf{G^S_2}(t)\geq&\frac{s}{2}  \bigintsss_{-\infty}^{\xi_*}\left[a\sigma^\prime(\wt v^S)\right]_\xi\left[\phi^{-\mb X}+\frac{a_\xi\sigma^\prime(\wt v^S)}{s\left(a\sigma^\prime(\wt v^S)\right)_\xi}\psi^{-\mb X}\right]^2d\xi\\
&+\frac{7v_+^2}{64s^2}(v_+-v_*) \int^{\xi_*}_{-\infty} \wt v^S_\xi  (\psi^{-\mb X})^2d\xi+ \int^{\xi_*}_{-\infty}\mu a( \psi^{-\mb X}_\xi)^2 d\xi.
 \end{aligned}
\end{equation}
\end{lemma}
\begin{proof}
Since $a^{\prime\prime}=0$ for all $\xi\in(-\infty,\xi_*]$, we have
\begin{equation}\label{4.38}
     \begin{aligned}
\mathbf{G^S_2}(t)=&\frac{s}{2}  \bigintsss_{-\infty}^{\xi_*}\left[a\sigma^\prime(\wt v^S)\right]_\xi\left[\phi^{-\mb X}+\frac{a_\xi\sigma^\prime(\wt v^S)}{s\left(a\sigma^\prime(\wt v^S)\right)_\xi}\psi^{-\mb X}\right]^2d\xi+ \int^{\xi_*}_{-\infty}\mu a( \psi^{-\mb X}_\xi)^2 d\xi\\
&+ \bigintsss_{-\infty}^{\xi_*} {\left[-\frac{\left(a^\prime\sigma^\prime(\wt v^S)\right)^2}{2s\left(a \sigma^\prime(\wt v^S)\right)^\prime}\right]}\wt v^S_\xi(\psi^{-\mb{X}})^2d\xi+ \int^{\xi_*}_{-\infty}\left( \frac{s}{2} a_\xi -\frac{\mu}{2} a_{\xi\xi}\right) (\psi^{-\mb X})^2 d\xi\\
=&\frac{s}{2} \bigintsss_{-\infty}^{\xi_*}\left[a\sigma^\prime(\wt v^S)\right]_\xi\left[\phi^{-\mb X}+\frac{a_\xi\sigma^\prime(\wt v^S)}{s\left(a\sigma^\prime(\wt v^S)\right)_\xi}\psi^{-\mb X}\right]^2d\xi+ \int^{\xi_*}_{-\infty}\mu a( \psi^{-\mb X}_\xi)^2 d\xi\\
&-\frac{1}{2s}\bigintsss_{-\infty}^{\xi_*} \frac{\left[1-3(\wt v^S)^2\right]^2(a^\prime)^2}{a^\prime\sigma^\prime(\wt v^S)+a\sigma^{\prime\prime}(\wt v^S)} \wt v^S_\xi (\psi^{-\mb{X}})^2d\xi\\
&+\frac{1}{2s}\int^{\xi_*}_{-\infty} a^\prime \left[1+3(\wt v^S)^2-2\left(v_m^2+v_+^2+v_mv_+\right) \right]\wt v^S_\xi (\psi^{-\mb{X}})^2d\xi\\
=&\frac{s}{2}  \bigintsss_{-\infty}^{\xi_*}\left[a\sigma^\prime(\wt v^S)\right]_\xi\left[\phi^{-\mb X}+\frac{a_\xi\sigma^\prime(\wt v^S)}{s\left(a\sigma^\prime(\wt v^S)\right)_\xi}\psi^{-\mb X}\right]^2d\xi+ \int^{\xi_*}_{-\infty}\mu a( \psi^{-\mb X}_\xi)^2 d\xi\\
&+\frac{1}{2s}\int^{\xi_*}_{-\infty} \frac{(a^\prime)^2}{a^\prime\sigma^\prime(\wt v^S)+a\sigma^{\prime\prime}(\wt v^S)} L(\wt v^S) \wt v^S_\xi (\psi^{-\mb{X}})^2d\xi,\\
 \end{aligned}
\end{equation}
where % $L(\wt v^S)$ is defined by 
$$
\begin{array}{ll}
\di L(\wt v^S):=-1+6(\wt v^S)^2-9(\wt v^S)^4+\big[1+3(\wt v^S)^2-2(v_m^2+v_+^2+v_mv_+)\big]\big[\sigma^\prime(\wt v^S)+\frac{a}{a^\prime}\sigma^{\prime\prime}(\wt v^S)\big]\\
\di \qquad\quad  =\big(\frac{57}{2}|v_m|-24v_+\big)\wt v^S-2(v_m^2+v_+^2+v_mv_+)+O(v_+^4).
\end{array}
$$
For $\xi\in(-\infty,\xi_*]$, we have $\wt v^S\in(v_m,v_*]$ with $v_m\in[-{v_+}/{2},(-{1}/{2}+\theta_*)v_+)$. 
%Consequently, $\left({v_+}/{2}+v_m\right)\in[0,\theta_*v_+)$. 
Choosing $\theta_*$ suitably small, we can estimate
\begin{equation*}
     \begin{aligned}
L(\wt v^S)\geq& \big(\frac{57}{2}|v_m|-48|v_m|\big)\wt v^S-2(v_m^2+v_+^2+v_mv_+)+O(v_+^4)\\
=&-\frac{39}{2}|v_m|\wt v^S-2(v_m^2+v_+^2+v_mv_+)+O(v_+^4)\\
\geq &-\frac{39}{2}|v_m| v_*-2v_m^2+2|v_m|v_+-2v_+^2+O(v_+^4)\\
=&\frac{31}{4}v_m^2+2|v_m|v_+-2v_+^2+O(v_+^4)\\
>&\frac{31(1-2\theta_*)^2}{16}v_+^2+(1-2\theta_*)v_+^2-2v_+^2+O(v_+^4)\\
=&v_+^2\left[\frac{15}{16}+O(\theta_*)\right]+O(v_+^4)\geq \frac{29}{32}v_+^2, 
  \end{aligned}
  \end{equation*}
  provided that both $v_+>0$ and $\theta_*>0$ are uniformly suitably small, which together with  
 $$
 \frac{a^\prime }{a^\prime\sigma^\prime(\wt v^S)+a\sigma^{\prime\prime}(\wt v^S)}=\frac{1}{1-9(\wt v^S)^2-24v_+\wt v^S-\frac{57}{2}v_m \wt v^S}\geq \frac{28}{29}, \ \ {\rm for} \ \ 0<v_+\ll 1,
 $$
implies the following lower bound:
\begin{equation}\label{4.39}
\frac{1}{2s}\frac{(a^\prime)^2 }{a^\prime\sigma^\prime(\wt v^S)+a\sigma^{\prime\prime}(\wt v^S)}L(\wt v^S)\geq \frac{28}{29}\cdot\frac{a^\prime}{2s} L(\wt v^S)\geq\frac{7v_+^2}{64s^2}(v_+-v_*), \ \ {\rm for} \ \  0<v_+\ll 1.
  \end{equation}
Finally, substituting \eqref{4.39} into \eqref{4.38} yields the desired inequality \eqref{4.37}, which completes the proof of Lemma \ref{L4.4}.
\end{proof}

From Lemma \ref{lem:GS1_lower_bound} and Lemma \ref{L4.4}, we have the following lemma.
 \begin{lemma}\label{L4.5} If $v_+$ satisfies $0<v_+\ll 1$, then the following estimation of $\mathbf{G^S}(t)$ holds true.
\begin{equation}\label{4.40}
     \begin{aligned}
\mathbf{G^S}(t)\geq& \frac{1}{8}\int_{\mathbb R}\mu a( \psi^{-\mb X}_\xi)^2 d\xi+\frac{s}{2}  \bigintsss_{\mathbb R}\left[a\sigma^\prime(\wt v^S)\right]_\xi\left[\phi^{-\mb X}+\frac{a_\xi\sigma^\prime(\wt v^S)}{s\left(a\sigma^\prime(\wt v^S)\right)_\xi}\psi^{-\mb X}\right]^2d\xi\\
&+\frac{3v_m^2}{8s^2} (v_+-v_*)\int_{\xi_*}^{+\infty}\wt v^S_\xi(\psi^{-\mb X})^2 d\xi+\frac{7v_+^2}{64s^2}(v_+-v_*) \int^{\xi_*}_{-\infty}\wt v^S_\xi(\psi^{-\mb X})^2 d\xi.
  \end{aligned}
  \end{equation}
 \end{lemma}
%\noindent {\bf {Estimation of  the $\mathbf{G^S}(t)+ \mathbf{P}(t)$}}. 

\begin{lemma}\label{L4.6} If $v_+$ satisfies $0<v_+\ll 1$, then there exist two constants $\theta_1,\theta_2>0$(independent of $v_+$) such that
\begin{equation}\label{4.41}
     \begin{aligned}
  \mathbf{G^{S,X}}(t)\geq&  \frac{1}{8}\int_{\mathbb R}\mu a^{\mb X}\psi_\xi^2 d\xi+\frac{(v_+-v_*)^2}{48}\left|\mathbf{\dot X}(t)\right|^2\\
  &+\theta_1v_+^3\int_{\mathbb{R}} (\wt v^S)^{\mb X}_\xi  \psi^2 d\xi+\theta_2v_+^3\int_{\mathbb{R}} (\wt v^S)^{\mb X}_\xi  \phi^2 d\xi.
 \end{aligned}
 \end{equation}
    \end{lemma}
    
\begin{proof} 
From the construction of the shift function $\mathbf{X}(t)$ in \eqref{ode.shift2}, we have
\begin{equation*}\label{4.42}
     \begin{aligned}
 \mathbf{\dot X}(t)&  \mathbf{Y_1}(t)= \mathbf{\dot X}(t)\left[\int_{\mathbb{R}}  a\wt u^S_\xi  \psi^{-\mb X}d\xi+\int_{\mathbb{R}} a  \sigma^\prime(\wt v^S) \wt v^S_\xi \phi ^{-\mb X} d\xi\right]\\
 =&\frac{(v_+-v_*)^2}{48}\left|\mathbf{\dot X}(t)\right|^2+\frac{12}{(v_+-v_*)^2}\bigg|\int_{\mathbb{R}}a\wt u^S_\xi \psi^{-\mb X} d\xi+\int_{\mathbb{R}} a \sigma^\prime(\wt v^S)\wt v^S_\xi\phi ^{-\mb X} d\xi\bigg|^2\\
 \geq&\frac{(v_+-v_*)^2}{48}\left|\mathbf{\dot X}(t)\right|^2-\frac{12}{(v_+-v_*)^2}\bigg| \bigintsss_{\mathbb R}a\sigma^\prime(\wt v^S)\wt v^S_\xi\left(\phi^{-\mb X}+\frac{a_\xi\sigma^\prime(\wt v^S)}{s\left(a\sigma^\prime(\wt v^S)\right)_\xi}\psi^{-\mb X}\right)d\xi\bigg|^2\\
 &+\frac{6}{(v_+-v_*)^2}\bigg|\bigintsss_{\mathbb R} \left(a\wt u^S_\xi-\frac{aa_\xi(\sigma^\prime(\wt v^S))^2\wt v^S_\xi}{s(\sigma^\prime(\wt v^S)a)_\xi} \right)\psi^{-\mb X}d\xi\bigg|^2\\
 =&\frac{(v_+-v_*)^2}{48}\left|\mathbf{\dot X}(t)\right|^2-\frac{12}{(v_+-v_*)^2}\bigg|\bigintsss_{\mathbb R}a\sigma^\prime(\wt v^S)\wt v^S_\xi\left(\phi^{-\mb X}+\frac{a_\xi\sigma^\prime(\wt v^S)}{s\left(a\sigma^\prime(\wt v^S)\right)_\xi}\psi^{-\mb X}\right)d\xi\bigg|^2\\
 &+\frac{6}{(v_+-v_*)^2}\bigg|\bigintsss_{\mathbb R}a\underbrace{\left(s+\frac{a^\prime(\sigma^\prime(\wt v^S))^2}{s(\sigma^\prime(\wt v^S)a)^\prime}\right)}_{m(\wt v^S)}\wt v^S_\xi \psi^{-\mb X} d\xi\bigg|^2\\
 \geq&\frac{(v_+-v_*)^2}{48}\left|\mathbf{\dot X}(t)\right|^2-\frac{12}{(v_+-v_*)^2}\bigg| \bigintsss_{\mathbb R}a\sigma^\prime(\wt v^S)\wt v^S_\xi\left(\phi^{-\mb X}+\frac{a_\xi\sigma^\prime(\wt v^S)}{s\left(a\sigma^\prime(\wt v^S)\right)_\xi}\psi^{-\mb X}\right)d\xi\bigg|^2\\
 &-\frac{6}{5(v_+-v_*)^2}\Big|\int_{-\infty}^{\xi_*}am(\wt v^S)\wt v^S_\xi \psi^{-\mb X} d\xi\Big|^2+\frac{1}{(v_+-v_*)^2}\Big|\int^{+\infty}_{\xi_*}a m(\wt v^S) \wt v^S_\xi \psi^{-\mb X} d\xi\Big|^2,
 \end{aligned}
 \end{equation*}
where we have used the inequalities $(c+d)^2\geq -c^2+\frac{d^2}{2}$ and $(c+d)^2\geq -\frac{c^2}{5}+\frac{d^2}{6}$.
Consequently, we obtain from \eqref{4.14}
\begin{equation}\label{4.43}
     \begin{aligned} 
      \mathbf{G^{S,X}}(t)\geq&\frac{(v_+-v_*)^2}{48}\left|\mathbf{\dot X}(t)\right|^2+\frac{1}{8}\int_{\mathbb R}\mu a( \psi^{-\mb X}_\xi)^2 d\xi\\
      &-\frac{6}{5(v_+-v_*)^2}\bigg|\int_{-\infty}^{\xi_*}a  m(\wt v^S) \wt v^S_\xi \psi^{-\mb X}d\xi\bigg|^2+\frac{7v_+^2(v_+-v_*)}{64s^2} \int^{\xi_*}_{-\infty}\wt v^S_\xi (\psi^{-\mb X})^2 d\xi\\
      &+\frac{1}{(v_+-v_*)^2}\bigg|\int^{+\infty}_{\xi_*}a m(\wt v^S) \wt v^S_\xi \psi^{-\mb X} d\xi\bigg|^2\\
      &-\frac{2}{(v_+-v_*)^2}\bigg|\int_{\xi_*}^{+\infty}a\wt v^S_\xi \psi^{-\mb X}  d\xi\bigg|^2+\frac{3v_m^2(v_+-v_*)}{8s^2} \int_{\xi_*}^{+\infty}\wt v^S_\xi  (\psi^{-\mb X})^2d\xi \\
      &+\frac{s}{2}  \bigintsss_{\mathbb R}\left[a\sigma^\prime(\wt v^S)\right]_\xi\bigg|\phi^{-\mb X}+\frac{a_\xi\sigma^\prime(\wt v^S)}{s\left(a\sigma^\prime(\wt v^S)\right)_\xi}\psi^{-\mb X}\bigg|^2d\xi\\
      &-\frac{12}{(v_+-v_*)^2}\left|\bigintsss_{\mathbb R}\left(a\sigma^\prime(\wt v^S)\right)_\xi\Big(\phi^{-\mb X}+\frac{a_\xi\sigma^\prime(\wt v^S)}{s\left(a\sigma^\prime(\wt v^S)\right)_\xi}\psi^{-\mb X}\Big)d\xi\right|^2\\
      =:&\frac{(v_+-v_*)^2}{48}\left|\mathbf{\dot X}(t)\right|^2+\frac{1}{8}\int_{\mathbb R}\mu a( \psi^{-\mb X}_\xi)^2 d\xi+O_1(t)+O_2(t)+O_3(t),
 \end{aligned}
 \end{equation}
with the terms $O_1(t)$, $O_2(t)$, and $O_3(t)$ defined respectively by
\begin{equation*}\label{4.+44}
     \begin{aligned} 
    O_1(t):=&   -\frac{6}{5(v_+-v_*)^2}\left|\int_{-\infty}^{\xi_*}a m(\wt v^S) \wt v^S_\xi \psi^{-\mb X} d\xi\right|^2+\frac{7v_+^2}{64s^2}(v_+-v_*) \int^{\xi_*}_{-\infty}\wt v^S_\xi (\psi^{-\mb X})^2 d\xi,
 \end{aligned}
 \end{equation*}
\begin{equation*}\label{4.+45}
     \begin{aligned} 
    O_2(t):=&\frac{1}{(v_+-v_*)^2}\left|\int^{+\infty}_{\xi_*}a m(\wt v^S) \wt v^S_\xi \psi^{-\mb X} d\xi\right|^2
    -\frac{2}{(v_+-v_*)^2}\left|\int_{\xi_*}^{+\infty}  a\wt v^S_\xi \psi^{-\mb X}d\xi\right|^2\\
      &+\frac{3v_m^2(v_+-v_*)}{8s^2}  \int_{\xi_*}^{+\infty}\wt v^S_\xi (\psi^{-\mb X})^2d\xi,
 \end{aligned}
 \end{equation*}
and 
\begin{equation*}\label{4.+46}
     \begin{aligned} 
    O_3(t):=&\frac{s}{2}  \bigintsss_{\mathbb R}\left[a\sigma^\prime(\wt v^S)\right]_\xi\left|\phi^{-\mb X}+\frac{a_\xi\sigma^\prime(\wt v^S)}{s\left(a\sigma^\prime(\wt v^S)\right)_\xi}\psi^{-\mb X}\right|^2d\xi\\
      &-\frac{12}{(v_+-v_*)^2}\left| \bigintsss_{\mathbb R}a\sigma^\prime(\wt v^S)\wt v^S_\xi\Big(\phi^{-\mb X}+\frac{a_\xi\sigma^\prime(\wt v^S)}{s\left(a\sigma^\prime(\wt v^S)\right)_\xi}\psi^{-\mb X}\Big)d\xi\right|^2.
 \end{aligned}
 \end{equation*}

If $\theta_*$ is suitably small, then H\"older's inequality yields
\begin{equation}\label{4.47}
     \begin{aligned} 
    O_1(t)\geq &\int_{-\infty}^{\xi_*}\left[\frac{7v_+^2}{64s^2}(v_+-v_*) -\frac{6}{5(v_+-v_*)^2}\int_{-\infty}^{\xi_*}a^2\wt v^S_\xi  m^2(\wt v^S) d\xi  \right] \wt v^S_\xi(\psi^{-\mb X})^2d\xi\\
    \geq &\int_{-\infty}^{\xi_*} \left[\frac{7v_+^2}{64s^2}(v_++\frac12|v_m|) -\frac{6\|m\|^2_{L^\infty}}{5(v_+-v_*)^2}\int_{-\infty}^{\xi_*}a^2\wt v^S_\xi  d\xi \right] \wt v^S_\xi(\psi^{-\mb X})^2d\xi\\
     \geq &\int_{-\infty}^{\xi_*} \left[\frac{7v_+^2}{64s^2}\big(v_++\frac12(\frac12-\theta_*)v_+\big) -\frac{151}{1280s^2}(1+\theta_*)v_+^3 \right] \wt v^S_\xi(\psi^{-\mb X})^2d\xi\\
  =&\frac{v_+^3}{s^2}\Big(\frac{3}{160}-\frac{221}{1280}\theta_*\Big)\int_{-\infty}^{\xi_*} \wt v^S_\xi  (\psi^{-\mb X})^2d\xi\\
 \geq &\frac{v_+^3}{80s^2}\int_{-\infty}^{\xi_*} \wt v^S_\xi  (\psi^{-\mb X})^2d\xi, \qquad \text{for } 0<v_+, \theta_*\ll 1.
 \end{aligned}
 \end{equation}
Here we used that 
$$m(\wt v^S)=s+\frac{a^\prime(\sigma^\prime(\wt v^S))^2}{s(a\sigma^\prime(\wt v^S))^\prime}=s+\frac{1-6v^2+9v^4}{s\big[1-(24 v_++\frac{57}{2}v_m)v-9v^2\big]}\to 2,\qquad \text{as } v_+\to 0+,
$$ and
\begin{equation}\label{4.48}
     \begin{aligned} 
    \int_{-\infty}^{\xi_*}a^2\wt v^S_\xi d\xi = &   \int_{v_m}^{v_*}\left[\frac{v_+-v_*}{4s}\wt v^S+\frac{v_+-v_*}{s}\left(v_++\frac{19}{16}v_m\right)  \right]^2d\wt v^S\\
    =&\frac{(v_+-v_*)^2}{s^2}\left(-\frac{1453}{1536}v_m^3-\frac{11}{8}v_m^2v_+-\frac{v_m}{2}v_+^2\right)\\
    \leq&\frac{(v_+-v_*)^2}{s^2}\frac{151}{6144}v_+^3, \qquad \text{for } 0<\theta_*\ll 1.
 \end{aligned}
 \end{equation}

Moreover, we estimate $O_2(t)$ as follows:
\begin{equation}\label{4.49}
     \begin{aligned} 
O_2(t)=&\bigg|\frac{1}{v_+-v_*}\int^{+\infty}_{\xi_*}a m(\wt v^S) \wt v^S_\xi \psi^{-\mb X} d\xi\bigg|^2
    -\frac12\bigg|\frac{2}{v_+-v_*}\int_{\xi_*}^{+\infty}  a\wt v^S_\xi \psi^{-\mb X}d\xi\bigg|^2\\
 \geq&-\left|\frac{1}{v_+-v_*}\int^{+\infty}_{\xi_*}am(\wt v^S)\wt v^S_\xi \psi^{-\mb X} d\xi-\frac{2}{v_+-v_*}\int_{\xi_*}^{+\infty}a\wt v^S_\xi  \psi^{-\mb X} d\xi\right|^2\\
    &+\frac{3v_m^2(v_+-v_*)}{8s^2} \int_{\xi_*}^{+\infty}\wt v^S_\xi  (\psi^{-\mb X})^2d\xi\\
    =&-\bigg| \int^{+\infty}_{\xi_*}a\frac{m(\wt v^S)-2}{v_+-v_*}\wt v^S_\xi \psi^{-\mb X}d\xi\bigg|^2+\frac{3v_m^2(v_+-v_*)}{8s^2} \int_{\xi_*}^{+\infty}\wt v^S_\xi(\psi^{-\mb X})^2 d\xi\\
    \geq& \int^{+\infty}_{\xi_*}\bigg[\frac{3v_m^2(v_+-v_*)}{8s^2}-\frac{1}{(v_+-v_*)^2}  \int_{\xi_*}^{+\infty}a^2\wt v^S_\xi  \left({m(\wt v^S)-2}\right)^2  d\xi \bigg] \wt v^S_\xi (\psi^{-\mb X})^2 d\xi\\
     \geq& \int^{+\infty}_{\xi_*}\bigg[\frac{3v_m^2(v_+-v_*)}{8s^2}-\frac{ \|{m(\wt v^S)-2}\|_{L^\infty}^2}{(v_+-v_*)^2}  \int_{\xi_*}^{+\infty}a^2\wt v^S_\xi d\xi   \bigg] \wt v^S_\xi (\psi^{-\mb X})^2d\xi\\
      \geq&\frac{3v_m^2(v_+-v_*)}{16s^2} \int^{+\infty}_{\xi_*}\wt v^S_\xi (\psi^{-\mb X})^2 d\xi, \quad \text{for } 0<v_+\ll 1,
 \end{aligned}
 \end{equation}
where we have used the inequality $d^2-\frac{c^2}{2}\geq -(d-c)^2$, H\"older's inequality and the fact that $m(\wt v^S)\to 2$ as $v_+\to 0+$.
Next, by using H\"older's inequality again, we have
\begin{equation}\label{4.50}
     \begin{aligned} 
O_3(t) \geq&\left[\frac{s}{2}-\frac{12}{(v_+-v_*)^2} \displaystyle \bigintsss_{\mathbb R}\frac{a^2\Big(\sigma^\prime(\wt v^S)\Big)^2}{\Big|\Big(a\sigma^\prime(\wt v^S)\Big)^\prime\Big|}\wt v^S_\xi d\xi   \right]\\
&\cdot\bigintsss_{\mathbb R}\big|\left(a\sigma^\prime(\wt v^S)\right)_\xi\big|\bigg|\phi^{-\mb X}+\frac{a_\xi\sigma^\prime(\wt v^S)}{s\big(a\sigma^\prime(\wt v^S)\big)_\xi}\psi^{-\mb X}\bigg|^2d\xi\\
\geq&\frac{1}{3}\bigintsss_{\mathbb R}\big|\left(a\sigma^\prime(\wt v^S)\right)_\xi\big|\bigg|\phi^{-\mb X}+\frac{a_\xi\sigma^\prime(\wt v^S)}{s\left(a\sigma^\prime(\wt v^S)\right)_\xi}\psi^{-\mb X}\bigg|^2d\xi, \qquad \text{for } 0<v_+\ll 1.
 \end{aligned}
 \end{equation}
The last inequality holds because $s\to 1$ as $v_+\to 0+$ and $\displaystyle\int_{\mathbb R}\frac{a^2\left(\sigma^\prime(\wt v^S)\right)^2}{\left(a\sigma^\prime(\wt v^S)\right)^\prime}\wt v^S_\xi d\xi=O(v_+^4)$.

Combining the estimates \eqref{4.43}-\eqref{4.50} and using the fact
$$
v_m^2(v_+-v_*)\geq \big(\frac12-\theta_*\big)^2\big(\frac 54-\frac 12\theta_*\big)v_+^3>\frac 3{10} v_+^3, \quad \text{for } 0<\theta_*\ll 1,
$$
%and
%$$
%v_m^2(v_+-v_*)\geq \big(\frac12-\theta_*\big)^2\big(\frac 54-\frac 12\theta_*\big)v_+^3>\frac 3{10} v_+^3, \quad \text{for } 0<\theta_*\ll 1,
%$$
we deduce
\begin{equation}\label{4.51}
     \begin{aligned} 
   \mathbf{G^{S,X}}(t)\geq&\frac{(v_+-v_*)^2}{48}\left|\mathbf{\dot X}(t)\right|^2+\frac{1}{8}\int_{\mathbb R}\mu a( \psi^{-\mb X}_\xi)^2 d\xi+\frac{v_+^3}{80s^2}\int_{\mathbb R}\wt v^S_\xi (\psi^{-\mb X})^2 d\xi\\
   &+\frac{1}{3}\bigintsss_{\mathbb R}\big|\left(a\sigma^\prime(\wt v^S)\right)_\xi\big|\left|\phi^{-\mb X}+\frac{a_\xi\sigma^\prime(\wt v^S)}{s\left(a\sigma^\prime(\wt v^S)\right)_\xi}\psi^{-\mb X}\right|^2d\xi.
 \end{aligned}
 \end{equation}
Note  that by \eqref{w.bounds2} and \eqref{w.deriv2}
$$
\frac{(a\sigma^\prime(\wt v^S))^\prime}{\frac{v_+^3}{160}+\frac{1}{3}\frac{(a^\prime\sigma^\prime(\wt v^S))^2}{(a\sigma^\prime(\wt v^S))^\prime}}=\frac{\big(a(\wt v^S)(1-3(\wt v^S)^2)\big)^\prime}{\frac{v_+^3}{160}+\frac{1}{3}\frac{\big(a^\prime(\wt v^S)(1-3(\wt v^S)^2)\big)^2}{\big(a(\wt v^S)(1-3(\wt v^S)^2)\big)^\prime}}= O(1), \ \ {\rm as}\ \  v_+\rightarrow 0+,
$$
and
\begin{equation*}
     \begin{aligned} 
   &\frac{v_+^3}{80s^2}\int_{\mathbb R}\wt v^S_\xi (\psi^{-\mb X})^2 d\xi+\frac{1}{3}\bigintsss_{\mathbb R}\big|\left(a\sigma^\prime(\wt v^S)\right)_\xi\big|\left|\phi^{-\mb X}+\frac{a_\xi\sigma^\prime(\wt v^S)}{s\left(a\sigma^\prime(\wt v^S)\right)_\xi}\psi^{-\mb X}\right|^2d\xi\\
   &=\frac{v_+^3}{160s^2}\int_{\mathbb R}\wt v^S_\xi (\psi^{-\mb X})^2 d\xi+\bigintsss_{\mathbb R}\wt v^S_\xi\left[\frac{v_+^3}{160s^2}+\frac{1}{3s^2}\frac{(a^\prime\sigma^\prime(\wt v^S))^2}{(a\sigma^\prime(\wt v^S))^\prime}\right]\\
   &\qquad\qquad\qquad\qquad\qquad\qquad\cdot\left[\psi^{-\mb{X}}+\frac{\frac{a^\prime\sigma^\prime(\wt v^S)}{3s}}{\frac{v_+^3}{160s^2}+\frac{1}{3s^2}\frac{(a^\prime\sigma^\prime(\wt v^S))^2}{(a\sigma^\prime(\wt v^S))^\prime}}\phi^{-\mb{X}} \right]^2 d\xi\\
   &\qquad+\frac{1}{3}\bigintss_{\mathbb R}\wt v^S_\xi\left[ (a\sigma^\prime(\wt v^S))^\prime-\frac{\frac{(a^\prime\sigma^\prime(\wt v^S))^2}{3s^2}}{\frac{v_+^3}{160s^2}+\frac{1}{3s^2}\frac{(a^\prime\sigma^\prime(\wt v^S))^2}{(a\sigma^\prime(\wt v^S))^\prime}}\right] (\phi^{-\mb{X}})^2 d\xi\\
   &=\frac{v_+^3}{160s^2}\int_{\mathbb R}\wt v^S_\xi (\psi^{-\mb X})^2 d\xi+\int_{\mathbb R}\wt v^S_\xi\left[\frac{v_+^3}{160s^2}+\frac{1}{3s^2}\frac{(a^\prime\sigma^\prime(\wt v^S))^2}{(a\sigma^\prime(\wt v^S))^\prime}\right]\\
   &\qquad\qquad\qquad\qquad\qquad\qquad\cdot  \left[\psi^{-\mb{X}}+\frac{\frac{a^\prime\sigma^\prime(\wt v^S)}{3s}}{\frac{v_+^3}{160s^2}+\frac{1}{3s^2}\frac{(a^\prime\sigma^\prime(\wt v^S))^2}{(a\sigma^\prime(\wt v^S))^\prime}}\phi^{-\mb{X}} \right]^2 d\xi\\
  &\qquad+\frac{v_+^3}{480}\bigintss_{\mathbb R}\wt v^S_\xi \frac{(a\sigma^\prime(\wt v^S))^\prime}{\frac{v_+^3}{160}+\frac{1}{3}\frac{(a^\prime\sigma^\prime(\wt v^S))^2}{(a\sigma^\prime(\wt v^S))^\prime}}(\phi^{-\mb{X}})^2 d\xi\\
   &\geq\theta_1v_+^3\int_{\mathbb{R}} \wt v^S_\xi  (\psi^{-\mb{X}})^2 d\xi+\theta_2v_+^3\int_{\mathbb{R}} \wt v^S_\xi  (\phi^{-\mb{X}})^2 d\xi,
 \end{aligned}
 \end{equation*}
 for some positive constants $\theta_1$ and $\theta_2$ independent of $v_+\ll1$. Therefore, from \eqref{4.51} we have
 \begin{equation}\label{4.51+}
     \begin{aligned} 
   \mathbf{G^{S,X}}(t)
  \geq&  \frac{1}{8}\int_{\mathbb R}\mu a( \psi^{-\mb X}_\xi)^2 d\xi+\frac{(v_+-v_*)^2}{48}\big|\mathbf{\dot X}(t)\big|^2\\
   &+\theta_1v_+^3\int_{\mathbb{R}} \wt v^S_\xi  (\psi^{-\mb{X}})^2 d\xi+\theta_2v_+^3\int_{\mathbb{R}} \wt v^S_\xi  (\phi^{-\mb{X}})^2 d\xi,
 \end{aligned}
 \end{equation}
Finally, changing variables via $\xi\mapsto \xi+\mathbf{X}(t)$, we can prove \eqref{4.41} and complete the proof of Lemma \ref{L4.6}.
\end{proof}

\vskip0.3cm
\textbf{\underline{Estimation of $\mathbf{G}^{\mb S,\mb X}(t)$ for Case 2}: $v_m \in [-(\frac12 - \theta_*) v_+,0)$}

   We consider the case $v_m \in [(-1/2 + \theta_*) v_+,0)$.  By the construction of weight function in \eqref{w.nondeg}, we can split 
$\mathbf{G^{S,X}}(t)$ as follows:
\begin{equation}\label{g}
     \begin{aligned}
\mathbf{G^{S,X}}(t)=\mathbf{G^S}(t)+\mathbf{G^X}(t)%=\mathbf{G}^S_1(t)+\mathbf{G}^S_2(t)+\mathbf{G^X}(t),
 \end{aligned}
\end{equation}
where 
   \begin{equation}\label{4.14.}
    \begin{aligned}
\mathbf{G^S}(t):=& \int_{\mathbb{R}}\left( \frac{s}{2} a^{\mb X}_\xi -\frac{\mu}{2} a^{\mb X}_{\xi\xi}\right) \psi^2 d\xi+  \int_{\mathbb{R}} \mu a^{\mb X}\psi_\xi^2 d\xi+\int_{\mathbb{R}} a^{\mb X}_\xi\sigma^\prime(( \wt v^S)^{\mb X}) \phi\psi d\xi\\
&+\frac{s}{2}\int_{\mathbb{R}}\left[a^{\mb X}\sigma^\prime((\wt v^S)^{\mb X})\right]_\xi \phi^2 d\xi
+\frac{2}{\delta_S^2}\left[\int_{\mathbb{R}}a^{\mb X}(\wt v^S)^{\mb X}_\xi \psi d\xi\right]^2,
 \end{aligned}
   \end{equation}
and
      \begin{equation}\label{4.2+.}
     \begin{aligned}
\mathbf{G^X}(t):= \mathbf{\dot X}(t)\mathbf{Y_1}(t)-\frac{2}{\delta_S^2}\left[\int_{\mathbb{R}}a^{\mb X}(\wt v^S)^{\mb X}_\xi\psi  d\xi\right]^2.
 \end{aligned}
   \end{equation}
%\noindent {\bf Estimation of  the $\mathbf{G}^S(t)$}.  

\begin{lemma}\label{L4.3.}
The term $\mathbf{G^S}(t)$ defined in \eqref{4.14.} satisfies the following lower bound:
\begin{equation}\label{4.21.}
     \begin{aligned}
\mathbf{G^S}(t)\geq&{\mu} \int_{\mathbb{R}}(1-\alpha)a^{\mb{X}} \psi_\xi^2d\xi+\int_{\mathbb{R}}  \left|h((\wt v^S)^{\mb{X}})\right| (\wt v^S)^{\mb{X}}_\xi\psi^2 d\xi\\
& +\frac{s}{2}\bigintsss_{\mathbb{R}}\left[a^{\mb{X}} \sigma^\prime((\wt v^S)^{\mb{X}})\right]_\xi \left[\phi+\frac{a^{\mb{X}}_\xi \sigma^\prime((\wt v^S)^{\mb{X}})}{s(a^{\mb{X}} \sigma^\prime((\wt v ^S)^{\mb{X}}))_\xi }\psi\right]^2 d\xi,
 \end{aligned}
\end{equation}
where the coefficient function $h((\wt v^S)^{\mb{X}})$ is given by
\begin{equation}\label{4.22.}
  \begin{aligned}
h((\wt v^S)^{\mb{X}}):=&-\frac{\left[(a^\mb{X})^\prime\sigma^\prime((\wt v^S)^{\mb{X}})\right]^2}{2s \left[a^{\mb{X}} \sigma^\prime((\wt v^S)^{\mb{X}})\right]^\prime}+\frac{\alpha \delta_S}{2s^2}\left[1+3((\wt v ^S)^{\mb{X}}))^2-2v_+^2-2v_+v_m-2v_m^2\right]\\
&+\frac{\alpha^2 \delta_S}{s^2}\left[3(v_++v_m)(\wt v ^S)^{\mb{X}}+3(v_++v_m)^2\right]\\
\geq& \frac{\alpha \delta_S(3+2\theta_*)\theta_*v_+^2}{4s^2}>0, \qquad  \text{\rm for }\  0<v_+\ll 1.
 \end{aligned}
\end{equation}
\end{lemma}

\begin{proof}
Change the variable 
$$
y:=\frac{v_+-(\wt v^S)^{\mb{X}}}{v_+-v_m},
$$ 
which maps $\xi\in(-\infty,+\infty)$ to $y\in[0,1]$. Since $y_\xi=-\frac{(\wt v^S)^{\mb{X}}_\xi}{v_+-v_m}<0$, the mapping $\xi\rightarrow y$ is always invertible.
For any fixed $t>0$, we define the function 
\begin{equation}
\label{wp}
\Psi(t,y):=a^{\mb{X}}(\wt v^S(\xi))\psi(t,\xi).
\end{equation}

To apply the weighted Poincar\'e inequality to $\Psi$, we first compute the following two quantities:
\begin{equation}\label{4.23.}
 \begin{aligned}
2\Big(\int_0^1\Psi(t,y) dy\Big)^2%=2\Big(\int_{\mathbb{R}} \psi(t,\xi) a^{\mb{X}}\frac{(\wt v^S)^{\mb{X}}_\xi}{v_+-v_m}d\xi\Big)^2
=\frac{2}{\delta_S^2}\left[\int_{\mathbb{R}}a^{\mb X}(\wt v^S)^{\mb X}_\xi\psi  d\xi\right]^2,
 \end{aligned}
\end{equation}
and
\begin{equation}\label{4.24.}
 \begin{aligned}
2\int_0^1\Psi^2(t,y)dy&=2\int_{\mathbb{R}}   (a^{\mb X})^2\psi^2\frac{(\wt v^S)^{\mb X}_\xi}{v_+-v_m} d\xi
=\frac{2}{\delta_S}\int_{\mathbb{R}}  (a^{\mb X})^2(\wt v^S)^{\mb X}_\xi \psi^2  d\xi\\[1mm]
&=\frac{2\alpha^2\delta_S}{s^2}\int_{\mathbb{R}}    \left[((\wt v ^S)^{\mb{X}})^2+2(v_++v_m)(\wt v ^S)^{\mb{X}}+(v_++v_m)^2\right] (\wt v^S)^{\mb X}_\xi \psi^2 d\xi.
 \end{aligned}
\end{equation}
And the weighted derivitive term becomes
\begin{equation}\label{4.25.}
 \begin{aligned}
&\int_0^1  |\Psi_y(t,y)|^2y(1-y)dy\\
 &=-\int_{\mathbb{R}}\Big|\big[a^{\mb{X}}\psi(t,\xi) \big]_\xi\Big|^2\frac{1}{y_\xi}y(1-y)d\xi \\
&=\frac{1}{\delta_S}\bigintsss_{\mathbb{R}}\Big|\big[a^{\mb{X}}\psi(t,\xi) \big]_\xi\Big|^2\frac{[v_+-(\wt v^S)^{\mb X}][(\wt v^S)^{\mb X}-v_m]}{(\wt v^S)^{\mb X}_\xi}d\xi\\
&=\frac{\mu s}{\delta_S}\bigintsss_{\mathbb{R}}\Big|\big[a^{\mb{X}}\psi(t,\xi) \big]_\xi\Big|^2\frac{1}{(\wt v^S)^{\mb X}+v_m +v_+}d\xi   \\[4mm]
&=\frac{\mu s}{\delta_S}\bigintsss_{\mathbb{R}}\frac{( a^{\mb{X}})^2}{(\wt v^S)^{\mb X}+v_m+v_+ }\psi_\xi^2d\xi+\frac{\mu s}{\delta_S}\bigintsss_{\mathbb{R}}\frac{( a^{\mb{X}}_\xi)^2}{(\wt v^S)^{\mb X}+v_m+v_+}\psi^2d\xi\\[1mm]
&\quad +\frac{2\mu s}{\delta_S}\bigintsss_{\mathbb{R}} \frac{a^{\mb{X}}_\xi a^{\mb{X}}}{(\wt v^S)^{\mb X}+v_m+v_+ }\psi\psi_\xi d\xi\\
&:=I_1+I_2+I_3.
 \end{aligned}
\end{equation}
Note that
\begin{equation}\label{4.26.}
 \begin{aligned}
 I_1=\frac{\mu s}{\delta_S}\bigintsss_{\mathbb{R}}\frac{( a^{\mb{X}})^2}{(\wt v^S)^{\mb X}+v_m+v_+ } \psi_\xi^2d\xi =\mu\alpha\int_{\mathbb{R}} a^{\mb{X}}\psi_\xi^2d\xi,
 \end{aligned}
\end{equation}
and
\begin{equation}\label{4.27.}
 \begin{aligned}
I_3&=\frac{\mu s}{\delta_S}\bigintsss_{\mathbb{R}} \frac{a^{\mb{X}}_\xi a^{\mb{X}}}{(\wt v^S)^{\mb X}+v_m+v_+ }(\psi^2)_\xi d\xi \\
&=-\frac{\mu s}{\delta_S}\bigintsss_{\mathbb{R}}\left[\frac{a^{\mb{X}}_\xi a^{\mb{X}}}{(\wt v^S)^{\mb X}+v_m+v_+ }\right]_\xi \psi^2d\xi\\
&=-I_2+\frac{\mu s}{\delta_S}\bigintsss_{\mathbb{R}}a^{\mb{X}}\left[ \frac{a^{\mb{X}}_{\xi}(\wt v^S)^{\mb X}_\xi}{(\wt v^S)^{\mb X}+v_m+v_+ }-a^{\mb{X}}_{\xi\xi}\right]\frac{\psi^2}{(\wt v^S)^{\mb X}+v_m+v_+ }d\xi\\
&=-I_2+\frac{\alpha^2\delta_S}{s^2} \int_{\mathbb{R}}  \big[2((\wt v ^S)^{\mb{X}})^2+(v_++v_m)(\wt v ^S)^{\mb{X}}-(v_++v_m)^2\big](\wt v^S)^{\mb X}_\xi \psi^2d\xi.
\end{aligned}
\end{equation}
Then
\begin{equation}\label{4.25+}
 \begin{aligned}
&\int_0^1  |\Psi_y(t,y)|^2y(1-y)dy=I_1+I_2+I_3\\
&=\mu\alpha\int_{\mathbb{R}} a^{\mb{X}}\psi_\xi^2d\xi
+ \frac{\alpha^2\delta_S}{s^2} \int_{\mathbb{R}}  \big[2((\wt v ^S)^{\mb{X}})^2+(v_++v_m)(\wt v ^S)^{\mb{X}}-(v_++v_m)^2\big](\wt v^S)^{\mb X}_\xi \psi^2d\xi.
 \end{aligned}
\end{equation}
Now, applying the weighted Poincar\'e inequality \eqref{poincare} to $\Psi$ and using the expressions \eqref{4.23.}, \eqref{4.24.} and \eqref{4.25+}, we obtain
\begin{equation}\label{4.28}
 \begin{aligned}
\frac{2}{\delta_S^2}\left[\int_{\mathbb{R}}a^{\mb X}(\wt v^S)^{\mb X}_\xi \psi d\xi\right]^2& \geq
  -\mu\alpha\int_{\mathbb{R}} a^{\mb{X}} \psi_\xi^2 d\xi\\
  &\quad+ \frac{\alpha^2\delta_S}{s^2}
\int_{\mathbb{R}}\big[3(v_++v_m)(\wt v ^S)^{\mb{X}}
+3(v_++v_m)^2\big](\wt v^S)^{\mb X}_\xi \psi^2 d\xi.
\end{aligned}
\end{equation}
Now we rewrite the other terms in $\mathbf{G^S}(t)$. First,
\begin{equation}\label{4.29-.}
 \begin{aligned}
& \int_{\mathbb{R}}\left( \frac{s}{2} a^{\mb X}_\xi -\frac{\mu}{2} a^{\mb X}_{\xi\xi}\right) \psi^2 d\xi\\
 &= \int_{\mathbb{R}} \left[\frac{s}{2}(a^{\mb X})^\prime(\wt v^S)^{\mb X}_\xi-\frac{\mu}{2}(a^{\mb X})^\prime(\wt v^S)^{\mb X}_{\xi\xi}\right]\psi^2 d\xi\\
 &=\frac{\alpha \delta_S}{2s^2}   \int_{\mathbb{R}}\left[1-2v_+^2-2v_+v_m-2v_m^2+3((\wt v^S)^{\mb X})^2\right](\wt v^S)^{\mb X}_\xi \psi^2 d\xi.
\end{aligned}
\end{equation}
Then
\begin{equation}\label{4.29.}
 \begin{aligned}
&\int_{\mathbb{R}}a^{\mb X}_\xi \sigma^\prime((\wt v^S)^{\mb X}) \phi\psi d\xi+\frac{s}{2}\int_{\mathbb{R}}\left[a^{\mb X}\sigma^\prime((\wt v^S)^{\mb X})\right]_\xi \phi^2 d\xi\\
&=\frac{s}{2}\bigintsss_{\mathbb{R}}[a^{\mb{X}} \sigma^\prime((\wt v^S)^{\mb{X}})]_\xi \left[\phi+\frac{a^{\mb{X}}_\xi\sigma^\prime((\wt v^S)^{\mb{X}})}{s(a^{\mb{X}} \sigma^\prime((\wt v ^S)^{\mb{X}}))_\xi }\psi\right]^2 d\xi\\
&\quad-\bigintsss_{\mathbb{R}}  \frac{[(a^\mb{X})^\prime\sigma^\prime((\wt v^S)^{\mb{X}})]^2}{2s [a^{\mb{X}} \sigma^\prime((\wt v^S)^{\mb{X}}) ]^\prime} (\wt v^S)^{\mb X}_\xi \psi^2d\xi.
\end{aligned}
\end{equation}

Combining the identities \eqref{4.23.}-\eqref{4.29.} with the definition of $\mathbf{G^S}(t)$ in \eqref{4.14.} yields the desired lower bound \eqref{4.21.}.

It remains to verify \eqref{4.22.} for $0<v_+\ll1$. First, we have
\begin{equation*}\label{4.30.}
 \begin{aligned}
h((\wt v^S)^{\mb{X}})
=&\frac{\alpha \delta_S}{2s^2[1-9((\wt v^S)^{\mb{X}})^2-6(v_++v_m)(\wt v^S)^{\mb{X}}]}\Big[-\left(1-3((\wt v^S)^{\mb{X}})^2\right)^2\\[1mm]
&+\big(1+3((\wt v^S)^{\mb{X}})^2 +6\alpha(v_++v_m)(\wt v^S)^{\mb{X}}
+(12\alpha-2)v_+v_m+(6\alpha-2)(v_+^2+v_m^2)\big)\\[1mm]
&\cdot \left(1-9((\wt v^S)^{\mb{X}})^2-6(v_++v_m)(\wt v^S)^{\mb{X}}\right)\Big]\\
=&\frac{\alpha \delta_S}{2s^2[1-9((\wt v^S)^{\mb{X}})^2-6(v_++v_m)(\wt v^S)^{\mb{X}}]}\\
&\cdot\underbrace{[(6\alpha-6)(v_++v_m)(\wt v^S)^{\mb{X}}+(12\alpha-2)v_+v_m+(6\alpha-2)(v_+^2+v_m^2)+O(v_+^4)  ]}_{\bar h ((\wt v^S)^{\mb{X}})}.
\end{aligned}
\end{equation*}
Since  $v_m \in [(-1/2 + \theta_*) v_+,0)$ and $(\wt v^S)^{\mb{X}}\in(v_m,v_+)$, we have
\begin{equation*}\label{4.31.}
 \begin{aligned}
\bar h ((\wt v^S)^{\mb{X}})&\geq (6\alpha-6)(v_++v_m)v_++(12\alpha-2)v_+v_m+(6\alpha-2)(v_+^2+v_m^2)+O(v_+^4)  \\
&=(12\alpha-8)v_+^2+(18\alpha-8)v_mv_++(6\alpha-2)v_m^2+O(v_+^4) \\
&\geq (12\alpha-8)v_+^2+(18\alpha-8)\Big(-\frac{1}{2}+\theta_*\Big)v_+^2+(6\alpha-2)\Big(-\frac{1}{2}+\theta_*\Big)^2v_+^2+O(v_+^4) \\
&>(3\theta_*+2\theta_*^2)v_+^2+O(v_+^4)\geq \frac{(3+2\theta_*)\theta_*}{2} v_+^2, \qquad  \text{for } 0<v_+\ll 1.
\end{aligned}
\end{equation*}
%$$
%\textcolor{red}{condition: Cv_+^2\leq \frac{(3+2\theta_*)}2 \theta_*. \ which \ means \ that\  v_+\  depends\  on\  \theta_*? }  
%$$
%\textcolor{blue}{Yes, since the choice of $\theta_*$ is given by the proof of case 1 and is independent of $v_+$, here we can choose $v_+$ smaller than before. and I alwasys denote as $0<v_+\ll1$.}\\
which  verifies \eqref{4.22.}. % for $0<v_+\ll1$.
On the other hand, since $\frac{2}{3}<\alpha<1$ and $|v_m|\leq\frac{1}{2}v_+$,
 it holds that
\begin{equation*}\label{4.30.}
 \begin{aligned}
\left|\bar h ((\wt v^S)^{\mb{X}})\right|%&=\left|[(6\alpha-6)(v_++v_m)(\wt v^S)^{\mb{X}}+(12\alpha-2)v_+v_m+(6\alpha-2)(v_+^2+v_m^2)+O(v_+^4)  ]\right|\\
&\leq\left|   6-6\alpha\right|\left|v_++v_m\right|\left|v_+\right|+(12\alpha-2)\left|v_+\right|\left|v_m\right|+(6\alpha-2)(v_+^2+v_m^2)+C v_+^4\\
&\leq2v_+^2+10|v_+||v_m|+4(v_+^2+v_m^2)+Cv_+^4\\
&=12v_+^2+Cv_+^4\leq{13}v_+^2,  \qquad  \text{for } 0<v_+\ll 1,
\end{aligned}
\end{equation*}
which implies that
\begin{equation}\label{4.66.-}
h((\wt v^S)^{\mb{X}})\leq \frac{\alpha \delta_S}{2s^2[1-9((\wt v^S)^{\mb{X}})^2-6(v_++v_m)(\wt v^S)^{\mb{X}}]}13v_+^2\leq \frac{{13}\alpha \delta_S}{s^2}v_+^2.
\end{equation}

\end{proof}

%\noindent {\bf Estimation of  the $\mathbf{P}(t)$}. 
 \begin{lemma}\label{L4.4.} $\mathbf{G^X}(t)$ defined in \eqref{4.2+.} satisfies that
\begin{equation}\label{4.32.}
     \begin{aligned}
     \mathbf{G^X}(t)\geq \frac{\delta_S^2}{32}\left|\mathbf{\dot X}(t)\right|^2&-\frac{8}{\delta_S^2}\left[\bigintss_{\mathbb{R}}a^{\mb{X}} \sigma^\prime((\wt v^S)^{\mb{X}})(\wt v ^S)^{\mb{X}}_\xi \left(\phi+\frac{a^{\mb{X}}_\xi \sigma^\prime((\wt v^S)^{\mb{X}}) }{s(\sigma^\prime((\wt v^S)^{\mb{X}}) a^{\mb{X}})_\xi }\psi \right)d\xi\right]^2\\
     &+  \frac{4}{\delta_S^2}\left(\int_{\mathbb{R}} a^{\mb{X}}g((\wt v^S)^{\mb{X}}) (\wt v^S)^{\mb{X}}_\xi \psi d\xi\right)^2-\frac{2}{\delta_S^2}\left(\int_{\mathbb{R}}a^{\mb X}(\wt v^S)^{\mb X}_\xi \psi d\xi\right)^2,
 \end{aligned}
 \end{equation}
 where 
\begin{equation}\label{4.33.}
g((\wt v^S)^{\mb{X}}) :=s+\frac{(a^{\mb{X}})^\prime [\sigma^\prime((\wt v^S)^{\mb{X}})]^2}{s( a^{\mb{X}}\sigma^\prime((\wt v^S)^{\mb{X}}))^\prime }.
\end{equation}
  \end{lemma}
\begin{proof}  From \eqref{ode.shift1} and  \eqref{4.2+.}, we have
\begin{equation}\label{4.34.}
     \begin{aligned}
         \mathbf{G^X}(t)=&\frac{\delta_S^2}{16}\left|\mathbf{\dot X}(t)\right|^2-\frac{2}{\delta_S^2}\left(\int_{\mathbb{R}} a^{\mb X}(\wt v^S)^{\mb X}_\xi \psi d\xi\right)^2\\
         =&\frac{\delta_S^2}{32}\left|\mathbf{\dot X}(t)\right|^2+\frac{8}{\delta_S^2}|\mb{Y}_1(t)|^2-\frac{2}{\delta_S^2}\left(\int_{\mathbb{R}} a^{\mb X}(\wt v^S)^{\mb X}_\xi \psi d\xi\right)^2\\
    % &+\frac{8}{\delta_S^2}\left[ \int_{\mathbb{R}}a^{\mb X} \sigma^\prime((\wt v^S)^{\mb X})(\wt v^S)^{\mb X}_\xi \phi d\xi+\int_{\mathbb{R}}a^{\mb X} (\wt u^S)^{\mb X}_\xi \psi d\xi \right]^2\\
     =&\frac{\delta_S^2}{32}|\mathbf{\dot X}(t)|^2-\frac{2}{\delta_S^2}\left(\int_{\mathbb{R}}a^{\mb X}(\wt v^S)^{\mb X}_\xi  \psi d\xi\right)^2\\
     &+\frac{8}{\delta_S^2}\bigg|\bigintsss_{\mathbb{R}}a^{\mb{X}} \sigma^\prime((\wt v^S)^{\mb{X}})(\wt v ^S)^{\mb{X}}_\xi \Big(\phi+\frac{a^{\mb{X}}_\xi \sigma^\prime((\wt v^S)^{\mb{X}}) }{s(\sigma^\prime((\wt v^S)^{\mb{X}}) a^{\mb{X}})_\xi }\psi \Big)d\xi\\
     &\qquad\quad +\bigintsss_{\mathbb{R}}a^{\mb{X}}\Big((\wt u^S)^{\mb X}_\xi-\frac{a^{\mb{X}}_\xi (\sigma^\prime((\wt v^S)^{\mb{X}}))^2}{s(a^{\mb{X}}\sigma^\prime((\wt v^S)^{\mb{X}}) )^\prime }\Big)\psi d\xi\bigg|^2\\
     \geq&\frac{\delta_S^2}{32}\left|\mathbf{\dot X}(t)\right|^2\underbrace{-\frac{2}{\delta_S^2}\Big(\int_{\mathbb{R}} a^{\mb X}(\wt v^S)^{\mb X}_\xi  \psi d\xi\Big)^2+  \frac{4}{\delta_S^2}\left(\int_{\mathbb{R}} a^{\mb{X}}g((\wt v^S)^{\mb{X}}) (\wt v^S)^{\mb{X}}_\xi \psi d\xi\right)^2}_{   {P_1}(t)}\\
     &\underbrace{-\frac{8}{\delta_S^2}\left[  \bigintsss_{\mathbb{R}}a^{\mb{X}} \sigma^\prime((\wt v^S)^{\mb{X}})(\wt v ^S)^{\mb{X}}_\xi \bigg(\phi+\frac{a^{\mb{X}}_\xi \sigma^\prime((\wt v^S)^{\mb{X}}) }{s(\sigma^\prime((\wt v^S)^{\mb{X}}) a^{\mb{X}})_\xi }\psi \bigg)d\xi   \right]^2}_{   {P_2}(t)},
 \end{aligned}
 \end{equation}
   where in the last inequality we have used $(c+d)^2\geq \frac{d^2}{2}-c^2$ and $(\wt u^S)^{\mb X}_\xi=-s(\wt v^S)^{\mb X}_\xi$. The proof of Lemma \ref{L4.4.} is completed.
\end{proof}

%\textcolor{red}{HERE}\noindent {\bf {Estimation of  the $\mathbf{G^S}(t)+ \mathbf{P}(t)$}}. 
 \begin{lemma}\label{L4.5.} If $v_+$ satisfies $0<v_+\ll 1$, then it holds that
\begin{equation}\label{4.35.}
     \begin{aligned}
  \mathbf{G^{S,X}}(t)\geq&{\mu} \int_{\mathbb{R}}(1-\alpha)a^{\mb{X}} \psi_\xi^2d\xi+\frac{1}{4}\int_{\mathbb{R}} \left|h((\wt v^S)^{\mb{X}})\right| (\wt v^S)^{\mb{X}}_\xi  \psi^2 d\xi\\
  &+\frac{1}{4}\int_{\mathbb{R}} \left|l((\wt v^S)^{\mb{X}})\right|  (\wt v^S)^{\mb{X}}_\xi \phi^2 d\xi+\frac{\delta_S^2}{32}\left|\mathbf{\dot X}(t)\right|^2,
 \end{aligned}
 \end{equation}
 where $l((\wt v^S)^{\mb{X}}) $ is given by
 \begin{equation}\label{4.36.}
  \begin{aligned}
l((\wt v^S)^{\mb{X}}) :=&\frac{s^2h((\wt v^S)^{\mb{X}})[(a^{\mb{X}} \sigma^\prime((\wt v^S)^{\mb{X}}))^\prime]^2}{ [(a^{\mb{X}})^\prime \sigma^\prime((\wt v^S)^{\mb{X}})]^2+s^2[ a^{\mb{X}}\sigma^\prime((\wt v^S)^{\mb{X}})]^\prime h((\wt v^S)^{\mb{X}}) }\\[1mm]
\geq& \frac{\alpha \delta_S(3+2\theta_*)\theta_*v_+^2}{16(1+{13}sv_+^2)}>0, \qquad  for \ \ 0<v_+\ll 1,
 \end{aligned}
\end{equation}
with $h((\wt v^S)^{\mb{X}})$ being defined in \eqref{4.22.}.
\end{lemma}
\begin{proof} To control ${P_1}(t)$ and ${P_2}(t)$ defined in \eqref{4.34.}, using the inequality $d^2-\frac{c^2}{2}\geq-(d-c)^2$ with 
\begin{equation*}
\displaystyle c:=\frac{2}{\delta_S} \int_{\mathbb{R}}a^{\mb{X}}(\wt v^S)^{\mb{X}}_\xi  \psi  d\xi,\qquad d:=\displaystyle\frac{1}{\delta_S}  \int_{\mathbb{R}}  a^{\mb{X}}(\wt v^S)^{\mb{X}}_\xi {g((\wt v^S)^{\mb{X}})}\psi d\xi,
\end{equation*}
we can obtain 
\begin{equation}\label{4.37.}
     \begin{aligned}
{P_1}(t)&\geq \frac{1}{\delta_S^2}  \left(\int_{\mathbb{R}} a^{\mb{X}}{g((\wt v^S)^{\mb{X}})} (\wt v^S)^{\mb{X}}_\xi   \psi d\xi\right)^2-\frac{2}{\delta_S^2} \left( \int_{\mathbb{R}} a^{\mb{X}}(\wt v^S)^{\mb{X}}_\xi\psi  d\xi \right)^2\\
&\geq-\frac{4}{\delta_S^2}  \bigg(\bigintsss_{\mathbb{R}}  a^{\mb{X}}\bigg[ \frac{g((\wt v^S)^{\mb{X}})}{2}-1\bigg](\wt v^S)^{\mb{X}}_\xi \psi d\xi\bigg)^2\\
&\geq -\frac{4}{\delta_S^2} \int_{\mathbb{R}} h((\wt v^S)^{\mb{X}}) (\wt v^S)^{\mb{X}}_\xi   \psi^2 d\xi\bigintsss_{\mathbb{R}}\frac{ (a^{\mb{X}})^2(\wt v^S)^{\mb{X}}_\xi}{h((\wt v^S)^{\mb{X}})}\left( \frac{g((\wt v^S)^{\mb{X}})}{2}-1\right)^2d\xi\\
&\geq -\frac{4}{\delta_S} \bigintsss_{\mathbb{R}} h((\wt v^S)^{\mb{X}})  (\wt v^S)^{\mb{X}}_\xi \psi^2d\xi\left\|a\right\|_{L^\infty(\mathbb{R})}^2\left\|\frac{1}{ h((\wt v^S)^{\mb{X}})}\right\|_{L^\infty(\mathbb{R})}\left\| \frac{g((\wt v^S)^{\mb{X}})}{2}-1\right\|^2_{L^\infty(\mathbb{R})}\\
&\geq -\frac{1}{2} \int_{\mathbb{R}}  h((\wt v^S)^{\mb{X}})  (\wt v^S)^{\mb{X}}_\xi  \psi^2 d\xi, \qquad  {\rm for} \ \ 0<v_+\ll 1,
 \end{aligned}
 \end{equation}
where H$\ddot{\rm o}$lder inequality is used in the third inequality, and the facts that as $v_+\rightarrow 0+$,
$$\left\|a\right\|_{L^\infty(\mathbb{R})}^2\left\|\frac{1}{ h((\wt v^S)^{\mb{X}})}\right\|_{L^\infty(\mathbb{R})}=O(\delta_S),\quad \displaystyle \left\| \frac{g((\wt v^S)^{\mb{X}})}{2}-1\right\|_{L^\infty(\mathbb{R})}=o(1).$$
are used in the last inequality.
Since as $v_+\rightarrow 0+$
$$\displaystyle\left|-  \frac{8}{\delta_S^2}\displaystyle \bigintsss_{\mathbb{R} }
\frac{\left[a^{\mb{X}}\sigma^\prime((\wt v^S)^{\mb{X}})(\wt v ^S)^{\mb{X}}_\xi\right]^2 }{{[a^{\mb{X}}\sigma^\prime((\wt v^S)^{\mb{X}})]^\prime(\wt v ^S)^{\mb{X}}_\xi }} d \xi\right|=O(v_+^2),\qquad\displaystyle\frac{s}{2}-\frac{1}{2}=o(1),$$
we can estimate
\begin{equation}\label{4.38.}
     \begin{aligned}
{P_2}(t)+&  \frac{s}{2}\bigintsss_{\mathbb{R}}[a^{\mb{X}} \sigma^\prime((\wt v^S)^{\mb{X}})]_\xi \left[\phi+\frac{a^{\mb{X}}_\xi \sigma^\prime((\wt v^S)^{\mb{X}})}{s(a^{\mb{X}} \sigma^\prime((\wt v ^S)^{\mb{X}}))_\xi }\psi\right]^2 d\xi\\
%=&-\frac{8}{\delta_S^2}\left(  \int_{\mathbb{R}} \sqrt{[w^{\mb{X}}\sigma^\prime((\wt v^S)^{\mb{X}})]^\prime(\wt v ^S)^{\mb{X}}_\xi }
%\frac{w^{\mb{X}}\sigma^\prime((\wt v^S)^{\mb{X}})(\wt v ^S)^{\mb{X}}_\xi }{\sqrt{[w^{\mb{X}}\sigma^\prime((\wt v^S)^{\mb{X}})]^\prime(\wt v ^S)^{\mb{X}}_\xi }}\left[\phi+\frac{w^{\mb{X}}_\xi \sigma^\prime((\wt v^S)^{\mb{X}}) }{s(\sigma^\prime((\wt v^S)^{\mb{X}}) w^{\mb{X}})_\xi }\psi \right]d\xi   \right)^2\\
%&+  \frac{s}{2}\int_{\mathbb{R}}[w^{\mb{X}} \sigma^\prime((\wt v^S)^{\mb{X}})]^\prime(\wt v^S)^{\mb{X}}_\xi  \left [\phi+\frac{w^{\mb{X}}_\xi \sigma^\prime((\wt v^S)^{\mb{X}})}{s(w^{\mb{X}} \sigma^\prime((\wt v ^S)^{\mb{X}}))_\xi }\psi\right]^2 d\xi\\
\geq& \displaystyle\left[\frac{s}{2}-  \frac{8}{\delta_S^2}\displaystyle \bigintsss_{\mathbb{R}}
\frac{\left(a^{\mb{X}}\sigma^\prime((\wt v^S)^{\mb{X}})(\wt v ^S)^{\mb{X}}_\xi\right)^2 }{{[a^{\mb{X}}\sigma^\prime((\wt v^S)^{\mb{X}})]^\prime(\wt v ^S)^{\mb{X}}_\xi }} d \xi \right]  \\
&\cdot\bigintsss_{\mathbb{R}}[a^{\mb{X}} \sigma^\prime((\wt v^S)^{\mb{X}})]^\prime(\wt v^S)^{\mb{X}}_\xi\left [\phi+\frac{a^{\mb{X}}_\xi \sigma^\prime((\wt v^S)^{\mb{X}}) }{s(\sigma^\prime((\wt v^S)^{\mb{X}}) a^{\mb{X}})_\xi }\psi \right]^2d\xi\\
\geq& \frac{1}{4} \bigintsss_{\mathbb{R}}[a^{\mb{X}} \sigma^\prime((\wt v^S)^{\mb{X}})]^\prime(\wt v^S)^{\mb{X}}_\xi\left [\phi+\frac{a^{\mb{X}}_\xi \sigma^\prime((\wt v^S)^{\mb{X}}) }{s(\sigma^\prime((\wt v^S)^{\mb{X}}) a^{\mb{X}})_\xi }\psi \right]^2d\xi, \quad  {\rm for} \ \ 0<v_+\ll 1,
 \end{aligned}
 \end{equation}
where in the first inequality we have used the H$\ddot{\rm o}$lder inequality.

Hence, we conclude from \eqref{4.21.}, \eqref{4.32.}, \eqref{4.37.} and \eqref{4.38.} that
 \begin{equation}\label{4.39.}
     \begin{aligned}
  \mathbf{G^{S,X}}(t)\geq&{\mu} \int_{\mathbb{R}}(1-\alpha)a^{\mb{X}} \psi_\xi^2d\xi+\frac{1}{2}\int_{\mathbb{R}} h((\wt v^S)^{\mb{X}}) (\wt v^S)^{\mb{X}}_\xi  \psi^2 d\xi+\frac{\delta_S^2}{32}\left|\mathbf{\dot X}(t)\right|^2\\
  &+\frac{1}{4} \bigintsss_{\mathbb{R}}[a^{\mb{X}} \sigma^\prime((\wt v^S)^{\mb{X}})]^\prime(\wt v^S)^{\mb{X}}_\xi\left [\phi+\frac{a^{\mb{X}}_\xi \sigma^\prime((\wt v^S)^{\mb{X}}) }{s(\sigma^\prime((\wt v^S)^{\mb{X}}) a^{\mb{X}})_\xi }\psi \right]^2d\xi.
 \end{aligned}
 \end{equation}
We can compute that
 \begin{equation}\label{4.40.}
     \begin{aligned}
    \frac{1}{2}\int_{\mathbb{R}}h((\wt v^S)^{\mb{X}})& (\wt v^S)^{\mb{X}}_\xi  \psi^2 d\xi+\frac{1}{4} \bigintsss_{\mathbb{R}}[a^{\mb{X}} \sigma^\prime((\wt v^S)^{\mb{X}})]^\prime(\wt v^S)^{\mb{X}}_\xi\left [\phi+\frac{a^{\mb{X}}_\xi \sigma^\prime((\wt v^S)^{\mb{X}}) }{s(\sigma^\prime((\wt v^S)^{\mb{X}}) a^{\mb{X}})_\xi }\psi \right]^2d\xi\\
  =& \frac{1}{2}\int_{\mathbb{R}} h((\wt v^S)^{\mb{X}}) (\wt v^S)^{\mb{X}}_\xi  \psi^2 d\xi+\frac{1}{4} \int_{\mathbb{R}}[a^{\mb{X}} \sigma^\prime((\wt v^S)^{\mb{X}})]^\prime(\wt v^S)^{\mb{X}}_\xi \phi^2 d\xi\\
  &+\frac{1}{4}\bigintsss_{\mathbb{R}}\frac{[a^{\mb{X}}_\xi  \sigma^\prime((\wt v^S)^{\mb{X}}) ]^2}{s^2[a^{\mb{X}}\sigma^\prime((\wt v^S)^{\mb{X}})]_\xi}\psi^2 d\xi+\frac{1}{2s}\int_{\mathbb{R}} a^{\mb{X}}_\xi  \sigma^\prime((\wt v^S)^{\mb{X}})\phi\psi d\xi  \\
     \geq& \frac{1}{4}\int_{\mathbb{R}}h((\wt v^S)^{\mb{X}}) (\wt v^S)^{\mb{X}}_\xi  \psi^2 d\xi+\frac{1}{4}\int_{\mathbb{R}}  l((\wt v^S)^{\mb{X}}) (\wt v^S)^{\mb{X}}_\xi \phi^2 d\xi,
     \end {aligned}
      \end{equation}
      where $l((\wt v^S)^{\mb{X}})$ is defined in \eqref{4.36.} and in the last inequality we have used the $\epsilon$-Cauchy inequality $|cd| \leq \epsilon c^2+\frac{1}{4\epsilon} d^2$:
      \begin{equation}\label{4.41.}
     \begin{aligned}
  \frac{1}{2s}&\int_{\mathbb{R}} a^{\mb{X}}_\xi  \sigma^\prime((\wt v^S)^{\mb{X}})\phi\psi d\xi  \\
     \geq&- \frac{1}{4}\bigintsss_{\mathbb{R}} \frac{a^\prime((\wt v^S)^{\mb{X}}) \sigma^\prime((\wt v^S)^{\mb{X}})[a^{\mb{X}}\sigma^\prime((\wt v^S)^{\mb{X}}) ]^\prime} {[a^\prime((\wt v^S)^{\mb{X}}) \sigma^\prime((\wt v^S)^{\mb{X}})]^2+s^2h((\wt v^S)^{\mb{X}})[a^{\mb{X}}\sigma^\prime((\wt v^S)^{\mb{X}}) ]^\prime } a^{\mb{X}}_\xi  \sigma^\prime((\wt v^S)^{\mb{X}}) \phi^2 d\xi\\[3mm]
    &-\frac{1}{4s^2} \bigintsss_{\mathbb{R}}  \frac {[a^\prime((\wt v^S)^{\mb{X}}) \sigma^\prime((\wt v^S)^{\mb{X}})]^2+s^2h((\wt v^S)^{\mb{X}})[a^{\mb{X}}\sigma^\prime((\wt v^S)^{\mb{X}}) ]^\prime } {a^\prime((\wt v^S)^{\mb{X}})  \sigma^\prime((\wt v^S)^{\mb{X}})[a^{\mb{X}}\sigma^\prime((\wt v^S)^{\mb{X}}) ]^\prime}  a^{\mb{X}}_\xi  \sigma^\prime((\wt v^S)^{\mb{X}})\psi^2 d\xi,
     \end {aligned}
      \end{equation}
      with
        \begin{equation*}
        \begin{array}{ll}
        \displaystyle \epsilon&\di :=\frac{a^{\mb{X}}_\xi  \sigma^\prime((\wt v^S)^{\mb{X}})[a^{\mb{X}}\sigma^\prime((\wt v^S)^{\mb{X}}) ]_\xi} {[a^{\mb{X}}_\xi  \sigma^\prime((\wt v^S)^{\mb{X}})]^2+s^2h((\wt v^S)^{\mb{X}})[a^{\mb{X}}\sigma^\prime((\wt v^S)^{\mb{X}}) ]_\xi (\wt v^S)^{\mb{X}}_\xi }\\[5mm]
        &\di {=\frac{a^\prime((\wt v^S)^{\mb{X}})\sigma^\prime((\wt v^S)^{\mb{X}})[a^{\mb{X}}\sigma^\prime((\wt v^S)^{\mb{X}}) ]^\prime}{[a^\prime((\wt v^S)^{\mb{X}}) \sigma^\prime((\wt v^S)^{\mb{X}})]^2+s^2h((\wt v^S)^{\mb{X}})[a^{\mb{X}}\sigma^\prime((\wt v^S)^{\mb{X}}) ]^\prime }=1+O(v_+^2)>0}, \ \  {\rm for} \ \ 0<v_+\ll 1.
        \end{array}
 \end{equation*}
By \eqref{4.22.} and \eqref{4.66.-}, one has
 $$\displaystyle \frac{\alpha \delta_S(3+2\theta_*)\theta_*v_+^2}{4s^2} \leq h((\wt v^S)^{\mb{X}}) \leq  \frac{{13}\delta_S\alpha v_+^2}{s^2},$$
 and by \eqref{w.bounds1}  and \eqref{w.deriv1} 
 $$\displaystyle \frac{\alpha\delta_S}{2s}\leq[a^{\mb{X}}\sigma^\prime((\wt v^S)^{\mb{X}}) ]^\prime\leq \frac{\alpha\delta_S}{s},\qquad \displaystyle (a^{\mb{X}})^\prime \sigma^\prime((\wt v^S)^{\mb{X}})\leq  \frac{\alpha\delta_S}{s}.$$ 
 Therefore, we have
 \begin{equation}\label{4.42.}
    \begin{aligned}
l((\wt v^S)^{\mb{X}}) =&\frac{s^2h((\wt v^S)^{\mb{X}})[(a^{\mb{X}} \sigma^\prime((\wt v^S)^{\mb{X}}))^\prime]^2}{ [(a^{\mb{X}})^\prime \sigma^\prime((\wt v^S)^{\mb{X}})]^2+s^2[\sigma^\prime((\wt v^S)^{\mb{X}}) a^{\mb{X}}]^\prime h((\wt v^S)^{\mb{X}}) }\\[3mm]
\geq& \frac{\alpha \delta_S(3+2\theta_*)\theta_*v_+^2}{16(1+{13}sv_+^2)}>0,\qquad  {\rm for}  \ \ 0<v_+\ll 1.
 \end {aligned}
\end{equation}
The proof of Lemma \ref{L4.5.} is completed.
\end{proof}

\noindent {\bf \underline{Estimation of  the $\mathbf{N}(t)$ in \eqref{4.16}}}. For any $v_m\in [-{v_+}/{2},0)$, we have the following lemma.
 \begin{lemma}\label{L4.6} For any given positive constant $\epsilon$, there exists a positive constant $\eps_1$ such that if $\sup_{0 \le t \le T} \|(\phi,\psi)(t)\|_{H^1(\mathbb{R})}\le \eps_1$, then for all $t\le T$, it holds that
\begin{equation}\label{4.52}
     \begin{aligned}
\left|\mathbf{N}(t)\right|\leq& C{{\delta_S^2}}\eps_1 \int_{\mathbb{R}} (\wt v^S)^{\mb{X}}_\xi \psi^2 d\xi+C_\epsilon\left({{\eps_1\delta_S^2+{{\delta_S^5}\deltar^2}+\eps_1^2\delta_S}}\right)\int_{\mathbb{R}}  (\wt v^S)^{\mb{X}}_\xi  \phi^2 d\xi\\
&+C\eps_1{{\delta_S^2}}\int_{\mathbb{R}} \left|\wt u^R_\xi\right|   \phi^2 d\xi+ \epsilon{{\delta_S^2}}\left| \mathbf{\dot X}(t)\right|^2.
\end{aligned}
 \end{equation}
 \end{lemma}

\begin{proof} 
Rewrite  $\mathbf{N}(t)$ in \eqref{4.16} as 
$$\mathbf{N}(t):=\mathbf{N_1}(t)+\mathbf{N_2}(t),$$
where
  \begin{equation*}\label{4.53}
     \begin{aligned}
\mathbf{N_1}(t):=   \mathbf{\dot X}(t)[\mathbf{Y_2}(t)+\mathbf{Y_3}(t)] ,
 \end{aligned}
   \end{equation*}
   and 
  \begin{equation*}\label{4.54}
     \begin{aligned}
\mathbf{N_2}(t):=&  \int_{\mathbb{R}} a^{\mb X}\left[-s( \widetilde{v}^S)^{\mb X}_\xi + \widetilde{u}^R_\xi\right] \, \phi^3  \, d\xi -s \int_{\mathbb{R}} a^{\mb X}_\xi  \left( \widetilde v \phi^3+\frac{\phi^4}{4} \right)  d\xi\\
&- \int_{\mathbb{R}} a^{\mb X}_\xi (3\widetilde{v}\phi^2+\phi^3) \, \psi \, d\xi,
 \end{aligned}
   \end{equation*}
%   and 
   %  \begin{equation*}\label{4.55}
     %\begin{aligned}
%\mathbf{N_3}(t):=& -3s\int_{\mathbb{R}}a^{\mb X} ( \wt v^S)^{\mb X}_\xi \left(\wt v^R-v_m \right)\phi^2 d\xi\\
%&+\int_{\mathbb{R}} a^{\mb X}_\xi\left[\sigma^\prime(\wt v)-\sigma^\prime((\wt v^S)^{\mb X})\right] \phi \psi d\xi\\
%&+\frac{s}{2}\int_{\mathbb{R}}a^{\mb X}_\xi\left[\sigma^\prime(\wt v)-\sigma^\prime( (\wt v^S)^{\mb X})\right] \phi^2 d\xi.
 %\end{aligned}
   %\end{equation*}
   
For $\mathbf{N_1}(t)$, by $\epsilon$-Cauchy inequality and H${ \ddot{\rm o}}$lder inequality, we have 
 \begin{equation*}\label{4.56}
     \begin{aligned}
\left| \mathbf{\dot X}(t)\mathbf{Y_2}(t)\right|=&\left| \mathbf{\dot X}(t)\int_{\mathbb{R}}a^{\mb X}(\wt v^S)^{\mb X}_\xi \left[\sigma^\prime(\wt v)-\sigma^\prime((\wt v^S)^{\mb X})\right] \phi  d\xi\right|\\
&\leq \epsilon{{\delta_S^2}}\left| \mathbf{\dot X}(t)\right|^2+{\frac{C_{\epsilon}}{\delta_S^2}}\left|\int_{\mathbb{R}}a^{\mb X}(\wt v^S)^{\mb X}_\xi \left[\sigma^\prime(\wt v)-\sigma^\prime((\wt v^S)^{\mb X})\right] \phi  d\xi\right|^2\\
&\leq \epsilon{{\delta_S^2}}\left| \mathbf{\dot X}(t)\right|^2+{C_{\epsilon}{{{\delta_S^4}}}}\int_{\mathbb{R}}(\wt v^S)^{\mb X}_\xi  \phi^2d\xi\int_{\mathbb{R}}\left|\wt v^R-v_m\right|^2(\wt v^S)^{\mb X}_\xi d\xi\\
&\leq \epsilon{{\delta_S^2}}\left| \mathbf{\dot X}(t)\right|^2+{C_{\epsilon}{{{\delta_S^5}}}}\delta_R^2\int_{\mathbb{R}}(\wt v^S)^{\mb X}_\xi  \phi^2d\xi,
 \end{aligned}
   \end{equation*}
where $\epsilon>0$ is a small positive constant and $C_\epsilon$ depends on $\epsilon$, and we have used the facts that $a^{\mb X}\sim (\delta^S)^2$ and 
%\textcolor{blue}{$\wt v \leq C\delta_S$ (since $-\frac{v_+}{2} \leq v_m < v_- < 0 < v_+$)}, and 
$$ \left|\sigma^\prime(\wt v)-\sigma^\prime((\wt v^S)^{\mb X})\right|=3|\wt v^2-(\wt v^S)^2|=3|\wt v^R-v_m||\wt v+\wt v^S|\leq C {\delta_S}|\wt v^R-v_m|\leq C{ \delta_S}\delta_R. $$
Here, the inequality $\wt v \leq C\delta_S$ follows from $\delta_R < \delta_S$ and $\wt v^S < \delta_S$ by \eqref{s}.
By \eqref{bound.M} and \eqref{bound.dM}, we know that $ \left|\mathbf{\dot X}(t)\right|\leq C\eps_1\delta_S$. From \eqref{w.deriv2}, \eqref{w.deriv1} and \eqref{eq:Y_decomposition}, we have
 \begin{equation*}\label{4.57}
|\mathbf{Y_3}(t)|\leq C{{\delta_S}} \int_{\mathbb{R}} (\wt v^S)^{\mb{X}}_\xi \psi^2 d\xi+C{{\delta_S}} \int_{\mathbb{R}}  (\wt v^S)^{\mb{X}}_\xi  \phi^2 d\xi.
   \end{equation*}
Hence, we have 
 \begin{equation*}\label{4.58}
      \begin{aligned}
\left|\mathbf{N_1}(t)\right|\leq&  \epsilon{{\delta_S^2}}\left| \mathbf{\dot X}(t)\right|^2+C_{\epsilon}{{{\delta_S^5}}}\delta_R^2\int_{\mathbb{R}}(\wt v^S)^{\mb X}_\xi  \phi^2d\xi+C\eps_1 {{\delta_S^2}}\int_{\mathbb{R}} (\wt v^S)^{\mb{X}}_\xi \left(\psi^2+\phi^2\right) d\xi.
\end{aligned}
   \end{equation*}
For $\mathbf{N_2}(t)$, Sobolev  inequality  yields
 \begin{equation*}\label{4.59}
|\mathbf{N_2}(t)| \leq  C\eps_1{{\delta_S}}\left( {\delta_S}+ \eps_1     \right)\int_{\mathbb{R}} (\wt v^S)^{\mb{X}}_\xi \phi^2 d\xi+ C\eps_1{{\delta_S^2}}\int_{\mathbb{R}} \left|\wt u^R_\xi\right|   \phi^2 d\xi,.
   \end{equation*}
   
%   For $\mathbf{N_3}(t)$, by Talor expansion yields
 %\begin{equation*}\label{4.60}
%|\mathbf{N_3}(t)| \leq  C\textcolor{red}{{\delta_S^2}}\delta_R \int_{\mathbb{R}} (\wt v^S)^{\mb{X}}_\xi \psi^2 d\xi+C\textcolor{red}{{\delta_S^2}}\delta_R \int_{\mathbb{R}}  (\wt v^S)^{\mb{X}}_\xi  \phi^2 d\xi.
  % \end{equation*}
Hence, we complete the proof of Lemma \ref{L4.6}.
\end{proof}

From \eqref{s}, we have \(v_+ \sim \delta_S\). Moreover, we assume $\delta_R$ and $ \di \sup_{t\in[0,T]} \|(\phi,\psi)(t)\|_{H^1(\mathbb{R})}$ sufficiently small.
Observing that
$$
\mathbf{G^R}(t)=3 \int_{\mathbb{R}} a^{\mb X} \wt u^R_\xi \wt v^R  \phi^2 d\xi=3 \int_{\mathbb{R}} a^{\mb X} |\wt u^R_\xi| |\wt v^R|  \phi^2 d\xi\geq C|v_-|\delta_S^2 \int_{\mathbb{R}}\left| \wt u^R_\xi\right|   \phi^2 d\xi,
$$
and
\[
\int_{\widetilde{v}}^{\widetilde{v} + \phi} \sigma(s) \, ds - \sigma(\widetilde{v}) \phi = \frac{\sigma'(v^*)}{2} \phi^2, \qquad{v^* \in [\widetilde{v}, \widetilde{v} + \phi]},
\]
then integrating \eqref{lem:energy_identity} with respect to \(t\), by applying Lemmas \ref{L4.2}-\ref{L4.6}, we obtain the following result for any \(v_m \in [-v_+/2, 0)\).
\begin{lemma}\label{L4.7} For a given state $(v_+, u_+) \in \mathbb{R}_+ \times \mathbb{R}$, there exists a constant $\delta_1 > 0$ such that if $|v_+| \leq \delta_1$, then the following holds true. For any $(v_m, u_m) \in S_2(v_+, u_+)$ and $(v_-, u_-) \in R_1(v_m, u_m)$ satisfying \eqref{s}, there exist positive constants $\delta_2$ and $\varepsilon_2$ such that if $\delta_R \leq \delta_2$ and $\di \sup_{t\in[0,T]} \|(\phi,\psi)(t)\|_{H^1(\mathbb{R})} \leq\varepsilon_2$, then for all $t \in [0, T]$, it holds that
%For a given state $(v_+, u_+) \in \mathbb{R}_+ \times \mathbb{R}$,  there exists a constant $\delta_1 > 0$ such that if $|v_+| < \delta_1$, the the following holds true. For any $(v_m, u_m) \in S_2(v_+, u_+)$ and $(v_-, u_-) \in R_1(v_m, u_m)$ satisfying $-\frac{v_+}{2} \leq v_m < v_- < 0 < v_+$, there exist  positive constants $\delta_2$  and $\eps_2$ such that if $\delta_R<\delta_2$ and $ \mathcal{E}(T):= \displaystyle \sup_{0\leq t \leq T} \|(\phi,\psi)(t)\|_{H^1(\bbr)}<\eps_2$. Then,  $\forall t\in[0,T]$, it holds that
\begin{equation}\label{4.61}
     \begin{aligned}
\|(\phi,\psi)(t)\|^2&+ \int_0^t\left|\dot{\mb{X}}(\tau)\right|^2 d\tau+  {{\delta_S}}\int_0^t\int_\bbr\left|(\wt v^S)^{\mb X}_\xi\right|  (\phi^2+\psi^2)d\xi d\tau\\
&+{{|v_-|}}\int_0^t \int_\bbr \left|\wt u^R_\xi\right|  \phi^2d\xi d\tau+\int_0^t \|\psi_\xi(\tau)\|^2d\tau \\
\leq& \bar C\left[ \|\phi_0,\psi_0)\|^2+ \frac{1}{{\delta_S^2}} \left(\int_0^t \mathbf{F}(\tau)d \tau- \int_0^t \mathbf{J}(\tau)d \tau\right)\right].
\end{aligned}
 \end{equation}
 \end{lemma}
 It remains to estimate the wave interaction terms and error terms due to inviscid rarefaction wave profile, i.e., $\displaystyle \int_0^t  \mathbf{F}(\tau)  d\tau$ defined in \eqref{eq:def_F} and $\displaystyle \int_0^t  \mathbf{J}(\tau) \, d\tau$ defined in \eqref{4.17}.  For the wave interaction estimates, since Liu shock has polynomial decay in the degenerate case $v_m=-v_+/2$ and exponential decay in the non-degenerate case $v_m\in(-{v_+}/{2},0)$, we will estimate them separately.
% we are considering composite waves consisting of a 1-rarefaction and a 2-Liu shock, the viscous shock and the rarefaction wave are strictly separated. The required wave interaction estimates were inspired by the work of Kang, Vasseur, and Wang in \cite{M-V-W}. For a {2-non-degenerate Liu shock} (i.e., $v_m\in(-{v_+}/{2},0)$), these estimates can be directly obtained from \cite{M-V-W}. However, for the degenerate case $v_m=-{v_+}/{2}$, more detailed information is necessary. Therefore, we  estimate separately for $\displaystyle \int_0^t  \mathbf{F}(\tau) \, d\tau$ and $\displaystyle \int_0^t  \mathbf{J}(\tau) \, d\tau$ for degenerate case(i.e., $v_m=-v_+/2$) and non-degenerate case(i.e., $v_m \in (v_+/2, 0)$).
\vskip0.3cm
\textbf{\underline{Degenerate case: $v_m=-v_+/2$}}.% For $v_m=-v_+/2$, we derive the following estimates for $\displaystyle \int_0^t  \mathbf{F}(\tau) \, d\tau$ and $\displaystyle \int_0^t  \mathbf{J}(\tau) \, d\tau$.
\begin{lemma}\label{L4.8}\noindent {\bf(Wave interaction estimates)}
Let $\mb X(t)$ be the shift defined by \eqref{ode.shift2}. Under the same hypotheses as in Proposition \ref{prop:apriori}, the following holds: $\forall t\le T$,
 \begin{align*}
& \left\| \left(\wt v^R-v_m\right)(\wt v^S)^{\mb{X}}_\xi\right\|_{L^1(\mathbb{R})}\leq C\delta_R{{\delta_S}}\left[ \left(1+C{{\delta_S^2}}t\right)^{-1}+e^{-C t}\right],\\
& \left\| \left((\wt v^S)^{\mb{X}}-v_m\right)\wt v^R_\xi\right\|_{L^1(\mathbb{R})}\leq C\delta_R{{\delta_S}}\left[ \left(1+C{{\delta_S^2}}t\right)^{-1}+e^{-C t}\right],\\
& \left\| \left((\wt v^S)^{\mb{X}}-v_m\right)^2\wt v^R_\xi\right\|_{L^1(\mathbb{R})}\leq C\delta_R{\delta_S^2}\left[ \left(1+C{{\delta_S^2}}t\right)^{-2}+e^{-C t}\right],\\
& \left\| \left((\wt v^S)^{\mb{X}}-v_m\right)\wt v^R_\xi\right\|^2_{L^2(\mathbb{R})}\leq  C\delta_R^2{\delta_S^2}\left[ \left(1+C{{\delta_S^2}}t\right)^{-2}+e^{-C t}\right],\\
& \left\| \left(\wt v^R-v_m\right)(\wt v^S)^{\mb{X}}_\xi\right\|^2_{L^2(\mathbb{R})}\leq  C\delta_R^2\delta_S^3\left[(1+C\delta_S^2t)^{-3}+ e^{-C t}\right].
  \end{align*}
\end{lemma}
\begin{proof}
First, by \eqref{bound.M}, it holds that
$$ |\mathbf{\dot X}(t)|\leq C\eps_0$$
which together with $\mb X(0)=x_0$ yields that
$$
|\mb{X}(t)-x_0|\leq C\eps_0 t,\qquad 0\le t\le T.
$$
Take $\eps_0$ suitably small such that the above bound is less than $ \frac{s t}{4}$, i.e.,
\[
C\eps_0 t < \frac{s t}{4}.
\]
Then, by Lemma \ref{lemma:shock_properties} and Lemma \ref{lemma1.2}, we have 
   \begin{align*}
\displaystyle \int_{-\infty}^{-\frac{st}{2}}&\left|\wt v^R-v_m\right|\left|(\wt v^S)^{\mb{X}}_\xi\right|d\xi\leq \delta_R\left|\wt v^S\left(\mb{X}(t)-\frac{st}{2}\right)-v_m   \right|\\
&\leq\delta_R\left|\wt v^S\left(-\frac{st}{4}+x_0\right)-v_m   \right|\leq \begin{cases}
       C\delta_R{{\delta_S}}  \left(1+C{{\delta_S^2}}t\right)^{-1},  \quad  t\geq\frac {8x_0}{s}, \\[4mm]
        \delta_R{{\delta_S}} , \quad  t<\frac {8x_0}{s}, 
         \end{cases}
    \end{align*}
              \begin{align*}
&\displaystyle \int^{+\infty}_{-\frac{st}{2}}\left|\wt v^R-v_m\right|\left|(\wt v^S)^{\mb{X}}_\xi\right|d\xi\leq C\delta_S\delta_Re^{-2\left|\frac{st}{2}-\lambda_-(v_m)t\right|}\leq  C\delta_R{{\delta_S}}e^{-C t},
%&\displaystyle \int_{-\infty}^{-\frac{st}{2}}\left|\wt v^R-v_m\right|^2\left|(\wt v^S)^{\mb{X}}_\xi\right|d\xi\leq \delta_R^2\left|\wt v^S\left(\mb{X}(t)-\frac{st}{2}\right)-v_m   \right|\leq C\delta_R^2{{\delta_S}} \left(1+C{{\delta_S^2}}t\right)^{-1},\\
%&\displaystyle \int^{+\infty}_{-\frac{st}{2}}\left|\wt v^R-v_m\right|^2\left|(\wt v^S)^{\mb{X}}_\xi\right|d\xi\leq C\delta_S\delta_R^2e^{-2\left|\frac{st}{2}-\lambda_-(v_m)t\right|}\leq  C\delta_R^2{{\delta_S}}e^{-C t},\\
  \end{align*}
              \begin{align*}
 \displaystyle \int_{-\infty}^{-\frac{st}{2}}\left|(\wt v^S)^{\mb{X}}-v_m\right|\left|\wt v^R_\xi\right|d\xi &\leq\delta_R\left|\wt v^S\left(\mb{X}(t)-\frac{st}{2}\right)-v_m   \right|\\
 & \leq \begin{cases}
       C\delta_R{{\delta_S}}  \left(1+C{{\delta_S^2}}t\right)^{-1}, \quad t\geq\frac {8x_0}{s}, \\[4mm]
        \delta_R{{\delta_S}} , \quad t<\frac {8x_0}{s}, 
         \end{cases}
         \end{align*}
              \begin{align*}
  & \displaystyle \int^{+\infty}_{-\frac{st}{2}}\left|(\wt v^S)^{\mb{X}}-v_m\right|\left|\wt v^R_\xi\right|d\xi\leq \delta_S\left|v_m-\wt v^R\left(1+t,\frac{st}{2}\right) \right|   \leq  C\delta_R{{\delta_S}}e^{-C t},
  \end{align*}
              \begin{align*}
\displaystyle \int_{-\infty}^{-\frac{st}{2}}\left|(\wt v^S)^{\mb{X}}-v_m\right|^2\left|\wt v^R_\xi\right|d\xi&\leq \delta_R\left|\wt v^S\left(\mb{X}(t)-\frac{st}{2}\right)-v_m   \right|^2\\
&\leq 
 \begin{cases}
 C\delta_R{{\delta_S^2}} \left(1+C{{\delta_S^2}}t\right)^{-2}, \quad t\geq\frac {8x_0}{s} \\[4mm]
        \delta_R{{\delta_S^2}} , \quad t<\frac {8x_0}{s}, 
 \end{cases}
   \end{align*}
              \begin{align*}
& \displaystyle \int^{+\infty}_{-\frac{st}{2}}\left|(\wt v^S)^{\mb{X}}-v_m\right|^2\left|\wt v^R_\xi\right|d\xi \leq  \delta_S^2 \left|v_m-\wt v^R\left(1+t,\frac{st}{2}\right) \right|    \leq  C\delta_R{{\delta_S^2}}e^{-C t}.
  \end{align*}
  From above inequalities, we have
   \begin{align*}
& \left\| \left(\wt v^R-v_m\right)(\wt v^S)^{\mb{X}}_\xi\right\|_{L^1(\mathbb{R})}\leq C\delta_R{{\delta_S}}\left[ \left(1+C{{\delta_S^2}}t\right)^{-1}+e^{-C t}\right],\\
& \left\| \left((\wt v^S)^{\mb{X}}-v_m\right)\wt v^R_\xi\right\|_{L^1(\mathbb{R})}\leq C\delta_R{{\delta_S}}\left[ \left(1+C{{\delta_S^2}}t\right)^{-1}+e^{-C t}\right],\\
& \left\| \left((\wt v^S)^{\mb{X}}-v_m\right)^2\wt v^R_\xi\right\|_{L^1(\mathbb{R})}\leq C\delta_R{\delta_S^2}\left[ \left(1+C{{\delta_S^2}}t\right)^{-2}+e^{-C t}\right].\\
  \end{align*}
Furthermore, it holds that
 \begin{equation*}
  \begin{aligned}
&\int_{\mathbb{R}}\left| ((\wt v^S)^{\mb{X}}-v_m) (\wt v^R)_\xi\right|^2d\xi\leq C\|(\wt v^R)_\xi\|_{L^\infty(\mathbb R)}\int_{\mathbb{R}}\left|( (\wt v^S)^{\mb{X}}-v_m)^2 \wt v^R_\xi\right|d\xi\\
&\leq C\delta_R^2{\delta_S^2}\left[ \left(1+C{{\delta_S^2}}t\right)^{-2}+e^{-C t}\right].
\end{aligned}
 \end{equation*}
  We have
  \begin{equation*}
  \begin{aligned}
  \int^{+\infty}_{-\frac{st}{2}} \left|(\wt v^R-v_m)(\wt v^S)^{\mb{X}}_{\xi} \right|^2d\xi&\leq  \delta_R\left\|(\wt v^S)^{\mb{X}}_{\xi}\right\|_{L^\infty(\mathbb R)} \int^{+\infty}_{-\frac{st}{2}}\left|\wt v^R-v_m\right|\left|(\wt v^S)^{\mb{X}}_\xi\right|d\xi\\&\leq C\left\|(\wt v^S)^{\mb{X}}_{\xi}\right\|_{L^\infty(\mathbb R)}\delta_S\delta_R^2e^{-C t}\leq C\delta_S^3\delta_R^2e^{-C t},
    \end{aligned}
    \end{equation*}
 and
 \begin{equation*}
  \begin{aligned}
  \int_{-\infty}^{-\frac{st}{2}} \left|(\wt v^R-v_m)(\wt v^S)^{\mb{X}}_{\xi} \right|^2d\xi&\leq  \delta_R\left\|(\wt v^S)^{\mb{X}}_{\xi}\right\|_{L^\infty(-\infty,-\frac{st}{2})} \int_{-\infty}^{-\frac{st}{2}}\left|\wt v^R-v_m\right|\left|(\wt v^S)^{\mb{X}}_\xi\right|d\xi\\&\leq  \begin{cases}
   C\delta_S^4\delta_R^2(1+C\delta_S^2t)^{-3}, \quad t\geq\frac {8x_0}{s}, \\[4mm]
   C \delta_S^3\delta_R^2, \quad t<\frac {8x_0}{s}, 
   \end{cases}
    \end{aligned}
    \end{equation*}  
where we have used the fact that
 \begin{equation*}
  \begin{aligned}
\left\|(\wt v^S)^{\mb{X}}_{\xi}\right\|_{L^\infty(-\infty,-\frac{st}{2})}&\leq \left\| \wt v^S_{\xi}\right\|_{L^\infty(-\infty,-\frac{st}{4}+x_0)}\\ &\leq \begin{cases}
   C\delta_S^3(1+C\delta_S^2t)^{-2}, \quad t\geq\frac {8x_0}{s}, \\[4mm]
    C\delta_S^2, \quad t<\frac {8x_0}{s}.
   \end{cases}
       \end{aligned}
    \end{equation*}  
Therefore, it holds that 
  \begin{equation*}
  \begin{aligned}
  \int_{\mathbb {R}} \left|(\wt v^R-v_m)(\wt v^S)^{\mb{X}}_{\xi} \right|^2d\xi&\leq C\delta_S^3\delta_R^2\left[(1+C\delta_S^2t)^{-3}+ e^{-C t}\right].
    \end{aligned}
    \end{equation*}
 %  Hence, we have 
   %\begin{equation*}
  %\begin{aligned}
  %\int_{-\infty}^{-\frac{st}{2}} \left|(\wt v^R-v_m)(\wt v^S)^{\mb{X}}_{\xi} \right|^2d\xi\leq C \delta_S^3\delta_R^2 (1+C\delta_S^2t)^{-3}.
    %\end{aligned}
    %\end{equation*} 

\end{proof}
%\noindent {\bf \underline{Estimation of  the $\displaystyle\int_0^t  \mathbf{J}(\tau)d\tau$.}} 
 \begin{lemma}\label{L4.9}There exists a positive constant $\eps_3$ such that if $\di \sup_{t\in[0,T]} \|(\phi,\psi)(t)\|_{H^1(\mathbb{R})}\le \eps_3$, then for all $t\le T$, it holds that
\begin{equation*}\label{4.47}
     \begin{aligned}
\left|\int_0^t  \mathbf{J}(\tau)d\tau\right|\leq C \deltar{{\delta_S^2}}+C\eps_3^2\delta_S^2\int_0^t\int_{\mathbb{R}}\left| \wt u^R_\xi\right|\phi^2 d\xi d\tau+C\delta_S^2 \delta_R\int_{\mathbb{R}}( \wt v^S)^{\mb X}_\xi (\phi^2+\psi^2) d\xi.
\end{aligned}
 \end{equation*}
 \end{lemma}
\begin{proof} $\mathbf{J}(t)$ is defined in \eqref{4.17} and we rewrite  $\mathbf{J}(t)$ as 
$$\mathbf{J}(t)=\mathbf{J_1}(t)+\mathbf{J_2}(t)+\mathbf{J_3}(t)+\mathbf{J_4}(t),$$
 with
  \begin{equation*}
     \begin{aligned}
\mathbf{J_1}(t):=  3\int_{\mathbb{R}} a^{\mb X}\wt u^R_\xi\left[(\wt v^S)^{\mb X}-v_m\right]\phi^2 d\xi,
 \end{aligned}
   \end{equation*}
  \begin{equation*}
     \begin{aligned}
\mathbf{J_2}(t):=-3s\int_{\mathbb{R}} a^{\mb X}( \wt v^S)^{\mb X}_\xi \left(\wt v^R-v_m \right) \phi^2 d\xi,
 \end{aligned}
   \end{equation*}
     \begin{equation*}
      \begin{aligned}
\mathbf{J_3}(t):=\int_{\mathbb{R}} a^{\mb X}_\xi\left[\sigma^\prime(\wt v)-\sigma^\prime((\wt v^S)^{\mb X})\right] \phi \psi d\xi,
 \end{aligned}
   \end{equation*}
and 
\begin{equation*}
      \begin{aligned}
\mathbf{J_4}(t):=\frac{s}{2}\int_{\mathbb{R}}a^{\mb X}_\xi \left[\sigma^\prime(\wt v)-\sigma^\prime( (\wt v^S)^{\mb X})\right]\phi^2 d\xi.
 \end{aligned}
   \end{equation*}

By $\epsilon$-Cauchy inequality, we have

   \begin{equation*}
      \begin{aligned}
|\mathbf{J_1}(t)|&=3\left|\int_{\mathbb{R}} a^{\mb X}\wt u^R_\xi\left[(\wt v^S)^{\mb X}-v_m\right]\phi^2 d\xi\right|\\
&\leq\int_{\mathbb{R}}  a^{\mb X}\left|\wt u^R_\xi\right|\left[(\wt v^S)^{\mb X}-v_m\right]^2d\xi+C\int_{\mathbb{R}}a^{\mb X} \left| \wt u^R_\xi\right|\phi^4 d\xi\\
%&\leq C{{\delta_S^4}}\delta_R \left(1+C{{\delta_S^2}}t\right)^{-2}+ C\delta_R{{\delta_S^4}}e^{-C t}+C\eps_3^2\delta_S^2\int_{\mathbb{R}} \left| \wt u^R_\xi\right|\phi^2 d\xi,
&\leq C{{\delta_S^2}}\int_{\mathbb{R}} \left|\wt u^R_\xi\right|\left[(\wt v^S)^{\mb X}-v_m\right]^2d\xi+ C\eps_3^2\delta_S^2\int_{\mathbb{R}} \left| \wt u^R_\xi\right|\phi^2 d\xi,
 \end{aligned}
   \end{equation*}
   which implies that 
 $$  
   \left|\int_0^t \mathbf{J_1}(\tau)d\tau \right|\leq C \deltar{{\delta_S^2}}+C\eps_3^2\delta_S^2\int_0^t\int_{\mathbb{R}}\left| \wt u^R_\xi\right|\phi^2 d\xi d\tau,
   $$
 by Lemma \ref{L4.8}.  
 
   We have
   \begin{equation*}
      \begin{aligned}
|\mathbf{J_2}(t)|&=3s\left|\int_{\mathbb{R}} a^{\mb X}( \wt v^S)^{\mb X}_\xi \left(\wt v^R-v_m \right) \phi^2 d\xi\right|\leq C\delta_S^2 \delta_R\int_{\mathbb{R}}( \wt v^S)^{\mb X}_\xi \phi^2 d\xi,
%&\leq\int_{\mathbb{R}}  a^{\mb X}\left|( \wt v^S)^{\mb X}_\xi \right|\left(\wt v^R-v_m\right)^2d\xi+C\int_{\mathbb{R}}a^{\mb X} \left| ( \wt v^S)^{\mb X}_\xi \right|\phi^4 d\xi\\
%&\leq C{{\delta_S^3}}\delta_R^2\left[ \left(1+C{{\delta_S^2}}t\right)^{-2}+ e^{-C t}\right]+C\eps_3^2\delta_S^2\int_{\mathbb{R}} \left|  ( \wt v^S)^{\mb X}_\xi\right|\phi^2 d\xi.
 \end{aligned}
   \end{equation*}
   Similarly, it is easy to get
  $$
  (|\mathbf{J_3}(t)|,|\mathbf{J_4}(t)|)\leq C\delta_S^2 \delta_R\int_{\mathbb{R}}( \wt v^S)^{\mb X}_\xi (\phi^2+\psi^2) d\xi.
  %\leq C{{\delta_S^3}}\delta_R^2\left[ \left(1+C{{\delta_S^2}}t\right)^{-2}+ e^{-C t}\right]+C\eps_3^2\delta_S^2\int_{\mathbb{R}} \left| ( \wt v^S)^{\mb X}_\xi\right|\phi^2 d\xi.
  $$

Integrating with respect to $t$, we complete the proof.
\end{proof}

%\noindent {\bf \underline{Estimation of  the $\displaystyle \int_0^t\mathbf{F}(\tau)d\tau$}. } 
\begin{lemma}\label{L4.10} There exists a positive constant $\eps_4$ such that if $\di \sup_{t\in[0, T]} \|(\phi,\psi)(t)\|_{H^1(\mathbb{R})}\le \eps_4$, then for all $t\le T$, it holds that
\begin{equation*}\label{4.63}
     \begin{aligned}
\left|\int_0^t  \mathbf{F}(\tau)d\tau\right|\leq C\eps_4^{\frac{2}{3}} {{\delta_S^2}}\left( \deltar^{\frac{1}{3}}+ \deltar^\frac{4}{3} \delta_S^{\frac{1}{3}}\right)+\frac{{{\delta_S^2}}}{4 \bar C} \int_0^t\|\psi_\xi(\tau)\|^2 d\tau .
\end{aligned}
 \end{equation*}
 \end{lemma}
\begin{proof} Since $F$ is defined in \eqref{F}, we have 
\begin{equation*}\label{4.64}
     \begin{aligned}
\left|\int_{\mathbb{R}}  a^{\mb X} F\psi d\xi\right|& \leq C{{\delta_S^2}}\int_{\mathbb{R}} |F||\psi| d\xi\leq C{{\delta_S^2}}\int_{\mathbb{R}} |F|d\xi \|\psi\|_{L^\infty}\\
&\leq C{{\delta_S^2}}\int_{\mathbb{R}} |F|d\xi \|\psi\|^{\frac{1}{2}}\|\psi_\xi\|^{\frac{1}{2}}\leq \frac{{{\delta_S^2}}}{4 \bar C} \|\psi_\xi\|^2+C{{\delta_S^2}}\|\psi\|^{\frac{2}{3}}\left(\int_{\mathbb{R}} |F|d\xi\right)^{\frac{4}{3}},
\end{aligned}
 \end{equation*}
where we have used Sobolev inequality $\|\psi\|_{L^\infty(\bbr)}\leq \sqrt 2  \|\psi_\xi\|^{\frac{1}{2}} \|\psi\|^{\frac{1}{2}}$ and $\epsilon$-Cauchy inequality.
It holds that 
\begin{equation*}\label{4.65}
\int_{\mathbb{R}} |F|d\xi\leq C{{\delta_S}}\int_{\mathbb{R}}\left|\wt v^R-v_m\right|\left|(\wt v^S)^{\mb{X}}_\xi\right|d\xi+   C{{\delta_S}}\int_{\mathbb{R}}\left|(\wt v^S)^{\mb{X}}-v_m\right|\left|\wt v^R_\xi\right|d\xi+  \int_\mathbb{R} |\wt u^R_{\xi\xi}|d\xi.
 \end{equation*}
By Lemma \ref{L4.8}, we have 
\begin{equation*}\label{4.66}
 \int_0^{t} \Big(\int_{\mathbb{R}}\left|\wt v^R-v_m\right|\left|(\wt v^S)^{\mb{X}}_\xi\right|d\xi\Big)^\frac{4}{3}d\tau \leq C\delta_R^\frac{4}{3}{{\delta_S^{-\frac{2}{3}}}},
  \end{equation*}
  and 
\begin{equation*}\label{4.67}
 \int_0^{t} \Big(\int_{\mathbb{R}}\left|(\wt v^S)^{\mb{X}}-v_m\right|\left|\wt v^R_\xi\right|d\xi\Big)^\frac{4}{3}d\tau \leq C\delta_R^\frac{4}{3}{{\delta_S^{-\frac{2}{3}}}}.
  \end{equation*}
By Lemma \ref{lemma1.2}, we have
\begin{equation*}\label{4.52}
    \int_\mathbb{R} |\wt u^R_{\xi\xi}|d\xi\leq
\left\{\begin{array}{ll}
\delta_R, &\di  1+t\leq \delta_R^{-1},\\[2mm]
(1+t)^{-1}, &\di 1+t\geq \delta_R^{-1},
  \end{array}\right.
    \end{equation*}
  which implies 
  \begin{equation*}\label{4.53}
 \int_0^{t} \Big(\int_\mathbb{R} |\wt u^R_{\xi\xi}|d\xi \Big)^\frac{4}{3}d\tau \leq C\delta_R^\frac{1}{3}.
  \end{equation*}
Thus, we complete the proof of Lemma \ref{L4.10}.  
\end{proof}
\vskip0.3cm
\textbf{\underline{Non-degenerate case: $v_m\in(-v_+/2,0)$}}. % For $v_m\in(-v_+/2,0)$, we derive the following estimates for $\displaystyle \int_0^t  \mathbf{F}(\tau) \, d\tau$ and $\displaystyle \int_0^t  \mathbf{J}(\tau) \, d\tau$.
\begin{lemma}\label{L4.8.}\noindent {\bf(Wave interaction estimates)}
Let $\mb X(t)$ be the shift defined by \eqref{ode.shift2} or \eqref{ode.shift1}. Under the same hypotheses as in Proposition \ref{prop:apriori}, the following holds: $\forall t\le T$,
%\begin{align*}
%&\|(\wt v^S)^{\mb{X}}_{\xi} (\wt v^R-v_m) \|_{L^1(\bbr)} +\|(\wt v^R)_\xi ((\wt v^S)^{\mb X}-v_m)\|_{L^1(\bbr)} \le C\deltar \deltas e^{-C \deltas t},\\
%&\|(\wt v^S)^{\mb{X}}_{\xi} (\wt v^R-v_m) \|_{L^1(\bbr)} + \| (\wt v^R)_\xi (\wt v^S)^{\mb{X}}_\xi \|_{L^1(\bbr)} \le C\deltar \deltas e^{-C \deltas t},\\
%&\|(\wt v^S)^{\mb{X}}_{\xi} (\wt v^R-v_m) \|_{L^2(\bbr)} + \| (\wt v^R)_\xi (\wt v^S)^{\mb{X}}_\xi \|_{L^2(\bbr)} \le C\deltar \deltas^{3/2} e^{-C \deltas t},\\
%&\|(\wt v^R)_\xi ((\wt v^S)^{\mb X}-v_m)\|_{L^2(\bbr)} \le C\deltar \deltas e^{-C \deltas t} ,\\
%&\|(\wt v^R)_\xi ((\wt v^S)^{\mb X}-v_m)\|_{L^1(\bbr)} \le C\deltar \deltas e^{-C \deltas t}.
%\end{align*}
 \begin{align*}
& \left\| \left(\wt v^R-v_m\right)(\wt v^S)^{\mb{X}}_\xi\right\|_{L^1(\mathbb{R})}\leq C\delta_R{\delta_S}e^{-C \delta_S t},\\
& \left\| \left((\wt v^S)^{\mb{X}}-v_m\right)\wt v^R_\xi\right\|_{L^1(\mathbb{R})}\leq C\delta_R{\delta_S}e^{-C \delta_S t},\\
& \left\| \left((\wt v^S)^{\mb{X}}-v_m\right)\wt v^R_\xi\right\|^2_{L^2(\mathbb{R})}\leq  C\delta_R^2{\delta_S^2}e^{-C \delta_S t},\\
& \left\| \left(\wt v^R-v_m\right)(\wt v^S)^{\mb{X}}_\xi\right\|^2_{L^2(\mathbb{R})}\leq  C\delta_R^2{\delta_S^3}e^{-C \delta_S t}.
  \end{align*}
% \begin{align*}
    %&\displaystyle \int_{-\infty}^{-\frac{st}{2}}\left|\wt v^R-v_m\right|\left|(\wt v^S)^{\mb{X}}_\xi\right|d\xi\leq C\delta_R\deltas e^{-C \deltas t},\\
    %&\displaystyle \int^{+\infty}_{-\frac{st}{2}}\left|\wt v^R-v_m\right|\left|(\wt v^S)^{\mb{X}}_\xi\right|d\xi\leq C\delta_R{{\delta_S}}e^{-C t},\\
   %& \displaystyle \int_{-\infty}^{-\frac{st}{2}}\left|(\wt v^S)^{\mb{X}}-v_m\right|\left|\wt v^R_\xi\right|d\xi\leq C\delta_R\deltas e^{-C \deltas t}, \\
   %& \displaystyle \int^{+\infty}_{-\frac{st}{2}}\left|(\wt v^S)^{\mb{X}}-v_m\right|\left|\wt v^R_\xi\right|d\xi\leq  C\delta_R{{\delta_S}}e^{-C t},\\
   %   \end{align*}
\end{lemma}      
\begin{proof}
Take $\eps_0$ suitably small such that
\[
|\mb{X}(t)-x_0|< \frac{s t}{4}.
\]
Then, by Lemma \ref{lemma:shock_properties} and Lemma \ref{lemma1.2}, we have 
   \begin{align*}
&\displaystyle \int_{-\infty}^{-\frac{st}{2}}\left|\wt v^R-v_m\right|\left|(\wt v^S)^{\mb{X}}_\xi\right|d\xi\leq \delta_R\left|\wt v^S\left(\mb{X}(t)-\frac{st}{2}\right)-v_m   \right|\leq
\begin{cases}
       C\delta_R{{\delta_S}} e^{-C \deltas t},  \forall t\geq\frac {8x_0}{s}, \\[4mm]
        \delta_R{{\delta_S}} , \forall t<\frac {8x_0}{s}, 
         \end{cases}\\
&\displaystyle \int^{+\infty}_{-\frac{st}{2}}\left|\wt v^R-v_m\right|\left|(\wt v^S)^{\mb{X}}_\xi\right|d\xi\leq C\delta_S\delta_Re^{-2\left|\frac{st}{2}-\lambda_-(v_m)t\right|}\leq  C\delta_R{{\delta_S}}e^{-C t},\\
 &\displaystyle \int_{-\infty}^{-\frac{st}{2}}\left|(\wt v^S)^{\mb{X}}-v_m\right|\left|\wt v^R_\xi\right|d\xi \leq\delta_R\left|\wt v^S\left(\mb{X}(t)-\frac{st}{2}\right)-v_m   \right|\leq
\begin{cases}
       C\delta_R{{\delta_S}} e^{-C \deltas t},  \forall t\geq\frac {8x_0}{s}, \\[4mm]
        \delta_R{{\delta_S}} , \forall t<\frac {8x_0}{s}, 
         \end{cases}\\
  & \displaystyle \int^{+\infty}_{-\frac{st}{2}}\left|(\wt v^S)^{\mb{X}}-v_m\right|\left|\wt v^R_\xi\right|d\xi\leq \delta_S\left|v_m-\wt v^R\left(1+t,\frac{st}{2}\right) \right|   \leq  C\delta_R{{\delta_S}}e^{-C t}.\\
  \end{align*}
    From above inequalities, we have
     \begin{align*}
& \left\| \left(\wt v^R-v_m\right)(\wt v^S)^{\mb{X}}_\xi\right\|_{L^1(\mathbb{R})}\leq C\delta_R{\delta_S}e^{-C \delta_S t},\\
& \left\| \left((\wt v^S)^{\mb{X}}-v_m\right)\wt v^R_\xi\right\|_{L^1(\mathbb{R})}\leq C\delta_R{\delta_S}e^{-C \delta_S t}.
  \end{align*}
Furthermore, it holds that
 \begin{equation*}
  \begin{aligned}
&\int_{\mathbb{R}}\left| ((\wt v^S)^{\mb{X}}-v_m) (\wt v^R)_\xi\right|^2d\xi\leq C\delta_S\|(\wt v^R)_\xi\|_{L^\infty(\mathbb R)}\int_{\mathbb{R}}\left| ((\wt v^S)^{\mb{X}}-v_m) \wt v^R_\xi\right|d\xi\leq C\delta_R^2{\delta_S^2}e^{-C \delta_S t}.
\end{aligned}
 \end{equation*}
 and 
  \begin{equation*}
  \begin{aligned}
&\int_{\mathbb{R}} \left|(\wt v^R-v_m)(\wt v^S)^{\mb{X}}_{\xi} \right|^2 d\xi\leq C\delta_R\|(\wt v^S)_\xi\|_{L^\infty(\mathbb R)}\int_{\mathbb{R}}\left|\wt v^R-v_m\right|\left|(\wt v^S)^{\mb{X}}_\xi \right|d\xi\leq C\delta_R^2{\delta_S^3}e^{-C \delta_S t}.
\end{aligned}
 \end{equation*}
\end{proof}

%\noindent {\bf \underline{Estimation of  the $\displaystyle \int_0^t  \mathbf{J}(\tau)d\tau$.}}  
\begin{lemma}\label{L4.9.}There exists a positive constant $\eps_5$ such that if $\di \sup_{t\in[0,T]} \|(\phi,\psi)(t)\|_{H^1(\mathbb{R})}\le \eps_5$, then for all $t\le T$, it holds that
\begin{equation*}\label{4.47.}
     \begin{aligned}
\left|\int_0^t  \mathbf{J}(\tau)d\tau \right|\leq C \deltar{\delta_S^2\eps_5^2}.%+\frac{1}{4 \bar C}\int_0^t|\dot{\mb{X}}|^2 d\tau .
\end{aligned}
 \end{equation*}
 \end{lemma}
\begin{proof} 
%We rewrite  $\mathbf{J}(t)$ as  $\mathbf{J}(t)= \mathbf{J_1}(t)+ \mathbf{J_2}(t)$, where
 %\begin{equation}\label{4.48.}
    %  \begin{aligned}
%\mathbf{J_1}(t):=3\int_{\mathbb{R}} \wt u^R_\xi[(\wt v^S)^{\mb X}-v_m]\phi^2 w^{\mb X}d\xi-3s\int_{\mathbb{R}} (\wt v^S)^{\mb X}_\xi (\wt v^R-v_m )w^{\mb X}\phi^2 d\xi\\
%+\int_{\mathbb{R}} w^{\mb X}_\xi[\sigma^\prime(\wt v)-\sigma^\prime( (\wt v^S)^{\mb X})] \phi \psi d\xi+\frac{s}{2}\int_{\mathbb{R}}[\sigma^\prime(\wt v)-\sigma^\prime( (\wt v^S)^{\mb X})]w^{\mb X}_\xi \phi^2 d\xi,
 %\end{aligned}
   %\end{equation}
   %and
   %\begin{equation}\label{4.49}
      %\begin{aligned}
%\mathbf{J_2}(t):=\mathbf{\dot X}(t)\mathbf{Y_2}(t).
 %\end{aligned}
   %\end{equation}
%Since $ \wt u^R_\xi=-\lambda_-( \wt v^R) \wt v^R_\xi$, by Lemma \ref{L4.8}, we have 
%$$
By Lemma \ref{L4.8.}, we have
$$
\left|\mathbf{J}(t)\right|\leq C\eps_5^2 \delta_S^3\delta_Re^{-C\delta_St}.
$$
%$$
%
%By $\eps$-Cauchy inequality and Lemma \ref{L4.8}, we have 
%$$
%\left|\mathbf{J_2}(t)\right|\leq C \eps_4^2 \delta_S^2\delta_R^2e^{-C\delta_St}+\frac{1}{4 \bar C} |\dot{\mb{X}}|^2.
%$$
Integrating with respect to $t$, we complete the proof.
\end{proof}

%\noindent {\bf \underline{Estimation of  the $\displaystyle \int_0^t  \mathbf{F}(\tau)d\tau$.} } 
\begin{lemma}\label{L4.10.} There exists a positive constant $\eps_6$ such that if $\di \sup_{t\in[0,T]} \|(\phi,\psi)(t)\|_{H^1(\mathbb{R})}\le \eps_6$, then for all $t\le T$, it holds that
\begin{equation*}\label{4.50.}
     \begin{aligned}
\left|\int_0^t  \mathbf{F}(\tau)d\tau \right|\leq C \eps_6^{\frac{2}{3}} {\delta_S^2} \deltar^{\frac{1}{3}}+\frac{ {\delta_S^2}}{4 \bar C} \int_0^t\|\psi_\xi(\tau)\|^2 d\tau .
\end{aligned}
 \end{equation*}
 \end{lemma}
\begin{proof} By Sobolev inequality $\|\psi\|_{L^\infty(\bbr)}\leq \sqrt 2  \|\psi_\xi\|^{\frac{1}{2}} \|\psi\|^{\frac{1}{2}}$ and $\epsilon$-Cauchy inequality, we have 
\begin{equation*}\label{4.51.}
     \begin{aligned}
     \left|\int_{\mathbb{R}}  a^{\mb X} F\psi d\xi\right|& \leq C{{\delta_S^2}}\int_{\mathbb{R}} |F||\psi| d\xi\leq C{{\delta_S^2}}\int_{\mathbb{R}} |F|d\xi \|\psi\|_{L^\infty}\\
&\leq C{{\delta_S^2}}\int_{\mathbb{R}} |F|d\xi \|\psi\|^{\frac{1}{2}}\|\psi_\xi\|^{\frac{1}{2}}\\
&\leq \frac{{{\delta_S^2}}}{4 \bar C} \|\psi_\xi\|^2+C{{\delta_S^2}}\|\psi\|^{\frac{2}{3}}\left(\int_{\mathbb{R}} |F|d\xi\right)^{\frac{4}{3}}\\
&\leq \frac{{{\delta_S^2}}}{4 \bar C} \|\psi_\xi\|^2+C\delta_S^2\eps_6^{\frac{2}{3}}\left( \delta_R^{\frac{4}{3}}\delta_S^{\frac{4}{3}}+\delta_R^{\frac{1}{3}}   \right).
\end{aligned}
 \end{equation*}
 Since we have
\begin{equation*}\label{4.65}
\int_{\mathbb{R}} |F|d\xi\leq C{{\delta_S}}\int_{\mathbb{R}}\left|\wt v^R-v_m\right|\left|(\wt v^S)^{\mb{X}}_\xi\right|d\xi+   C{{\delta_S}}\int_{\mathbb{R}}\left|(\wt v^S)^{\mb{X}}-v_m\right|\left|\wt v^R_\xi\right|d\xi+  \int_\mathbb{R} |\wt u^R_{\xi\xi}|d\xi,
 \end{equation*}
 it follows by Lemma \ref{L4.8.} that 
\begin{equation*}\label{4.66}
 \int_0^{t} \Big(\int_{\mathbb{R}}\left|\wt v^R-v_m\right|\left|(\wt v^S)^{\mb{X}}_\xi\right|d\xi\Big)^\frac{4}{3}d\tau \leq C\delta_R^\frac{4}{3}{{\delta_S^{\frac{1}{3}}}},
  \end{equation*}
  and 
\begin{equation*}\label{4.67}
 \int_0^{t} \Big(\int_{\mathbb{R}}\left|(\wt v^S)^{\mb{X}}-v_m\right|\left|\wt v^R_\xi\right|d\xi\Big)^\frac{4}{3}d\tau \leq C\delta_R^\frac{4}{3}{{\delta_S^{\frac{1}{3}}}},
  \end{equation*}
and
  \begin{equation*}\label{4.53.}
 \int_0^{t} \Big(\int_\mathbb{R} \left|\wt u^R_{\xi\xi}\right|d\xi \Big)^\frac{4}{3}d\tau \leq C\delta_R^\frac{1}{3}.
  \end{equation*}
Thus, we complete the proof of Lemma \ref{L4.10.}.  
\end{proof}

By Lemmas \ref{L4.7}-\ref{L4.10.}, %under the conditions $v_+\ll1$ and  $\mathcal{E}(T)\ll1$, 
we can obtain Proposition \ref{prop:key_energy}. The following two Propositions are concerned with $L^2$-estimates for the first order derivatives.

\begin{proposition} \label{P4.2} Under the same hypotheses of Proposition \ref{prop:apriori}, there exists a positive constant $C$ such that  for all $t\le T$, it holds that

\begin{align}
\begin{aligned}\label{4.53}
\int_\mathbb{R} \phi_\xi^2 d\xi+ \int_0^t \|\phi_\xi(\tau)\|^2d\tau \leq C\left(\|\phi_0\|^2_{H^1(\bbr)}+\|\psi_0\|^2 +\delta_R^{\frac{1}{3}}\right).
\end{aligned}
\end{align}
 \end{proposition}
\begin{proof}   
Multiplying  $\eqref{Per}_2$  by $-\phi_\xi$, we have
 $$
-\psi_t\phi_\xi=-(\psi \phi_\xi)_t+\psi\phi_{t\xi}=-(\psi \phi_\xi)_t+\psi\left[s\phi_\xi+\psi_\xi- \dot{\mb{X}}(t)(\wt v^S)^\mb{X}_\xi \right]_\xi.
  $$
  and 
  $$
\left[s\phi_\xi+\psi_\xi-\dot{\mb{X}}(t) (\wt v^S)^\mb{X}_\xi\right]_\xi \psi=s(\psi \phi_\xi)_\xi-s\phi_\xi\psi_\xi+(\psi\psi_\xi)_\xi-\psi_\xi^2-\dot{\mb{X}}(t) (\wt v^S)^\mb{X}_{\xi\xi}\psi. 
$$
  Furthermore, we have
 $$
\phi_\xi [\sigma(\phi+\wt v)-\sigma(\wt v)]_\xi=\sigma^\prime(\wt v+\phi)\phi_\xi^2+\phi_\xi[\sigma^\prime(\wt v+\phi)-\sigma^\prime(\wt v) ]\wt v_\xi ,
$$
   and 
  $$
     \mu\psi_{\xi\xi}\phi_\xi=\mu \phi_\xi[\phi_t-s\phi_\xi+\dot{\mb{X}}(t)(\wt v^S)^\mb{X}_\xi ]_\xi=\frac{\mu}{2}(\phi_\xi^2)_t-\frac{\mu s}{2}(\phi_\xi^2)_\xi+\mu  \dot{\mb{X}}(t)(\wt v^S)^\mb{X}_{\xi\xi}   \phi_\xi, 
$$
Then it holds that
 \begin{equation}\label{4.54}
     \begin{aligned}  
    &\left (\frac{\mu}{2}\phi_\xi^2-\psi \phi_\xi\right)_t+{s(\psi\phi)_\xi+\frac{1}{2}(\psi^2)_{\xi\xi}-\psi_\xi^2-\dot{\mb{X}}(t) (\wt v^S)^\mb{X}_{\xi\xi }\psi}+\sigma^\prime(\wt v+\phi)\phi_\xi^2\\
    & +[\sigma^\prime(\wt v+\phi)-\sigma^\prime(\wt v)]\wt v_\xi \phi_\xi-\dot{\mb{X}}(t)(\wt u^S)^\mb{X}_\xi\phi_\xi-\frac{\mu s}{2}(\phi_\xi^2)_\xi+\mu\dot{\mb{X}}(t)(\wt v^S)^\mb{X}_{\xi\xi}\phi_\xi=F\phi_\xi.
        \end{aligned}  
   \end{equation} 
By $\epsilon$-Cauchy inequality, we have
$$
\left|\dot{\mb{X}}(t) (\wt v^S)^\mb{X}_{\xi\xi}\psi\right|\leq \epsilon\left|\dot{\mb{X}}(t)\right|^2+C_\epsilon \left|(\wt v^S)^\mb{X}_{\xi\xi}\right|^2\psi^2 \leq \epsilon\left|\dot{\mb{X}}(t)\right|^2+C_\epsilon {\delta_S^2} \left|(\wt v^S)^\mb{X}_{\xi}\right|\psi^2 ,
$$
$$
\left|[\sigma^\prime(\wt v+\phi)-\sigma^\prime(\wt v)]\wt v_\xi \phi_\xi\right|\leq \epsilon\phi_\xi^2+C_\epsilon{\delta_S^2}  \left|(\wt v^S)^\mb{X}_{\xi}\right|^2 \phi^2+C_\epsilon{\delta_S^2} \left|\wt v^R_{\xi}\right|^2 \phi^2,
$$
$$
\left| \dot{\mb{X}}(t)  (\wt u^S)^\mb{X}_{\xi} \phi_\xi \right|\leq \epsilon\phi^2_\xi+C_\epsilon|\dot{\mb{X}}(t)|^2,~\left| \mu\dot{\mb{X}}(t)(\wt v^S)^\mb{X}_{\xi\xi}\phi_\xi\right |\leq \epsilon\phi^2_\xi+C_\epsilon\left|\dot{\mb{X}}(t)\right|^2,
$$
$$
|F\phi_\xi|\leq \epsilon\phi_\xi^2+C_\epsilon|F|^2,~|F|^2\leq C(\wt u^R_{\xi\xi})^2+C{\delta_S^2}|(\wt v^S)^{\mb{X}}_{\xi} (\wt v^R-v_m) |^2+C{\delta_S^2}|(\wt v^R)_\xi ((\wt v^S)^{\mb X}-v_m)|^2.
$$
Observe that
\begin{equation*}
     \begin{aligned}
 \int_{\mathbb{R}}\left| \wt u^R_{\xi\xi}\right|^2d\xi\leq  \left\| \wt u^R_{\xi\xi}\right\|_{L^\infty(\mathbb R)} (1+t)^{-1}\leq C\delta_R^\frac{2}{3}(1+t)^{-\frac{4}{3}},
\end{aligned}
 \end{equation*}
 which implies 
 \begin{equation*}
  \begin{aligned}
\int_0^t \int_{\mathbb{R}}\left| \wt u^R_{\xi\xi}\right|^2d\xi d\tau\leq C\delta_R^\frac{2}{3}.
\end{aligned}
 \end{equation*}
Furthermore, by Lemma \ref{L4.8}(or Lemma \ref{L4.8.}), we have
 \begin{equation*}
  \begin{aligned}
\int_0^t \int_{\mathbb{R}}\left| ((\wt v^S)^{\mb{X}}-v_m) (\wt v^R)_\xi\right|^2d\xi d\tau\leq \begin{cases}
        C\delta_S\delta_R^{2}, & \text{if }\ \ -\dfrac{v_+}{2} < v_m < 0, \\[4mm]
        C\delta_R^{2}, & \text{if }\ \  v_m = -\dfrac{v_+}{2}.
         \end{cases}
\end{aligned}
 \end{equation*}
 and 
  \begin{equation*}
  \begin{aligned}
\int_0^t \int_{\mathbb{R}}\left| ((\wt v^R)-v_m)(\wt v^S)^{\mb{X}}_\xi\right|^2d\xi d\tau\leq \begin{cases}
        C\delta_S^2\delta_R^{2}, & \text{if }\ \ -\dfrac{v_+}{2} < v_m < 0, \\[4mm]
        C\delta_S\delta_R^{2}, &  \text{if }\ \  v_m = -\dfrac{v_+}{2}.
         \end{cases}
\end{aligned}
 \end{equation*}

Thus, integrating \eqref{4.54} with respect to $t$ and $\xi$, and using Proposition  \ref{prop:key_energy}, we complete the proof.
\end{proof}   

\begin{proposition} \label{P4.3} Under the same hypotheses of Proposition \ref{prop:apriori}, there exists a positive constant $C$ such that  for all $t\le T$, it holds that

\begin{align}
\begin{aligned}\label{4.55}
\int_\mathbb{R} \psi_\xi^2 d\xi+ \int_0^t \|\psi_{\xi\xi}(\tau)\|^2d\tau \leq C\left(\|\psi_{0\xi}\|^2+\delta_R^{\frac{1}{3}}\right).
\end{aligned}
\end{align}
 \end{proposition}
\begin{proof}   
 We multiply  $\eqref{Per}_2$  by $-\psi_{\xi\xi}$ and obtain
\begin{align}
\begin{aligned}\label{4.56}
\left(\frac{\psi_\xi^2}{2}\right)_t&+[\sigma^\prime(\phi+\wt v)-\sigma^\prime(\wt v) ]\wt v_\xi \psi_{\xi\xi}+\sigma^\prime(\phi+\wt v)\phi_\xi\psi_{\xi\xi}+\dot{\mb{X}}(t)(\wt u^S)^\mb{X}_{\xi\xi}\psi_{\xi}+\mu\psi_{\xi\xi}^2\\
&=F\psi_{\xi\xi}+\left(\psi_\xi\psi_t-\frac{s}{2}\psi_\xi^2\right)_\xi+\dot{\mb{X}}(t)\left[(\wt u^S)_\xi^\mb{X}\psi_\xi\right]_\xi.
\end{aligned}
\end{align}
By $\epsilon$-Cauchy inequality, we have
$$
\left|\left[\sigma^\prime(\phi+\wt v)-\sigma^\prime(\wt v)\right ]\wt v_\xi \psi_{\xi\xi}\right|\leq \epsilon\psi_{\xi\xi}^2+C_\epsilon{\delta_S^2}\left[\left|\wt v^R_\xi\right|^2+\left|(\wt v^S)^\mb{X}_\xi\right|^2\right]\phi^2,~|F\psi_{\xi\xi}|\leq \epsilon\psi_{\xi\xi}^2+C_\epsilon F^2,
$$
and
$$
\left|\sigma^\prime(\phi+\wt v)\phi_\xi\psi_{\xi\xi}\right|\leq \epsilon\psi^2_{\xi\xi}+C_\epsilon\phi^2_\xi,~\left| \dot{\mb{X}}(t)(\wt u^S)^\mb{X}_{\xi\xi}\psi_{\xi}\right| \leq \epsilon|\dot{\mb{X}}(t)|^2+C_\epsilon\psi^2_\xi.
$$
Hence,  integrating \eqref{4.56} with respect to $t$ and $\xi$, by Lemma \ref{L4.8}(or Lemma \ref{L4.8.}) and  Propostions \ref{prop:key_energy}-\ref{P4.2}, we obtain Propostion \ref{P4.3}.
\end{proof}

Therefore, combining Propositions \ref{prop:key_energy}-\ref{P4.3}, we obtain the uniform-in-time a priori estimates in Proposition \ref{prop:apriori}.

 \section{Appendix}
{ \bf  A.1 The proof of \eqref{4.33}.} When $\wt v^S\in (v_*,0]$, it holds that $$a(\wt v^S)=\frac{v_+-v_*}{s}\left[v_++v_m+\wt v^S-\frac{3}{v_m^2}(\wt v^S)^3+\frac{3}{v_m^3}(\wt v^S)^4\right],$$
$$a^\prime(\wt v^S)=\frac{v_+-v_*}{s}\left[1-\frac{9}{v_m^2}(\wt v^S)^2+\frac{12}{v_m^3}(\wt v^S)^3\right],$$
and
$$ a^{\prime\prime}(\wt v^S)=\frac{v_+-v_*}{s}\left[-\frac{18}{v_m^2}\wt v^S+\frac{36}{v_m^3}(\wt v^S)^2\right].$$ 
Hence, we have 
\begin{equation*}
\begin{aligned}
H_1 =& -\frac{v_+-v_*}{2s^2}
\left[
\left(1-\frac{9}{v_m^2}(\wt v^S)^2+\frac{12}{v_m^3}(\wt v^S)^3\right)^2
\left(1-6(\wt v^S)^2+9(\wt v^S)^4\right)
\right]\\
&\cdot
\Bigg[
\left(1-\frac{9}{v_m^2}(\wt v^S)^2+\frac{12}{v_m^3}(\wt v^S)^3\right)
\left(1-3(\wt v^S)^2\right)\\
&\qquad\qquad\qquad-6\wt v^S\left(v_++v_m+\wt v^S-\frac{3}{v_m^2}(\wt v^S)^3+\frac{3}{v_m^3}(\wt v^S)^4\right)
\Bigg]^{-1}\\
:=&-\frac{v_+-v_*}{2s^2}\frac{K_1^2\left[1-6(\wt v^S)^2+9(\wt v^S)^4\right]}{K_1+K_2}
\end{aligned}
\end{equation*}
\begin{equation*}
\begin{aligned}
%H_1=&-\frac{v_+-v_*}{2s^2}\frac{\left[1-\frac{9}{v_m^2}(\wt v^S)^2+\frac{12}{v_m^3}(\wt v^S)^3\right]^2\left[1-6(\wt v^S)^2+9(\wt v^S)^4\right]}{\left[1-\frac{9}{v_m^2}(\wt v^S)^2+\frac{12}{v_m^3}(\wt v^S)^3\right]\left[1-3(\wt v^S)^2\right]-6\wt v^S\left[v_++v_m+\wt v^S-\frac{3}{v_m^2}(\wt v^S)^3+\frac{3}{v_m^3}(\wt v^S)^4\right]}\\
%=&-\frac{v_+-v_*}{2s^2}\frac{K_1^2\left[1-6(\wt v^S)^2+9(\wt v^S)^4\right]}{K_1+K_2},\\
H_2=&\frac{v_+-v_*}{s^2}\Bigg[(v_++v_m)\left(3v_++\frac{5v_m}{2}\right)+\left(\frac{9v_+^2}{v_m}+12v_++\frac{5v_m}{2}\right)\wt v^S\\
&\qquad\qquad-\left(\frac{45v_+^2}{v_m^2}+\frac{99v_+}{2v_m}+\frac{27}{2}\right)(\wt v^S)^2+ \left(\frac{48v_+^2}{v_m^3}+\frac{48v_+}{v_m^2}+\frac{33}{v_m} \right)(\wt v^S)^3\\
&\qquad\qquad+\left(\frac{15}{2v_m^2}-\frac{24 v_+}{v_m^3}\right)(\wt v^S)^4+\left(\frac{162v_+}{v_m^4}-\frac{27}{2v_m^3}\right)(\wt v^S)^5\\
&\qquad\qquad-\left(\frac{405}{2v_m^4}+\frac{279v_+}{v_m^5}\right)(\wt v^S)^6
+\left(\frac{324}{v_m^5}+\frac{144v_+}{v_m^6}\right)(\wt v^S)^7-\frac{162}{v_m^6}(\wt v^S)^8\Bigg]\\
:=&\frac{v_+-v_*}{s^2}K_3,\\
H_3=&\frac{v_+-v_*}{2s^2}\Bigg[1-2v_m^2-2v_+^2-2v_mv_+-\frac{9}{v_m^2}(\wt v^S)^2+\frac{12}{v_m^3}(\wt v^S)^3\\
&\qquad\qquad+\left(21+\frac{18 v_+}{v_m}+\frac{18v_+^2}{v_m^2}\right)(\wt v^S)^2-\left(\frac{24}{v_m}+\frac{24v_+}{v_m^2}+\frac{24v_+^2}{v_m^3}\right)(\wt v^S)^3\\
&\qquad\qquad-\frac{27}{v_m^2}(\wt v^S)^4+\frac{36}{v_m^3}(\wt v^S)^5 \Bigg],\\
H_4=&-\frac{v_+-v_*}{2s^2}\Bigg[\left(18v_++\frac{18v_+^2}{v_m}\right)\wt v^S -\left(18+\frac{54v_+}{v_m}+\frac{54v_+^2}{v_m^2}\right)(\wt v^S)^2\\
&\qquad\qquad\qquad+\left(\frac{36}{v_m}+\frac{36v_+}{v_m^2}+\frac{36v_+^2}{v_m^3}\right)(\wt v^S)^3+\frac{18}{v_m^2}(\wt v^S)^4-\frac{36}{v_m^3}(\wt v^S)^5 \Bigg].
\end{aligned}
\end{equation*}
From the expressions of $H_3$ and $H_4$, a direct calculation gives
\begin{equation*}
\begin{aligned}
H_3+H_4=&\frac{v_+-v_*}{2s^2}\Bigg[1-\frac{9}{v_m^2}(\wt v^S)^2+\frac{12}{v_m^3}(\wt v^S)^3-2v_m^2-2v_+^2-2v_mv_+\\
&\qquad\qquad-\left(18v_++\frac{18v_+^2}{v_m}\right)\wt v^S+\left(39+\frac{72v_+}{v_m}+\frac{72v_+^2}{v_m^2}\right)(\wt v^S)^2\\
&\qquad\qquad-\left(\frac{60}{v_m}+\frac{60v_+}{v_m^2}+\frac{60v_+^2}{v_m^3}\right)(\wt v^S)^3-\frac{45}{v_m^2}(\wt v^S)^4+\frac{72}{v_m^3}(\wt v^S)^5\Bigg]\\
:=&\frac{v_+-v_*}{2s^2}(K_1+K_4).
\end{aligned}
\end{equation*}
Here $K_1, K_2, K_3, K_4$ are defined in \eqref{K}. Thus, we finally obtain \eqref{4.33}.

\vskip0.3cm
\noindent{\bf Conflict of Interest:}  The authors declare that they have no conflicts of interest with respect to this work.

\end{document}